\documentclass[11pt]{amsart}
\usepackage{amsmath}
\usepackage{amssymb}
\usepackage{amsthm}
\usepackage{latexsym}
\usepackage{graphicx}
\usepackage{hyperref}
\usepackage{enumerate}
\usepackage[all]{xy}
\usepackage{chngcntr}

\newtheorem{thm}{Theorem}[section]
\newtheorem{cor}[thm]{Corollary}
\newtheorem{lem}[thm]{Lemma}
\newtheorem{prop}[thm]{Proposition}

\newtheorem{thmintro}{Theorem}

\theoremstyle{definition}
\newtheorem{defn}[thm]{Definition}

\newtheorem{cond}[thm]{Condition}
\newtheorem{ex}[thm]{Example}

\newcommand{\Z}{\mathbb Z}
\newcommand{\Q}{\mathbb Q}
\newcommand{\R}{\mathbb R}
\newcommand{\C}{\mathbb C}
\newcommand{\mf}{\mathfrak}
\newcommand{\mc}{\mathcal}

\def\Irr{{\rm Irr}}
\newcommand{\mr}{\mathrm}
\newcommand{\ind}{\mathrm{ind}}
\newcommand{\Ind}{\mathrm{Ind}}
\newcommand{\enuma}[1]{\begin{enumerate}[\textup{(}a\textup{)}] {#1} \end{enumerate}}

\newcommand{\Rep}{\mathrm{Rep}}
\newcommand{\Res}{\mathrm{Res}}
\newcommand{\af}{\mathrm{aff}}

\newcommand{\Mod}{\mathrm{Mod}}
\newcommand{\Hom}{\mathrm{Hom}}
\newcommand{\End}{\mathrm{End}}

\newcommand{\he}{\dagger}

\begin{document}

\title[\hspace{-5mm} Hermitian duals and generic representations for affine Hecke algebras]{Hermitian 
duals and generic representations\\ for affine Hecke algebras}
\date{\today}
\author[E.M. Opdam]{Eric Opdam}
\address{Korteweg--de Vries instituut voor wiskunde, Universiteit van Amsterdam, 
Science Park 107, 1098 XH, Amsterdam, the Netherlands}
\email{e.m.opdam@uva.nl}
\author[M.S. Solleveld]{Maarten Solleveld}
\address{IMAPP, Radboud Universiteit Nijmegen, Heyendaalseweg 135,
6525AJ Nijmegen, the Netherlands}
\email{m.solleveld@science.ru.nl}
\maketitle

\begin{abstract}
We further develop the abstract representation theory of affine Hecke algebras with arbitrary positive
parameters. We establish analogues of several results that are known for reductive $p$-adic groups.
These include: the relation between parabolic induction/restriction and Hermitian duals, Bernstein's
second adjointness and generalizations of the Langlands classification. We check that, in the known
cases of equivalences between module categories of affine Hecke algebras and Bernstein blocks for 
reductive $p$-adic groups, such equivalences preserve Hermitian duality.

We also initiate the study of generic representation of affine Hecke algebras. Based on an analysis of
the Hecke algebras associated to generic Bernstein blocks for quasi-split reductive $p$-adic groups, 
we propose a fitting definition of genericity for modules over
affine Hecke algebras. With that notion we prove special cases of the generalized injectivity 
conjecture, about generic subquotients of standard modules for affine Hecke algebras.
\end{abstract}
\vspace{5mm}

\tableofcontents

\section*{Introduction}

Affine Hecke algebras typically arise in two different ways:
\begin{itemize}
\item from a presentation with generators and relations,
\item from a Bernstein block of smooth representations of a reductive $p$-adic group.
\end{itemize}
The former is more general because the $q$-parameters for roots of different lengths
can be chosen independently, whereas for $p$-adic groups there is always some algebraic
relation between the various $q$-parameters. Affine Hecke algebras are simpler than $p$-adic 
groups, and that has made it possible to derive many results about representations of 
reductive $p$-adic groups by studying Hecke algebras.  

The motivation for this paper comes from two directions. Firstly, there are well-known
results in the representation theory of $p$-adic groups for which no Hecke algebra
version has been worked out. Here we are thinking mainly of more algebraic aspects,
roughly speaking the parts of Renard's monograph \cite{Ren} that also make sense for
Hecke algebras. We want to prove analogues of those results using only Hecke algebras,
that should be easier than for $p$-adic groups.

Secondly, we are interested in the generalized injectivity conjecture \cite{CaSh}, about
generic subquotients of standard representations of quasi-split reductive $p$-adic groups.
While this has been verified in many cases \cite{Dij}, it remains open in general. We 
hope that an approach via Hecke algebras can provide new insights in that conjecture.\\

\textbf{Hermitian duals}

In the representation theory of groups, contragredients of representations play a substantial 
role. Therefore it would be desirable to develop a notion of contragredient representations 
for Hecke algebras. While that can be done, there is a problem. Namely, given a smooth 
representation $\pi$ in Bernstein block for a reductive $p$-adic group $G$, the
contragredient $\pi^\vee$ need not lie in the same Bernstein block. So, if this Bernstein
block would be equivalent to the module category of an affine Hecke algebra $\mc H$, a
notion of contragredience for $\mc H$ would never agree with contragredience for smooth
$G$-representations.

Instead, we prefer to use Hermitian duals of complex $G$-representations, that is, the 
contragredient of the complex conjugate of a representation. The main advantage is
that Hermitian duality for reductive $p$-adic groups always sends representations in one
Bernstein block to the same Bernstein block \cite[Lemma 2.2]{SolQS}.

For an affine Hecke algebra $\mc H$, with underlying (extended) affine Weyl group
$W \ltimes X$ and positive $q$-parameters, there is a natural conjugate-linear involution. 
In the Iwahori--Matsumoto presentation, it is given simply by $T_w^* = T_{w^{-1}}$ for
all $w \in W \ltimes X$. The Hermitian dual of an $\mc H$-representation $(\pi,V)$ is
defined as the vector space $V^\he$ of conjugate-linear functions $V \to \C$, with the action
\begin{equation}\label{eq:1}
\pi^\he (h) \lambda (v) = \lambda (\pi (h^*) v) \qquad v \in V, \lambda \in V^\he .
\end{equation}
Before we formulate our first result, let us point out that the affine Hecke algebras that 
arise from reductive $p$-adic groups are often of a slightly more general kind. Let $\Gamma$ 
be a finite group acting on $\mc H$, preserving all the structure used to define $\mc H$.
(See Section \ref{sec:extend} for the precise setup.) Then we can form the crossed product
$\mc H \rtimes \Gamma$, which is sometimes called an extended affine Hecke algebra.
We may also involve a 2-cocycle $\natural : \Gamma^2 \to \C^\times$, which gives rise to
a twisted affine Hecke algebra $\mc H \rtimes \C [\Gamma,\natural]$. Of course $\Gamma$ may 
be the trivial group, in which case $\mc H \rtimes \Gamma$ and 
$\mc H \rtimes \C [\Gamma,\natural]$ reduce to $\mc H$. 
We prove all our results first for $\mc H$, and we generalize them to
$\mc H \rtimes \Gamma$ or $\mc H \rtimes \C [\Gamma,\natural]$ in Section \ref{sec:extend}.

\begin{thmintro}\label{thm:A}
\textup{(see Theorem \ref{thm:5.3} and Section \ref{sec:extend})}\\
Let $G$ be a reductive group over a non-archimedean local field and let $\Rep (G)^{\mf s}$
be a Bernstein block in the category of smooth complex $G$-representations. Suppose that
$\Rep (G)^{\mf s}$ is equivalent to the module category of a twisted affine Hecke algebra 
$\mc H \rtimes \C[\Gamma,\natural]$, via a Morita equivalence as in \cite{Hei} or 
\cite[\S 10]{SolEnd}. Then the equivalence $\Rep (G)^{\mf s} \cong 
\Mod (\mc H \rtimes \C[\Gamma,\natural])$ preserves Hermitian duals.
\end{thmintro}

The Hermitian duals from \eqref{eq:1} play a crucial role in our new results about 
representations of affine Hecke algebras, they are involved in the proofs of all the main
results mentioned below.\\

\textbf{Representation theory of affine Hecke algebras}

For good notions of parabolic subalgebras, parabolic induction and parabolic restriction
for $\mc H \rtimes \C[\Gamma,\natural]$ with $\Gamma$ nontrivial, we need some conditions on 
subgroups of $\Gamma$. These are listed in Condition \ref{cond:8.1}, which we assume the 
remainder of the introduction. In our setup, the root system $R$ underlying $\mc H$ comes 
with a basis $\Delta$, and parabolic subalgebras $\mc H^P \rtimes \C[\Gamma_P,\natural]$ 
are parametrized bijectively by subsets $P \subset \Delta$. 

Let $w_\Delta$ be the longest element of $W = W (R)$ and define $P^{op} = w_\Delta (-P)$. 
This is a subset of $\Delta$, which plays the role that an opposite parabolic subgroup plays 
for reductive groups. There is a *-algebra isomorphism
\[
\begin{array}{ccccl}
\psi_{\Delta P} : & \mc H^P \rtimes \C[\Gamma_P,\natural] & \to & \mc H^{P^{op}} \rtimes 
\C[\Gamma_{P^{op}},\natural]  \\
& T_w & \mapsto & T_{w_\Delta w_P w w_P w_\Delta} & \qquad w \in W_P \ltimes X ,
\end{array}
\]
where $w_P$ is the longest element of $W_P = W(R_P)$.

\begin{thmintro}\label{thm:B} 
\textup{(see Propositions \ref{prop:2.3}, \ref{prop:2.5} and Section \ref{sec:extend})} 
\enuma{
\item Let $\rho$ be a representation of $\mc H^P \rtimes \C[\Gamma_P,\natural]$. Then 
$\ind_{\mc H^P \rtimes \C[\Gamma_P,\natural]}^{\mc H \rtimes \C[\Gamma,\natural]} (\rho^\he)$ 
is canonically isomorphic to 
$\ind_{\mc H^P \rtimes \C[\Gamma_P,\natural]}^{\mc H \rtimes \C[\Gamma,\natural]} (\rho)^\he$.
\item Let $\pi$ be an $\mc H \rtimes \C[\Gamma,\natural]$-representation. 
There is a canonical isomorphism
\[
\Res^{\mc H \rtimes \C[\Gamma,\natural]}_{\mc H^P \rtimes \C[\Gamma_P,\natural]} (\pi^\he) \cong
\Res^{\mc H \rtimes \C[\Gamma,\natural]}_{\mc H^{P^{op}} \rtimes \C[\Gamma_{P^{op}},\natural]} 
(\pi)^\he \circ \psi_{\Delta P}.
\]
}
\end{thmintro}

For Hecke algebras it is easily seen that the parabolic restriction functor
\[
\Res^{\mc H \rtimes \C[\Gamma,\natural]}_{\mc H^P \rtimes \C[\Gamma_P,\natural]} : 
\Mod (\mc H \rtimes \C[\Gamma,\natural]) \to \Mod (\mc H^P \rtimes \C[\Gamma_P,\natural] )
\]
is the right adjoint of the parabolic induction functor 
\[
\ind_{\mc H^P \rtimes \C[\Gamma_P,\natural]}^{\mc H \rtimes \C[\Gamma,\natural]} : 
\Mod (\mc H^P \rtimes \C[\Gamma_P,\natural] ) \to \Mod (\mc H \rtimes \C[\Gamma,\natural] ) .
\] 
Like for $p$-adic groups, it required more effort to find the second adjointness relation for 
parabolic induction. For graded Hecke algebras that had been achieved in \cite{BaCi},
the arguments for affine Hecke algebras are somewhat more complicated.

\begin{thmintro}\label{thm:C}
\textup{(see Theorem \ref{thm:3.1} and Section \ref{sec:extend})}
\enuma{
\item The left adjoint of $\ind_{\mc H^P \rtimes \C[\Gamma_P,\natural]}^{\mc H \rtimes 
\C[\Gamma,\natural]}$ is
\[
\psi_{\Delta P}^* \circ \Res^{\mc H \rtimes \C[\Gamma,\natural]}_{\mc H^{P^{op}} \rtimes 
\C[\Gamma_{P^{op}},\natural]} : \pi \mapsto \Res^{\mc H \rtimes \C[\Gamma,\natural]}_{
\mc H^{P^{op}} \rtimes \C[\Gamma_{P^{op}},\natural]} (\pi) \circ \psi_{\Delta P} .
\]
\item The right adjoint of $\Res^{\mc H \rtimes \C[\Gamma,\natural]}_{\mc H^P \rtimes 
\C[\Gamma_P,\natural]}$ is
\[
\ind_{\mc H^{P^{op}} \rtimes \C[\Gamma_{P^{op}},\natural]}^{\mc H \rtimes \C[\Gamma,\natural]} 
\circ \psi_{\Delta P *} : \rho \mapsto \ind_{\mc H^{P^{op}} \rtimes \C[\Gamma_{P^{op}},\natural]}^{
\mc H \rtimes \C[\Gamma,\natural]} (\rho \circ \psi_{\Delta P}^{-1}) .
\]
}
\end{thmintro}

This is useful in several ways, for instance to find a filtration of the functor parabolic
induction followed by parabolic restriction (Proposition \ref{prop:8.6}).

Recall that the Langlands classification for a reductive $p$-adic groups says:
\begin{enumerate}[(i)]
\item Every standard $G$-representation has a unique irreducible quotient.
\item This yields a bijection between the set of standard $G$-representations (up to isomorphism) 
and the set of irreducible smooth $G$-representations (also up to isomorphism).
\end{enumerate}
By definition a standard $G$-representation is of the form $I_P^G (\tau \otimes \chi)$,
where $P = MU$ is a parabolic subgroup of $G$, $\tau$ is an irreducible tempered $M$-representation
and $\chi$ is an unramified character of $M$ in positive position with respect to $P$. In \cite{Ren}
the positivity of $\chi$ was relaxed to a more algebraic regularity condition, such that (i)
remains valid. Via contragredients or Hermitian duals, one can easily derive a version of the
Langlands classification with subrepresentations instead of quotients.

For affine Hecke algebras the normal version of the Langlands classification is known from 
\cite{Eve,SolAHA}, but variations like those mentioned above had not been worked out yet. 
We say that an $\mc H^P \rtimes \C[\Gamma_P,\natural]$-representation $\pi$ is $W\Gamma,P$-regular 
if: for all weights $t$ of $\pi$ and all $w \in W_P \Gamma_P D_+^{P,P}$, $w t$ is not a weight 
of $\pi$, where
\[
D_+^{P,P} = \{ d \in W \Gamma : d(P) \subset R^+, d^{-1}(P) \subset R^+, d \notin \Gamma_P \} .
\]
This notion relates to standard $\mc H$-modules in the following ways (Proposition \ref{prop:4.6}).
\begin{itemize}
\item Suppose that an irreducible tempered $\mc H^P$-representation $\tau$ is twisted by a weight
$t$ in positive position for $\mc H^P$. Then $\tau \otimes t$ is $W,P$-regular.
\item Suppose that an irreducible tempered $\mc H^P$-representation $\tau$ is twisted by a weight
$t$ in negative position for $\mc H^P$. Then $(\tau \otimes t) \circ \psi_{\Delta P}^{-1}$ is a 
$W,P^{op}$-regular $\mc H^{P^{op}}$-representation.
\end{itemize}

\begin{thmintro}\label{thm:D}
\textup{(see Theorem \ref{thm:4.7} and Section \ref{sec:extend})}\\
Let $P \subset \Delta$ and let $\pi$ be an irreducible representation of 
$\mc H^P \rtimes \C[\Gamma_P,\natural]$.
\enuma{
\item Suppose that $\pi$ is $W\Gamma,P$-regular. Then $\ind_{\mc H^P \rtimes \C[\Gamma_P,\natural]}^{
\mc H \rtimes \C[\Gamma,\natural]} (\pi)$ has a unique irreducible quotient, namely 
$\ind_{\mc H^P \rtimes \C[\Gamma_P,\natural]}^{\mc H \rtimes \C[\Gamma,\natural]}(\pi)$ modulo the 
kernel of the intertwining operator associated to $(w_\Delta w_P,P,\pi)$.
\item Suppose that $\pi \circ \psi_{\Delta P}^{-1}$ is $W\Gamma,P^{op}$-regular. Then 
$\ind_{\mc H^P \rtimes \C[\Gamma_P,\natural]}^{\mc H \rtimes \C[\Gamma,\natural]}(\pi)$ 
has a unique irreducible subrepresentation, namely the image of the intertwining operator 
associated to $(w_P w_\Delta,P^{op},\pi \circ \psi_{\Delta P}^{-1})$.
}
\end{thmintro}

\textbf{Genericity of representations}

For quasi-split reductive $p$-adic groups, the notion of genericity is well-known. For irreducible
representations it is equivalent to the existence of a Whittaker model. It is especially useful
for the normalization of intertwining operators, for $\gamma$-factors via the Langlands--Shahidi
method and to select one member from an L-packet in the local Langlands correspondence.

For arbitrary connected reductive $p$-adic groups, a similar definition of genericity is available
\cite{BuHe}. In that generality it is convenient to consider representations which are simply
generic, meaning that the multiplicity one property of Whittaker functionals holds by assumption.

For (extended) affine Hecke algebras no independent definition of genericity was known, so we 
provide one. The elements $T_w$ with $w \in W \Gamma$ span a finite dimensional subalgebra 
$\mc H (W,q^\lambda) \rtimes \Gamma$ of $\mc H \rtimes \Gamma$. Let $\det_X$ be the determinant of 
the action of $W \Gamma$ on the lattice $X$. The Steinberg representation of $\mc H (W,q^\lambda) 
\rtimes \Gamma$ has dimension one and is defined by St$(T_w) = \det_X (w)$. We say that a 
representation $\pi$ of $\mc H \rtimes \Gamma$ is generic if its restriction to 
$\mc H (W,q^\lambda) \rtimes \Gamma$ contains St. This definition is justified by the following result.

\begin{thmintro}\label{thm:E}
\textup{(see Proposition \ref{prop:6.3} and Theorem \ref{thm:8.8})}\\
Let $G$ be a connected reductive group over a non-archimedean local field. Let $\Rep (G)^{\mf s}$
be a Bernstein block of smooth complex $G$-representations, such that the underlying 
supercuspidal representations are simply generic. 
\enuma{ 
\item $\Rep (G)^{\mf s}$ is equivalent to the module category of an extended affine 
Hecke algebra $\mc H \rtimes \Gamma$ with parameters in $\R_{\geq 1}$. 
\item With the normalizations from \cite[\S 2]{SolQS}, the equivalence 
$\Rep (G)^{\mf s} \cong \Mod (\mc H \rtimes \Gamma)$ preserves genericity.
}
\end{thmintro}

For affine Hecke algebras with $q$-parameters in $\R_{\geq 1}$, one can hope for a version of the
generalized injectivity conjecture. Using our previous findings in the representation theory of
Hecke algebras, we take some steps in that direction. 

By definition the maximal commutative subalgebra $\mc A \cong \C [X]$ of $\mc H$ is the unique
minimal parabolic subalgebra of $\mc H \rtimes \Gamma$. The basis $\Delta$ of $R$ determines a
positive cone in $\Hom (X,\R_{>0})$.

\begin{thmintro}\label{thm:F}
\textup{(see Propositions \ref{prop:7.4} and \ref{prop:8.3})}\\
Let $\mc H \rtimes \Gamma$ be an extended affine Hecke algebra with $q$-parameters in $\R_{\geq 1}$.
\enuma{
\item For $t \in \Hom (X,\C^\times)$, the parabolically induced representation 
$\ind_{\mc A}^{\mc H \rtimes \Gamma}(t)$ has a unique generic constituent, say $\pi_t$.
\item When $|t|$ lies in the closure of the positive cone in $\Hom (X,\R_{>0})$, $\pi_t$ is a
subrepresentation of $\ind_{\mc A}^{\mc H \rtimes \Gamma}(t)$.
\item When $|t^{-1}|$ lies in the closure of the positive cone in $\Hom (X,\R_{>0})$, $\pi_t$ is a
quotient of $\ind_{\mc A}^{\mc H \rtimes \Gamma}(t)$.
}
\end{thmintro}

In spite of this result, the generalized injectivity conjecture does not always hold for standard
$\mc H$-representations that are induced from parabolic subalgebras other than $\mc A$, see
Example \ref{ex:7}. The problem seems to be that arbitrary $q$-parameters (in $\R_{\geq 1}$) offer
too much freedom. We expect that the generalized injectivity conjecture does hold for affine Hecke
algebras $\mc H$ whose $q$-parameters come from reductive $p$-adic groups. To all
appearances such $q$-parameters are geometric in the sense of \cite[\S 5.3]{SolHecke}, so that
algebro-geometric techniques to study representations of such Hecke algebras are available.

\renewcommand{\theequation}{\arabic{section}.\arabic{equation}}
\counterwithin*{equation}{section}

\section{Preliminaries}
\label{sec:prelim}

We fix notations and recall a few basic notions about affine Hecke algebras.
For more background we refer to \cite{Lus-Gr,Opd-Sp,SolHecke}.

\noindent Let $R$ be a root system with basis $\Delta$ and positive roots $R^+$.
Let $\mc R = (X,R,Y,R^\vee,\Delta)$ be a based root datum. It yields a Weyl group
$W = W(R)$, with set of simple reflections $S = \{ s_\alpha : \alpha \in \Delta \}$.
For $\alpha \in R$ such that $\alpha^\vee \in R^\vee$ is maximal with respect to $\Delta^\vee$,
we define the simple affine reflection
\[
\begin{array}{cccl}
s'_\alpha : & X & \to & X \\
& x & \mapsto & s_\alpha (x) + \alpha = x + (1 - \langle x ,\alpha^\vee \rangle) \alpha .
\end{array} 
\]
Then $S_{\mr{aff}} := S \cup \{ s'_\alpha : \alpha^\vee \in R^\vee_{\mr{max}} \}$ is a set of Coxeter
generators for the affine Weyl group 
\[
W_{\mr{aff}} = \langle S_{\mr{aff}} \rangle = W \ltimes \Z R .
\]
It is a normal subgroup of the extended affine Weyl group $W(\mc R) = W \ltimes X$. The length 
function $\ell$ of $W_{\mr{aff}}$ extends naturally to $W \ltimes X$. Moreover the set of length 
zero elements $\Omega = \{ w \in W \ltimes X : \ell (w) = 0 \}$ is a group and 
\[
W \ltimes X = W_{\mr{aff}} \rtimes \Omega .
\]
We fix $q \in \R_{>1}$ and we let $\lambda, \lambda^* : R \to \R$ be functions such that
\begin{itemize}
\item if $\alpha, \beta \in R$ are in the same $W$-orbit, then $\lambda (\alpha) = \lambda (\beta)$
and $\lambda^* (\alpha) = \lambda^* (\beta)$;
\item if $\alpha^\vee \notin 2Y$, then $\lambda^* (\alpha) = \lambda (\alpha)$.
\end{itemize}
To every simple (affine) reflection we associate a $q$-parameter, by
\[
q_{s_\alpha} = q^{\lambda (\alpha)} \quad \text{and} \quad q_{s'_\alpha} = q^{\lambda^* (\alpha)} .
\]
The Iwahori--Hecke algebra $\mc H (W_{\mr{aff}}, \lambda, \lambda^*, q)$ can be presented as 
the vector space with basis $\{ N_w : w \in W_{\mr{aff}} \}$ and multiplication rules (for
$w \in W_{\mr{aff}}$ and $s \in S_{\mr{aff}}$)
\begin{equation}\label{eq:1.11}
N_w N_s = \left\{ \begin{array}{ll}
N_{ws} & \text{ if } \ell (ws) = \ell (w) + 1 \\
N_{ws} + (q_s^{1/2} - q_s^{-1/2}) N_w & \text{ if } \ell (ws) = \ell (w) - 1 
\end{array} \right. .
\end{equation}
Notice that $q_s^{1/2}$ is unambiguous, because $q_s \in \R_{>0}$. The conjugation action of
$\Omega$ on $W_{\mr{aff}}$ induces an action on $\mc H (W_{\mr{aff}},\lambda,\lambda^*,q)$.
That enables us to construct the affine Hecke algebra
\[
\mc H := \mc H (\mc R,\lambda,\lambda^*,q) = \mc H (W_{\mr{aff}},\lambda,\lambda^*,q) \rtimes \Omega ,
\]
which has a vector space basis $\{ N_w : w \in W \ltimes X\}$. This is a version of the 
Iwahori--Matsumoto presentation of $\mc H (\mc R,\lambda,\lambda^*,q)$. More common, and already
used in \cite{IwMa}, is the same presentation expressed in terms of the basis $\{ T_w : w \in W 
\ltimes X\}$, where $T_s = q_s^{1/2} N_s$ for $s \in S_\af$.

There is another well-known presentation, due to
Bernstein. To that end, we define elements $\theta_x \; (x \in X)$ by the following recipe. 
If $x = x_1 - x_2$ where $\langle x_1, \alpha^\vee \rangle \geq 0$ and $\langle x_2, \alpha^\vee 
\rangle \geq 0$ for all $\alpha \in \Delta$, then $\theta_x = N_{x_1} N_{x_2}^{-1}$.

The set $\{ \theta_x : x \in X\}$ spans a commutative subalgebra $\mc A$ of $\mc H$, canonically
isomorphic with $\C [X]$. Let $\mc H (W, q^\lambda)$ be the Iwahori--Hecke algebra of $W$, with
respect to the parameter function $q^\lambda : R \to \R_{>0}$. According to Bernstein, the
multiplication maps
\begin{equation}\label{eq:1.5}
\mc H (W, q^\lambda) \otimes \mc A \; \to \; \mc H \; \leftarrow \; \mc A \otimes \mc H (W,q^\lambda)
\end{equation}
are bijections. The cross relations for multiplication of elements of $\mc H (W,q^\lambda)$
and of $\mc A$ can be described explicitly. It follows from those relations that the centre of
$\mc H$ is $\mc A^W$, where $W$ acts on $\mc A \cong \C [X]$ via its canonical action on $X$.

For a set of simple roots $P \subset \Delta$ we have a parabolic subrootsystem $R_P \subset R$
and a parabolic subgroup $W_P = W(R_P)$. The parabolic subalgebra $\mc H^P \subset \mc H$
is generated by $\mc A$ and the $N_w$ with $w \in W_P$. One can identify $\mc H^P$ with
$\mc H (X, R_P, Y, R_P^\vee, P, \lambda, \lambda^*, q)$. In particular $\mc H^\emptyset = \mc A$
and $\mc H^\Delta = \mc H$.

We write $T = \Hom_\Z (X,\C^\times)$, this is a complex torus. It has subtori
\[
T^P = \Hom_\Z (X / X \cap \Q P, \C^\times ) ,\qquad 
T_P = \Hom_\Z (X / P^{\vee \perp}, \C^\times) ,
\]
where $P^{\vee \perp} = \{x \in X : \langle x, \alpha^\vee \rangle = 0 \; \forall \alpha \in P \}$.
Any $t \in T^P$ gives rise to an algebra automorphism
\[
\begin{array}{llll}
\psi_t : & \mc H^P & \to & \mc H^P \\
 & N_w \theta_x & \mapsto & t(x) N_w \theta_x \qquad w \in W_P, x \in X.
\end{array}
\]
For an $\mc H^P$-representation $\pi$ and $t \in T^P$ we write $\pi \otimes t = \pi \circ \psi_t$.

We define a conjugate-linear involution * on $\mc H$ by 
\[
\big( \sum\nolimits_{w \in W \ltimes X} z_w N_w \big)^* = 
\sum\nolimits_{w \in W \ltimes X} \overline{z_w} N_{w^{-1}}. 
\]
Here we need $q_s \in \R_{>0}$ for all $s \in S_{\mr{aff}}$. We can regard * as an $\R$-linear
isomorphism from $\mc H$ to its opposite algebra. This involution interacts well with the trace
\begin{equation}\label{eq:1.4}
\tau : \mc H \to \C ,\qquad \tau (N_w) = \left\{ 
\begin{array}{ll}
1 & \text{ if } w = e \\
0 & \text{ if } w \in W(\mc R) \setminus \{e\} 
\end{array} \right. .
\end{equation}
Namely, the formula
\begin{equation}\label{eq:1.1}
\langle h_1, h_2 \rangle = \tau (h_1 h_2^*) \qquad h_1, h_2 \in \mc H
\end{equation}
defines an inner product on $\mc H$, linear in the first variable. The set $\{ N_w : w \in W (\mc R)\}$ 
is an orthonormal basis of $\mc H$ with this inner product.
We note that \eqref{eq:1.1} makes the left regular representation of $\mc H$ pre-unitary
(i.e. \hspace{-4mm} a *-representation on an inner product space that need not be complete):
\begin{equation}\label{eq:1.2}
\langle h_1 h_2, h_3 \rangle = \tau (h_1 h_2 h_3^*) = \tau (h_2 h_3^* h_1) = 
\langle h_2, h_1^* h_3 \rangle \qquad h_1, h_2, h_3 \in \mc H .
\end{equation}
The parabolic subalgebra $\mc H^P$ has its own involution $*_P$, which usually differs from
$* |_{\mc H^P}$. In fact $\mc H^P$ is typically not a *-subalgebra of $\mc H$. Let $w_P$ be the 
longest element of $W_P$. Recall that $w_P$ has order two and that the set of positive roots made 
negative by $w_P$ is precisely $R_P^+$. By \cite[Proposition 1.12]{Opd1}:
\begin{equation}\label{eq:1.3}
*_P (\theta_x) = N_{w_P} \theta_{-w_P (x)} N_{w_P}^{-1} .
\end{equation}
For $P = \emptyset$ we get $w_\emptyset = 1$, so $*_\emptyset (\theta_x) = \theta_{-x}$.

For a subset $\tilde W \subset W$, let $\mc H (\tilde W)$ be the linear subspace of $\mc H 
(W,q^\lambda)$ spanned by $\{ N_w : w \in \tilde W\}$. Let 
\[
W^P = \{w \in W : w (P) \subset R^+\}
\]
be the set of shortest length representatives for $W / W_P$. By \eqref{eq:1.5} the multiplication
map $\mc H (W^P) \otimes \mc H^P \to \mc H$ is a linear bijection. In particular every $h \in \mc H$
can be written as
\begin{equation}\label{eq:1.10}
h = \sum\nolimits_{w \in W^P} N_w h^P_w \quad \text{for unique } h^P_w \in \mc H^P .
\end{equation}
The next result is analogous to \cite[Proposition 1.4]{BaMo3} for graded Hecke algebras.

\begin{lem}\label{lem:1.1}
$(h^*)^P_e = (h^P_e)^{*_P}$ for all $h \in \mc H$.
\end{lem}
\begin{proof}
By conjugate-linearity it suffices to consider $h$ of the form $N_w \theta_x$ with $w \in W$ and 
$x \in X$. From \eqref{eq:1.3} we see that
\[
h^* = N_{w_\Delta} \theta_{-w_\Delta (x)} N_{w_\Delta}^{-1} N_{w^{-1}} \quad \text{and} \quad 
(h_e^P)^{*_P} = \left\{ \begin{array}{cc}
N_{w_P} \theta_{-w_P (x)} N_{w_P}^{-1} N_{w^{-1}} & w \in W_P , \\
0 & w \notin W_P .
\end{array} \right.
\]
We recall from \cite[\S 1.8]{Hum} that
\[
\ell (w^{-1}) + \ell (w w_\Delta) = \ell (w_\Delta) = \ell (w_\Delta w) + \ell (w^{-1}) .
\]
By definition of the multiplication in $\mc H (W,q^\lambda)$:
\begin{align}\label{eq:1.6}
& N_{w_\Delta w} N_{w^{-1}} = N_{w_\Delta} = N_{w^{-1}} N_{w w_\Delta} ,\\
\label{eq:1.7} & N_{w_\Delta} \theta_{-w_\Delta (x)} N_{w_\Delta}^{-1} N_{w^{-1}} = 
N_{w_\Delta} \theta_{-w_\Delta (x)} N_{w w_\Delta}^{-1} .
\end{align}
For a simple reflection $s \in W$, $N_s^{-1} = N_s + (q_s^{-1/2} - q_s^{1/2}) N_e$. That and 
the multiplication relations in the Bernstein presentation of $\mc H$ \cite[\S 3]{Lus-Gr} show that
\begin{equation}\label{eq:1.8}
\theta_y N_s^{-1} - N_s^{-1} \theta_{s (y)} \in \mc A \qquad \text{for all } y \in X.
\end{equation}
We denote the Bruhat order on $W$ by $\leq$.
Applying \eqref{eq:1.8} recursively, \eqref{eq:1.7} can be expressed as $N_{w_\Delta} \sum_{v \in W ,
v \leq w w_\Delta} N_v^{-1} a_v$ for suitable $a_v \in \mc A$. By \eqref{eq:1.6} that equals
$\sum_{v \in W ,v \leq w w_\Delta} N_{w_\Delta v^{-1}} a_v$. Here $v^{-1} \leq w_\Delta w^{-1}$,
so $w_\Delta v^{-1} \geq w^{-1}$.

Suppose that $w \notin W_P$. Any reduced expression of $w^{-1}$ contains simple reflections
not in $W_P$, so the same goes for $w_\Delta v^{-1}$ with $v$ as above. Hence $w_\Delta v^{-1} \notin 
W_P$, and it can be written as $u w'$ with $u \in W^P \setminus \{e\}$ and $w' \in W_P$. Thus 
$N_{w_\Delta v^{-1}} a_v \in N_u \mc H^P$. That works for every $v \leq w w_\Delta$, showing that
\eqref{eq:1.7} lies in $\mc H (W^P \setminus \! \{e\}) \mc H^P$. In other words, $(h^*)^P_e = 0$.

Suppose that $w \in W_P$. We need to show that $\mc H (W^P \setminus \! \{e\}) \mc H^P$ contains
\begin{equation}\label{eq:1.9}
N_{w_\Delta} \theta_{-w_\Delta (x)} N_{w_\Delta}^{-1} N_{w^{-1}} - 
N_{w_P} \theta_{-w_P (x)} N_{w_P}^{-1} N_{w^{-1}} . 
\end{equation}
With \eqref{eq:1.6} we rewrite this element as
\[
\big( N_{w_\Delta} \theta_{-w_\Delta (x)} N_{w_P w_\Delta}^{-1} - N_{w_P} \theta_{-w_P (x)} \big) 
N_{w_P}^{-1} N_{w^{-1}} .
\]
Reasoning as above we find
\[
N_{w_\Delta} \theta_{-w_\Delta (x)} N_{w_P w_\Delta}^{-1} =
\sum\nolimits_{v \in W ,v \leq w_P w_\Delta} N_{w_\Delta v^{-1}} a_v .
\]
Here $w_\Delta v^{-1} \geq w_P$, so this only belongs to $W_P$ if $v = w_P w_\Delta$. From
\eqref{eq:1.8} one obtains 
\[
a_{w_P w_\Delta} = \theta_{w_P w_\Delta (-w_\Delta (x))} = \theta_{-w_P(x)} .
\]
Then \eqref{eq:1.9} reduces to 
\[
\sum\nolimits_{v \in W ,v < w_P w_\Delta} N_{w_\Delta v^{-1}} a_v N_{w_P}^{-1} N_{w^{-1}}.
\]
The same argument as in the case $w \notin W_P$ shows that this lies in 
$\mc H (W^P \setminus \! \{e\}) \mc H^P$.
\end{proof}

\section{Hermitian duals}
\label{sec:herm}

For any complex vector space $V$, let $V^\he$ be space of conjugate-linear functions from $V$ to $\C$.
In case $V$ has a topology, it is understood that $V^\he$ consists of the continuous conjugate-linear
functionals on $V$. If $(\pi, V_\pi)$ is an $\mc H$-representation, then $\mc H$ acts on $V_\pi^\he$ by
\begin{equation}\label{eq:2.1}
(h \cdot \lambda) (v) = \lambda (h^* v) \qquad h \in \mc H, v \in V_\pi, \lambda \in V_\pi^\he .
\end{equation}
This defines the Hermitian dual $(\pi^\he, V_\pi^\he )$ of the $\mc H$-representation $(\pi,V_\pi)$.
For any $(\rho,V_\rho) \in \Mod (\mc H)$ there is a conjugate-linear ``transposition" isomorphism
\begin{equation}\label{eq:2.10}
\begin{array}{ccc}
\Hom_{\mc H} (\pi, \rho^\he) & \cong & \Hom_{\mc H}(\rho, \pi^\he) \\ 
\phi & \mapsto & \phi^\he
\end{array}
\end{equation}
Here $\phi^\he$ sends $w \in V_\rho$ to $[v \mapsto \overline{\phi (v) w}]$ with $v \in V_\pi$.

Sometimes a representation is isomorphic to its Hermitian dual. For example, suppose that $V_\pi$ 
is a Hilbert space and that the representation $\pi$ is unitary:
\[
\langle \pi (h) v, v' \rangle = \langle v, \pi (h^*) v' \rangle  \qquad v,v' \in V_\pi, h \in \mc H .
\]
Then $(\pi^\he ,V_\pi^\he)$ can be identified with $(\pi, V_\pi)$ via the inner product.
Similarly we can consider the left regular representation of $\mc H$. Via $\tau$, we can identify 
\begin{equation}\label{eq:2.2}
\mc H^\he = \prod\nolimits_{w \in W(\mc R)} \C N_w ,
\end{equation}
a completion of $\mc H$. This $\mc H^\he$ is naturally an $\mc H$-bimodule, and \eqref{eq:1.2} remains 
valid for $h_1, h_3 \in \mc H$, $h_2 \in \mc H^\he$. Hence the Hermitian dual of $\mc H$ is $\mc H^\he$,
with $\mc H$ acting by left multiplication on both. The projectivity of $\mc H$, in combination with
\eqref{eq:2.10}, implies that $\Hom_{\mc H}(?,\mc H^\he)$ is an exact functor. In other words,
$\mc H^\he$ is an injective $\mc H$-module.

The module $\mc H^\he$ enables us to describe Hermitian duals of modules induced from 
$\mc H (W,q^\lambda)$. We recall that, as $\mc H (W,q^\lambda)$ is finite dimensional and semisimple, 
all its irreducible modules appear in the (left) regular representation. In fact each irreducible 
module is the image of suitable a minimal idempotent.

\begin{lem}\label{lem:6.2}
\enuma{
\item Let $V \in \Mod (\mc H (W,q^\lambda))$ be irreducible, and let $p_V \in \mc H(W,q^\lambda)$
be an idempotent so that $V \cong \mc H (W,q^\lambda) p_V$. The Hermitian dual of $\ind_{\mc H 
(W,q^\lambda)}^{\mc H} V$ is $\mc H^\he \otimes_{\mc H (W,q^\lambda)} V^\he$, with the pairing
\[
\langle h_1 \otimes p_V^*, h_2 \otimes p_V \rangle = \tau (h_1 p_V^* h_2^* ) \qquad
h_1 \in \mc H^\he, h_2 \in \mc H .
\]
\item The functor $\mc H^\he \otimes_{\mc H (W,q^\lambda)} : \Mod (\mc H (W,q^\lambda)) \to
\Mod (\mc H)$ is right adjoint to the restriction functor $\mr{Res}^{\mc H}_{\mc H (W,q^\lambda)}$.
}
\end{lem}
\begin{proof}
(a) The $\mc H (W,q^\lambda)$-module $\mc H (W,q^\lambda) p_V^*$ is the Hermitian dual of 
$\mc H (W,q^\lambda) p_V$, with respect to the pairing
\[
\langle h_1 p_V^*, h_2 p_V \rangle = \tau (h_1 p_V^* h_2^*) \qquad h_1, h_2 \in \mc H (W,q^\lambda) .
\]
Since $\mc H(W,q^\lambda) p_V$ is a direct summand of the left regular representation of 
$\mc H (W,q^\lambda)$, $\ind_{\mc H (W,q^\lambda)}^{\mc H} \mc H(W,q^\lambda) p_V$ is a direct summand 
of the $\mc H$-representation on
\[
\ind_{\mc H (W,q^\lambda)}^{\mc H} \mc H (W,q^\lambda) = \mc H .
\]
Similarly $\mc H^\he \otimes_{\mc H (W,q^\lambda)} \mc H (W,q^\lambda) p_V^*$ is a direct summand 
of the $\mc H$-module $\mc H^\he$. For $h_1 \in \mc H^\he$ and $h_2 \in \mc H$ we have
\[
\langle h_1 p_V^*, h_2 \rangle = \tau (h_1 p_V^* h_2^* ) = \langle h_1 p_V^*, h_2 p_V \rangle . 
\]
Hence the Hermitian dual we want is
\[
\mc H^\he \underset{\mc H (W,q^\lambda)}{\otimes} \mc H (W,q^\lambda) p_V^* \cong
\mc H^\he \underset{\mc H (W,q^\lambda)}{\otimes} V^\he . 
\]
(b) Let $Y \in \Mod (\mc H)$ and $V \in \Mod (\mc H (W,q^\lambda))$. We regard 
$\mc H^\he \otimes_{\mc H (W,q^\lambda)} V$ as a set of conjugate-linear maps from $\mc H$ to $V$.
One checks readily that the maps 
\[
\begin{array}{ccc}
\Hom_{\mc H (W,q^\lambda)}(Y,V) & \longleftrightarrow & 
\Hom_{\mc H}( Y, \mc H^\he \otimes_{\mc H (W,q^\lambda)} V) \\
f & \mapsto & \big[ y \mapsto \; [h \mapsto f (h^* y)]\big] \\
\ [y \mapsto \phi (y)(1)] & \text{\reflectbox{$\mapsto$}} & \phi
\end{array} 
\]
are natural bijections.
\end{proof}

The relation between Hermitian duals and tensoring with characters can be described easily:

\begin{lem}\label{lem:2.6}
Let $(\pi, V_\pi)$ be an $\mc H^P$-representation and let $t \in T^P$. The Hermitian dual of 
$\pi \otimes t$ is $\pi^\he \otimes \bar{t}^{-1}$.
\end{lem}
\begin{proof}
Take $v \in V_\pi, \lambda \in V_\pi^\he, w \in W_P$ and $x \in X$. With \eqref{eq:1.3} we compute
\[
\begin{aligned}
\lambda \big( (\pi \otimes t)(N_w \theta_x) v \big) & = \lambda \big( t(x) \pi (N_w \theta_x) v \big) 
= \overline{t(x)} \lambda \big( \pi (N_w \theta_x) v \big) \\
& = \overline{t(x)} \big( \pi^\he ((N_w \theta_x)^{*_P}) \lambda \big) (v) =
\overline{t(x)} \big( \pi^\he ((N_w \theta_x)^{*_P}) \lambda \big) (v) \\
& = \overline{t(x)} \big( \pi^\he (N_{w_P} \theta_{-w_P (x)} N_{w_P}^{-1} N_{w^{-1}} ) \lambda \big) (v) \\
& = \big( (\pi^\he \otimes \overline{w_P t}^{-1}) (N_{w_P} \theta_{-w_P (x)} N_{w_P}^{-1} N_{w^{-1}} ) 
\lambda \big) (v) \\
& = \big( (\pi^\he \otimes \overline{w_P t}^{-1}) ((N_w \theta_x)^{*_P}) \lambda \big) (v) .
\end{aligned}
\]
This shows that $(\pi \otimes t)^\he = \pi^\he \otimes \overline{w_P t}^{-1}$. But $w_P t = t$
because $w_P \in W(R_P)$ and $t \in \Hom_\Z (X / X \cap \Q P, \C^\times )$.
\end{proof}

We want to find the relation between parabolic induction (from $\mc H^P$ to $\mc H$) and Hermitian duals.
That will be achieved in a few steps, the first of which is making the relation between * and $*_P$ 
explicit. 

\begin{lem}\label{lem:2.1}
For $w \in W_P$ and $x \in X$:
\[
* \, *_P (N_w \theta_x) = N_w N_{w_P w_\Delta} \theta_{w_\Delta w_P (x)} N_{w_P w_\Delta}^{-1} .
\]
\end{lem}
\begin{proof}
By definition
\[
* \, *_P (N_w) = * (N_{w^{-1}}) = N_w .
\]
From \eqref{eq:1.3} and the anti-homomorphism property of * we obtain  
\begin{equation}\label{eq:2.3}
* \, *_P (\theta_x) = *(N_{w_P}^{-1}) *(\theta_{-w_P (x)}) *(N_{w_P}) =
N_{w_P}^{-1} N_{w_\Delta} \theta_{w_\Delta w_P (x)} N_{w_\Delta}^{-1} N_{w_P} .
\end{equation}
We note that here the lengths of the involved elements of $W$ add up: 
\[
\ell (w_P w_\Delta) + \ell (w_P) = |R^+ \setminus R_P^+ | + |R_P^+| = |R^+| = \ell (w_\Delta)
\]
Therefore $N_{w_P} N_{w_P w_\Delta} = N_{w_\Delta}$, and the right-hand side of \eqref{eq:2.3}
simplifies to\\ $N_{w_P w_\Delta} \theta_{w_\Delta w_P (x)} N_{w_P w_\Delta}^{-1}$, proving
the statement for $\theta_x$.

To conclude, we use that $* \, *_P$ is an algebra homomorphism.
\end{proof}

For $h \in \mc H^\times$, let $\mf c_h : \mc H \to \mc H$ denote conjugation with $h$.
We define 
\[
\psi_{\Delta P} = \mf c_{N_{w_P w_\Delta}^{-1}} \circ * \, *_P : \mc H^P \to \mc H .
\]
Since $* \, *_P$ and $\mf c_{N_{w_P w_\Delta}^{-1}}$ are injective algebra homomorphisms,
so is $\psi_{\Delta P}$. We write 
\[
P^{op} = w_\Delta (-P). 
\]
This is a set of simple roots, it may or may not be equal to $P$. We note that
$w_\Delta W_P w_\Delta = W_{P^{op}}$. In comparison with reductive groups, $P^{op}$ replaces
the notion of an opposite parabolic subgroup.

\begin{lem}\label{lem:2.2}
\enuma{
\item For $w \in W_P$ and $x \in X$:
\[
\psi_{\Delta P} (N_w \theta_x) = N_{w_\Delta w_P w w_P w_\Delta} \theta_{w_\Delta w_P (x)} .
\]
\item $\psi_{\Delta P}$ is an *-algebra isomorphism from $\mc H^P$ to $\mc H^{P^{op}}$,
with inverse $\psi_{\Delta P^{op}}$.
}
\end{lem}
\begin{proof}
(a) Consider the algebra isomorphism 
\begin{equation}\label{eq:2.11}
\begin{array}{llll}
\psi_{w_P w_\Delta} : & \mc H^{P^{op}} & \to & \mc H^P \\
& N_{w'} \theta_x & \mapsto & N_{w_P w_\Delta w' w_\Delta wP} \theta_{w_P w_\Delta (x)} 
\qquad w' \in W_{P^{op}}, x \in X .
\end{array}
\end{equation}
Then $\psi_{\Delta P} \circ \psi_{w_P w_\Delta} : \mc H^{P^{op}} \to \mc H$ is an injective
algebra homomorphism, and by Lemma \ref{lem:2.1}:
\begin{equation}\label{eq:2.4}
\begin{aligned}
\psi_{\Delta P} \circ \psi_{w_P w_\Delta} (N_{w'} \theta_x) &
= \mf c_{N_{w_P w_\Delta}^{-1}} \big( N_{w_P w_\Delta w' w_\Delta wP} N_{w_P w_\Delta} 
\theta_{x} N_{w_P w_\Delta}^{-1} \big) \\
& = N_{w_P w_\Delta}^{-1} N_{w_P w_\Delta w' w_\Delta wP} N_{w_P w_\Delta} \theta_x .
\end{aligned}
\end{equation}
Notice that $\psi_{\Delta P} \circ \psi_{w_P w_\Delta}$ is the identity on $\mc A$ and
sends $\mc H ( W_{P^{op}}, q^\lambda)$ bijectively to itself. For $\alpha \in P^{op}$,
$N_{s_\alpha}$ commutes with the same elements of $\mc A$ as $\psi_{\Delta P} \circ 
\psi_{w_P w_\Delta} (N_{s_\alpha})$. That forces
\[
\psi_{\Delta P} \circ \psi_{w_P w_\Delta} (N_{s_\alpha}) \in \C N_e + \C N_{s_\alpha}.
\]
Furthermore $\psi_{\Delta P} \circ \psi_{w_P w_\Delta} (N_{s_\alpha})$ has the same eigenvalues
$q_{s_\alpha}^{1/2}$ and $-q_{s_\alpha}^{-1/2}$ as $N_{s_\alpha}$, so it can only be $N_{s_\alpha}$
or $-N_{s_\alpha}^{-1}$. The involved constructions work for any $q \in \R_{>0}$, and depend
continuously on $q$. For $q=1$ we see directly from \eqref{eq:2.4} that 
$\psi_{\Delta P} \circ \psi_{w_P w_\Delta} (N_{s_\alpha}) = N_{s_\alpha}$. Hence 
$\psi_{\Delta P} \circ \psi_{w_P w_\Delta} (N_{s_\alpha})$ cannot be $-N_{s_\alpha}^{-1}$ for
any $q \in \R_{>0}$. 

We deduce that $\psi_{\Delta P} \circ \psi_{w_P w_\Delta} (N_{w'}) = N_{w'}$ for $w'$ any
simple reflection in $W_{P^{op}}$, and then the same follows for all $w' \in W_{P^{op}}$.
Apply that to $w' = w_\Delta w_P w w_P w_\Delta$.\\
(b) By part (a) and \eqref{eq:2.4}, $\psi_{\Delta P} \circ \psi_{w_P w_\Delta}$ is the 
identity on $\mc H^{P^{op}}$. As $\psi_{w_P w_\Delta} : \mc H^{P^{op}} \to \mc H^P$ is an
isomorphism, this shows that $\psi_{\Delta P}$ is its inverse. By construction 
$w_{P^{op}} = w_\Delta w_P w_\Delta$. From that, part (a) and \eqref{eq:2.11} we see that
$\psi_{w_P w_\Delta} = \psi_{\Delta P^{op}}$. 

Further, from \eqref{eq:2.11} and the definition of $\theta_x$ we obtain 
\[
\psi_{w_P w_\Delta} (N_{w'}) = N_{w_P w_\Delta w' w_\Delta w_P} 
\qquad \text{for all } w' \in W_{P^{op}} \rtimes X.
\]
This shows that $\psi_{w_P w_\Delta}$ is in fact a *-isomorphism, and hence so it its inverse.
\end{proof}

Now we can relate Hermitian duals and parabolic restriction.

\begin{prop}\label{prop:2.3}
Let $(\pi, V_\pi)$ be an $\mc H$-representation.
\enuma{
\item The Hermitian dual of the $\mc H^P$-representation $\Res^{\mc H}_{\mc H^P}(\pi)$
is isomorphic with\\ $\Res^{\mc H}_{\mc H^{P^{op}}} (\pi^\he) \circ \psi_{\Delta P}$.
\item $\Res^{\mc H}_{\mc H^P} (\pi^\he) \cong \Res^{\mc H}_{\mc H^{P^{op}}}(\pi)^\he
\circ \psi_{\Delta P}$.
}
\end{prop}
\begin{proof} 
By definition
\[
\Res^{\mc H}_{\mc H^P}(\pi)^\he = \pi^\he \circ * \, *_P =
\pi^\he \circ \mf c_{N_{w_P w_\Delta}} \circ \psi_{\Delta P} .
\]
Since $N_{w_P w_\Delta} \in \mc H^\times$, multiplication with $\pi^\he (N^{-1}_{w_P w_\Delta})$ 
provides an isomorphism from the right-hand side to $\pi^\he \circ \psi_{\Delta P}$.
By Lemma \ref{lem:2.2}.b, that can be regarded as\\
$\Res^{\mc H}_{\mc H^{P^{op}}} (\pi^\he) \circ \psi_{\Delta P}$.\\
(b) Start with part (a) for $P^{op}$. Composing the representations on both sides with 
$\psi_{\Delta P^{op}}^{-1} = \psi_{\Delta P}$ gives 
\[
\Res^{\mc H}_{\mc H^P} (\pi^\he) \cong \Res^{\mc H}_{\mc H^{P^{op}}}(\pi)^\he \circ \psi_{
\Delta P^{op}}^{-1} = \Res^{\mc H}_{\mc H^{P^{op}}}(\pi)^\he \circ \psi_{\Delta P} . \qedhere
\]
\end{proof}

We note that the pairing underlying Proposition \ref{prop:2.3}.a is
\[
\begin{array}{ccl}
\Res_{\mc H^{P^{op}}}^{\mc H} (V_\pi^\he) \times \Res_{\mc H^P}^{\mc H} (V_\pi) & \to & \C \\
(\lambda,v) & \mapsto & \pi^\he (N_{w_P w_\Delta}) \lambda (v) = 
\lambda (\pi (N_{w_\Delta w_P}) v) .
\end{array} 
\]
Similarly the pairing underlying Proposition \ref{prop:2.3}.b is given by
\[
\begin{array}{ccl}
\Res_{\mc H^P}^{\mc H} (V_\pi^\he) \times \Res_{\mc H^{P^{op}}}^{\mc H} (V_\pi) & \to & \C \\
(\lambda,v) & \mapsto & \pi^\he (N_{w_\Delta w_P}) \lambda (v) = 
\lambda (\pi (N_{w_P w_\Delta}) v) .
\end{array} 
\]
An important special case arise when $P = \emptyset$. Then $\psi_{\Delta \emptyset} (\theta_x) =
\theta_{w_\Delta (x)}$ and Proposition \ref{prop:2.3} provides isomorphisms
\begin{equation}\label{eq:2.5}
\Res^{\mc H}_{\mc A} (\pi)^\he \cong \Res^{\mc H}_{\mc A} (\pi^\he) \circ \psi_{\Delta \emptyset}
\quad \text{and} \quad \Res^{\mc H}_{\mc A} (\pi^\he ) \cong \Res^{\mc H}_{\mc A} (\pi )^\he 
\circ \psi_{\Delta \emptyset} .
\end{equation}
We move on to parabolic induction. Consider an $\mc H^P$-representation $(\rho,V_\rho)$ and its
induction $\ind_{\mc H^P}^{\mc H} (\rho)$. The underlying vector space is
\[
\ind_{\mc H^P}^{\mc H} (V_\rho) = \mc H \otimes_{\mc H^P} V_\rho \cong \mc H (W^P) \otimes_\C V_\pi .
\]
Here $W^P = \{w \in W : w(P) \subset R^+ \}$ denotes the set of shortest length representatives for 
$W / W_P$ and $\mc H (W^P)$ is the linear subspace of $\mc H (W,q^\lambda)$ spanned by the 
corresponding $N_w$. Following \cite[(4.24)]{Opd-Sp} we define a sesquilinear pairing
\begin{equation}\label{eq:2.6}
\begin{array}{ccl}
\mc H (W^P) \otimes_\C V_\rho^\he \; \times \; \mc H (W^P) \otimes_\C  V_\rho & \to & \C \\
\langle h' \otimes \lambda , h \otimes v \rangle & = & \tau (h' h^*) \lambda (v)
\end{array}.
\end{equation}
As preparation for a more general statement, we consider the left regular representation of 
$\mc H^P$. Clearly $\ind_{\mc H^P}^{\mc H} (\mc H^P) = \mc H$, and we already know $\mc H^\he$ from 
\eqref{eq:2.2}. Multiplication in $\mc H$ induces a linear bijection $m : \mc H (W^P) \otimes 
\mc H^P \to \mc H$. The transpose of $m$ is the linear bijection
\begin{equation}\label{eq:2.12}
m^\he : \mc H^\he \; \to \; \mc H (W^P )^\he \otimes_\C \mc H^{P\he} \cong 
\mc H (W^P) \otimes_\C \mc H^{P\he} \cong \mc H \otimes_{\mc H^P} \mc H^{P\he} .
\end{equation}
In the middle of \eqref{eq:2.12} we identified $\mc H (W^P)^\he$ with $\mc H (W^P)$ via the
inner product on $\mc H$. Notice that $\mc H^\he$ and $\mc H \otimes_{\mc H^P} \mc H^{P\he}$
independently carry $\mc H$-module structures, the latter induced from the $\mc H^P$-module 
structure of $\mc H^{P\he}$.

\begin{lem}\label{lem:2.4}
The map $m^\he : \mc H^\he \to \ind^{\mc H}_{\mc H^P} (\mc H^{P\he})$ is an isomorphism of
$\mc H$-modules. In particular $\ind^{\mc H}_{\mc H^P} (\mc H^{P\he})$ with the pairing 
\eqref{eq:2.6} is the Hermitian dual of $\ind_{\mc H^P}^{\mc H}(\mc H^P)$.
\end{lem}
\begin{proof}
Let $a,b \in W^P, h_1 \in \mc H^P$ and $\lambda \in \mc H^{P\he}$. For $h_2 \in \mc H$ 
there are elements 
\[
m^\he \big( h_2 \cdot (m^\he)^{-1}(N_a \otimes \lambda) \big) \quad \text{and} \quad 
h_2 \cdot (N_a \otimes \lambda) \quad \text{in} \quad \ind^{\mc H}_{\mc H^P} (\mc H^{P\he}) .
\]
We can compare them by pairing $N_b \otimes \mc H (W^P) \otimes \mc H^P$, as in \eqref{eq:2.6}.
In the notation from \eqref{eq:1.10}, we compute
\begin{multline}\label{eq:2.8}
\langle h_2 \cdot (N_a \otimes \lambda), N_b \otimes h_1 \rangle = 
\langle h_2 N_a \cdot (1 \otimes \lambda), N_b \otimes h_1 \rangle = \\
\sum\nolimits_{v \in W^P} \langle N_v (h_2 N_a)^P_v \cdot (1 \otimes \lambda), N_b \otimes h_1 
\rangle = \langle (h_2 N_a)^P_b \cdot \lambda, h_1 \rangle .
\end{multline}
We note that, for any $w \in W_P, v \in W^P$:
\begin{equation}\label{eq:2.9}
\langle N_b^* N_v, N_w \rangle = \langle N_u , N_b N_w \rangle = \langle N_u, N_{bw} \rangle
= \left\{ \begin{array}{ll}
1 & v = b, w = e, \\
0 & \text{otherwise.} 
\end{array} \right.
\end{equation}
This implies
\begin{equation}\label{eq:2.15}
\big( N_b^* N_v (h_2 N_a )_v \big)^P_e = \left\{ \begin{array}{ll}
(h_2 N_a)_v & b = v,\\
0 & b \neq v. 
\end{array} \right.
\end{equation}
With that the right hand side of \eqref{eq:2.8} can be rewritten as
\begin{equation}\label{eq:2.16}
\langle (N_b^* h_2 N_a)^P_e \cdot \lambda, h_1 \rangle = 
\big\langle \lambda, \big( (N_b^* h_2 N_a)^P_e \big)^{*_P} h_1 \big\rangle .
\end{equation} 
On the other hand 
\begin{equation}\label{eq:2.14}
\begin{aligned}
\big\langle m^\he \big( h_2 \cdot & (m^\he)^{-1}(N_a \otimes \lambda) \big), N_b \otimes h_1 \big \rangle 
\;=\; \big\langle h_2 \cdot (m^\he)^{-1}(N_a \otimes \lambda) , N_b h_1 \big\rangle \\
& =\; \big\langle (m^\he)^{-1}(N_a \otimes \lambda) , h_2^* N_b h_1 \big\rangle \;=\;
\langle N_a \otimes \lambda , m^{-1}(h_2^* N_b) h_1 \rangle \\
& =\; \sum\nolimits_{v \in W^P} \langle N_a \otimes \lambda, N_v \otimes (h_2^* N_b)^P_v h_1 \rangle \;=\;
\langle \lambda, (h_2^* N_b)^P_a h_1 \rangle 
\end{aligned}
\end{equation}
Using \eqref{eq:2.15} we identify the last expression in \eqref{eq:2.14} with
\[
\langle \lambda, (N_a^* h_2^* N_b)^P_e h_1 \rangle =
\big\langle \lambda, \big( (N_b^* h_2 N_a)^* \big)^P_e h_1 \big\rangle .
\]
Lemma \ref{lem:1.1} guarantees that this equals \eqref{eq:2.16}, which proves that
the bijection $m^\he$ is an $\mc H$-module homomorphism. 

By construction $m^{-1} : \mc H \to \mc H (W^P) \otimes \mc H^P$ and $m^\he$ transfer the pairing 
between $\mc H$ and $\mc H^\he$ to the pairing \eqref{eq:2.6}. Thus we realized 
$\mc H^\he = \ind_{\mc H^P}^{\mc H}(\mc H^P)^\he$ as $\ind^{\mc H}_{\mc H^P} (\mc H^{P\he})$.
\end{proof}

In the special case $P = \emptyset$, Lemma \ref{lem:2.4} provides an isomorphism of $\mc H$-modules
\begin{equation}
\ind_{\mc A}^{\mc H} (\mc A^\he) \cong \mc H^\he .
\end{equation} 
Here the embedding of $\mc A^\he$ in $\mc H^\he$ comes from \eqref{eq:2.12}:
\begin{equation}\label{eq:2.19}
\begin{array}{cccc}
\imath : & \mc A^\he & \to & \mc H^\he \\
 & a & \mapsto & (m^\he)^{-1}( N_e \otimes a) 
\end{array}.
\end{equation}
The next result generalizes \cite[Theorem 2.20]{Opd1} and \cite[Proposition 4.19]{Opd-Sp}.

\begin{prop}\label{prop:2.5}
Let $(\rho,V_\rho)$ be an $\mc H^P$-representation. The pairing \eqref{eq:2.6} induces an
isomorphism $\ind_{\mc H^P}^{\mc H} (\rho^\he) \cong \ind_{\mc H^P}^{\mc H} (\rho )^\he$.
\end{prop}
\begin{proof}
We abbreviate $\ind_{\mc H^P}^{\mc H}$ to ind for the duration of this proof.

Recall that $\{ N_w : w \in W \}$ is an orthonormal basis of $\mc H (W,q^\lambda)$ for the
inner product \eqref{eq:1.1}. Hence \eqref{eq:2.6} identifies 
\[
\mc H (W^P) \otimes_\C V_\rho^\he \cong \mc H \otimes_{\mc H^P} (V_\rho^\he ) \quad \text{with} 
\quad (\mc H (W^P) \otimes_\C  V_\rho )^\he \cong (\mc H \otimes_{\mc H^P} V_\rho )^\he . 
\]
It remains to show that
\begin{equation}\label{eq:2.7}
\langle \ind (\pi^\he) (h^*) x , y \rangle \quad \text{equals} \quad \langle x, \ind (\pi)(h) y \rangle
\end{equation}
for all $h \in \mc H, x \in \ind (V_\rho^\he), y \in \ind (V_\rho)$.

Choose a surjective $\mc H^P$-homomorphism $p : F \otimes \mc H^P \to V_\rho$, where $F \otimes
\mc H^P$ is a free $\mc H^P$-module. Dually, that yields an injective $\mc H^P$-homomorphism
$p^\he : V_\rho^\he \to (F \otimes \mc H^P )^\he$. For $v \in V_\rho$ with
a preimage $\tilde v \in F \otimes \mc H^P$ and $\lambda \in V_\rho^\he$ with image
$\tilde \lambda \in (F \otimes \mc H^P )^\he$, that means
$\langle \lambda , v \rangle = \langle \tilde \lambda, \tilde v \rangle$.

With the functoriality of induction we obtain a surjective $\mc H$-homomorphism
\[
\ind (p) : F \otimes \mc H = \ind (F \otimes \mc H^P) \to \ind (V_\rho) ,
\]
and an injective $\mc H$-homomorphism 
\[
\ind (p^\he) : \ind (V_\rho^\he) \to \ind \big( (F \otimes \mc H^P)^\he \big) .
\]
Now we encounter the minor complication that it is difficult to work with
$(F \otimes \mc H^P)^\he$ when $F$ has infinite dimension. We overcome that by playing it
via finitely generated submodules. 
Choose a finite dimensional linear subspace $F_y \subset F$ such that  
\[
y \in \ind (p) (F_y \otimes \mc H) = \ind \big( p (F_y \otimes \mc H^P) \big) .
\]
It follows from Lemma \ref{lem:2.4} that
\begin{equation}\label{eq:2.13}
\ind_{\mc H^P}^{\mc H} \big( F_y^\he \otimes \mc H^{P\he} \big) = F_y^\he \otimes 
\ind_{\mc H^P}^{\mc H} ( \mc H^{P\he}) \cong \ind_{\mc H^P}^{\mc H} ( F_y \otimes \mc H^P )^\he ,
\end{equation}
with the pairing \eqref{eq:2.6}. The 
transpose of the inclusion $i_y : F_y \otimes \mc H^P \to F \otimes \mc H^P$ is the projection 
\[
i_y^\he : (F \otimes \mc H^P)^\he \to (F_y \otimes \mc H^P )^\he = F_y^\he \otimes (\mc H^P )^\he .
\]
To these maps we can also apply ind. In that way \eqref{eq:2.7} can be evaluated via the
pairing of $\ind (F_y \otimes \mc H^P)$ with $\ind \big( F_y^\he \otimes (\mc H^P )^\he \big)$
given by \eqref{eq:2.6}. More explicitly:
\[
\langle x, \ind (\pi)(h) y \rangle = \langle \ind (p^\he) (x) , h \, \ind (p) (y) \rangle =
\langle \ind (i_y^\he) \ind (p^\he) (x) , h \, \ind (p) (y) \rangle .
\]
By \eqref{eq:2.13} the right-hand side equals
\begin{align*}
\langle h^* \, \ind (i_y^\he) \ind (p^\he) (x) , \ind (p) (y) \rangle & =
\langle \ind (i_y^\he) \ind (p^\he) (\ind (\pi^\he) (h^*) x) , \ind (p) (y) \rangle \\
& = \langle \ind (\pi^\he) (h^*) x , y \rangle .
\end{align*}
This establishes \eqref{eq:2.7}.
\end{proof}

\section{Second adjointness} 
\label{sec:adjoint}

For affine Hecke algebras the standard adjointness for parabolic induction reads 
\begin{equation}\label{eq:3.1}
\Hom_{\mc H} \big( \ind_{\mc H^P}^{\mc H} (\rho), \pi \big) \cong
\Hom_{\mc H^P} \big( \rho, \Res_{\mc H^P}^{\mc H} (\pi) \big) 
\quad \rho \in \Mod (\mc H^P), \pi \in \Mod (\mc H).
\end{equation}
This can be regarded as an instance of Frobenius reciprocity or of Hom-tensor duality
(since $\ind_{\mc H^P}^{\mc H} (V_\rho) = \mc H \otimes_{\mc H^P} V_\rho$). In terms of
reductive $p$-adic groups, normalized parabolic induction and normalized Jacquet restriction, 
\eqref{eq:3.1} corresponds to Bernstein's second adjointness:
\begin{equation}\label{eq:3.2}
\Hom_G \big( I_P^G (\sigma), \tau \big) \cong \Hom_M \big( \sigma, J^G_{\overline P}(\tau) \big) ,
\end{equation}
where $\sigma \in \Rep (M), \tau \in \Rep (G)$ and $P, \overline{P}$ are opposite parabolic
subgroups of $G$ with $P \cap \overline{P} = M$. The comparison between the two settings
stems from \cite[Corollary 8.4]{BuKu}, but one needs some modifications that lead to
\cite[Condition 4.1]{SolComp}. By analogy, the first adjointness for $p$-adic groups (i.e.
Frobenius reciprocity)
\begin{equation}\label{eq:3.3}
\Hom_M \big( J_P^G (\tau), \sigma \big) \cong \Hom_G \big( \tau, I_P^G (\sigma) \big)
\end{equation}
should have a counterpart for affine Hecke algebras. In other words, we may expect that some
form of parabolic restriction is left adjoint to some form of parabolic induction.
By Frobenius reciprocity for co-induced modules:
\begin{equation}\label{eq:3.4}
\Hom_{\mc H^P} \big( \Res^{\mc H}_{\mc H^P} (\pi), \rho \big) \cong 
\Hom_{\mc H} \big( \pi, \Hom_{\mc H^P} (\mc H, \rho) \big) ,
\end{equation}
where $\mc H$ acts on $\Hom_{\mc H^P} (\mc H, \pi)$ via right multiplication on $\mc H$. 
However, \eqref{eq:3.4} is not yet satisfactory because it does not provide a left adjoint
for parabolic induction. 

For $p$-adic groups, one way to prove the second adjointness relation is via contragredients and 
Jacquet modules, see \cite[\S VI.9.6]{Ren}. For graded Hecke algebras, a similar proof works
with Hermitian duals instead of contragredients \cite[Lemma 3.8.1]{BaCi}. We follow the latter.

\begin{thm}\label{thm:3.1}
Let $P \subset \Delta$ and recall that $P^{op} = w_\Delta (-P)$.
\enuma{
\item The right adjoint of $\Res^{\mc H}_{\mc H^P}$ is $\ind_{\mc H^{P^{op}}}^{\mc H} \circ
\psi_{\Delta P *}$.
\item The left adjoint of $\ind_{\mc H^P}^{\mc H}$ is $\psi_{\Delta P}^* \circ 
\Res^{\mc H}_{\mc H^{P^{op}}}$.
}
\end{thm}
\begin{proof}
(a) Let $(\pi, V_\pi) \in \Mod (\mc H)$ and $(\rho, V_\rho) \in \Mod (\mc H^P)$. By the transposition
isomorphism \eqref{eq:2.10} and Proposition \ref{prop:2.3}
\begin{equation}\label{eq:3.5}
\Hom_{\mc H^P} \big( \Res^{\mc H}_{\mc H^P} (\pi), \rho^\he \big) \cong \Hom_{\mc H^P} \big( \rho,
\Res^{\mc H}_{\mc H^P} (\pi)^\he \big) \cong \Hom_{\mc H^P} \big( \rho, 
\Res^{\mc H}_{\mc H^{P^{op}}} (\pi^\he) \circ \psi_{\Delta P} \big) .
\end{equation}
We know from Lemma \ref{lem:2.2} that $\psi_{\Delta P} : \mc H^P \to \mc H^{P^{op}}$ is invertible,
so that the right-hand side of \eqref{eq:3.5} becomes isomorphic with
\begin{equation}\label{eq:3.6}
\Hom_{\mc H^{P^{op}}} \big( \rho \circ \psi_{\Delta P}^{-1}, 
\Res^{\mc H}_{\mc H^{P^{op}}} (\pi^\he) \big) .
\end{equation}
Now we apply Frobenius reciprocity in the form \eqref{eq:3.1} and again the transposition isomorphism:
\begin{equation}\label{eq:3.7}
\cong \Hom_{\mc H} \big( \ind_{\mc H^{P^{op}}}^{\mc H} (\rho \circ \psi_{\Delta P}^{-1}), \pi^\he \big)
\cong \Hom_{\mc H} \big( \pi,  \ind_{\mc H^{P^{op}}}^{\mc H} (\rho \circ \psi_{\Delta P}^{-1})^\he \big).
\end{equation}
Using Proposition \ref{prop:2.5}, we identify that with
\begin{equation}\label{eq:3.8}
\Hom_{\mc H} \big( \pi,  \ind_{\mc H^{P^{op}}}^{\mc H} (\rho^\he \circ \psi_{\Delta P}^{-1}) \big) =
\Hom_{\mc H} \big( \pi,  \ind_{\mc H^{P^{op}}}^{\mc H} \circ \psi_{\Delta P} (\rho^\he) \big) .
\end{equation}
The first isomorphism in \eqref{eq:3.5} and the second in \eqref{eq:3.7} are conjugate-linear. The
other above isomorphisms are complex linear, so the composition of \eqref{eq:3.5}--\eqref{eq:3.8}
is again a complex linear bijection. That proves the desired adjointness relation for
$(\pi, \rho' = \rho^\he)$, so whenever $\rho'$ is the Hermitian dual of some $\mc H^P$-module. The same
argument as in the analogous situation for reductive $p$-adic groups \cite[p. 232]{Ren} shows why that
implies part (a) for all $(\pi, \rho')$.\\
(b) Reverse the roles of $P$ and $P^{op}$ and apply part (a) with 
$\rho' = \psi_{\Delta P}^* (\rho) = \rho \circ \psi_{\Delta P}$. That gives isomorphisms
\[
\Hom_{\mc H} \big( \pi, \ind_{\mc H^P}^{\mc H} (\rho') \big) \cong 
\Hom_{\mc H^{P^{op}}} \big( \Res^{\mc H}_{\mc H^{P^{op}}} (\pi), \rho \big) .
\]
Left composition with $\psi_{\Delta P}^{-1}$ on both terms of the right-hand side makes this isomorphic
with $\Hom_{\mc H^{P}} \big( \psi_{\Delta P}^* \, \Res^{\mc H}_{\mc H^{P^{op}}} (\pi), \rho' \big)$.
\end{proof}

Next we discuss a topic related to second adjointness, namely expressions for parabolic induction
followed by parabolic restriction. In the setting of reductive $p$-adic groups this is known as
Bernstein's geometric lemma \cite[\S VI.5.1]{Ren}. A version for affine Hecke algebras should
provide a filtration of the functor $\Res^{\mc H}_{\mc H^Q} \ind_{\mc H^P}^{\mc H}$. Indeed that
was achieved in \cite[\S 11]{DeOp1}, but restricted to tempered representations. 
Here we formulate that result in larger generality. 

Let $P,Q \subset \Delta$ and let 
\[
W^{P,Q} = \{ w \in W : w(Q) \subset R^+, w^{-1}(Q) \subset R^+ \}
\] 
be the set of shortest length representatives of $W_P \backslash W / W_Q$. Each $d \in W^{P,Q}$ yields 
a bijection $d^{-1}(P) \cap Q \to P \cap d(Q)$ and an algebra isomorphism
\begin{equation}\label{eq:3.9}
\begin{array}{llll}
\psi_d : & \mc H^{d^{-1}(P) \cap Q} & \to & \mc H^{P \cap d(Q)} \\
& N_w \theta_x & \mapsto & N_{d w d^{-1}} \theta_{d(x)} 
\end{array}.
\end{equation}
We choose a total ordering of $W^{P,Q}$ such that $\ell : W^{P,Q} \to \Z_{\geq 0}$ becomes a weakly
increasing function. For $d \in W^{P,Q}$ and an $\mc H^Q$-representation $(\pi,V_\pi)$, we consider 
the linear subspace
\begin{equation}\label{eq:3.11}
\big( \Res^{\mc H}_{\mc H^P} \ind_{\mc H^Q}^{\mc H} \big)_{\leq d} (V_\pi) = 
\bigoplus\nolimits_{d' \in W^{P,Q}, d' \leq d} \mc H (W_P d' W_Q) \mc A \otimes_{\mc H^Q} V_\pi 
\end{equation}
of $\ind_{\mc H^Q}^{\mc H} (V_\pi)$. To analyse these subspaces, we need a result of Kilmoyer 
\cite[Theorem 2.7.4]{Car}:
\begin{equation}\label{eq:3.10}
W_P \cap d W_Q d^{-1} = W_{P \cap d (Q)} \qquad \text{for all } d \in W^{P,Q}.
\end{equation}
Using that, the following is shown in \cite[(11.3)--(11.6)]{DeOp1}:

\begin{prop}\label{prop:3.2}
For each $d \in W^{P,Q}$, $\big( \Res^{\mc H}_{\mc H^P} \ind_{\mc H^Q}^{\mc H} \big)_{\leq d} (V_\pi)$ 
is an $\mc H^P$-submodule of $\Res^{\mc H}_{\mc H^P} \ind_{\mc H^Q}^{\mc H} (V_\pi)$. There is an
isomorphism of $\mc H^P$-modules
\[
\big( \Res^{\mc H}_{\mc H^P} \ind_{\mc H^Q}^{\mc H} \big)_{\leq d} (V_\pi) \big/ \big( 
\Res^{\mc H}_{\mc H^P} \ind_{\mc H^Q}^{\mc H} \big)_{<d} (V_\pi) \cong \ind_{\mc H^{P \cap 
d(Q)}}^{\mc H^P} \big( \psi_d \: \Res^{\mc H^Q}_{\mc H^{d^{-1}(P) \cap Q}} (V_\pi) \big) ,
\]
where $<d$ means $\leq d'$ for the largest $d' \in W^{P,Q}$ which is smaller than $d$.

In other words, we have a filtration of the functor $\Res^{\mc H}_{\mc H^P} \ind_{\mc H^Q}^{\mc H}$,
indexed by $W^{P,Q}$ and with successive subquotients $\ind_{\mc H^{P \cap d(Q)}}^{\mc H^P} \circ 
\psi_{d *} \circ \Res^{\mc H^Q}_{\mc H^{d^{-1}(P) \cap Q}}$. 
\end{prop}

Notice the analogy with Mackey's restriction-induction formula for representations of finite groups.
From Proposition \ref{prop:3.2} and the two adjunctions, one can derive expressions for the
Hom-space between two parabolically induced $\mc H$-representations.

\section{Variations on the Langlands classification}
\label{sec:Langlands}

The Langlands classification for a reductive group $G$ over a local field \cite{Lan,Ren} classifies
irreducible admissible $G$-representations in terms of irreducible tempered representations of Levi
subgroups of $G$. The analogous result for affine/graded Hecke algebras can be found in \cite{Eve,SolAHA}.
Here we want to establish some useful variations, in particular with subrepresentations instead of
quotients.

The complex torus $T$ can be idenfied with the space $\Irr (\mc A)$ of irreducible representations
of $\mc A \cong \C [X] = \mc O (T)$. If $(\pi, V_\pi)$ is an $\mc H$-representation, $t \in T$
and there exists $v \in V_\pi \setminus \{0\}$ such that
\[
\pi (\theta_x) v = t(x) v \qquad \text{for all } x \in X,
\]
then $t$ is a called an $\mc A$-weight (or simply weight) of $\pi$. We denote the set of 
$\mc A$-weights of $(\pi,V_\pi)$ by Wt$(\pi)$ or Wt$(V_\pi)$.
If $V_\pi$ has finite dimension,
then there is a canonical decomposition in generalized $\mc A$-eigenspaces:
\begin{equation}\label{eq:4.2}
V_\pi = \bigoplus\nolimits_{t \in T} V_{\pi,t,\mr{gen}} .
\end{equation}
The $t \in T$ for which $V_{\pi,t,\mr{gen}} \neq 0$ are precisely the $\mc A$-weights of $\pi$.

\begin{lem}\label{lem:4.1}
Let $(\pi,V_\pi)$ be a finite dimensional $\mc H$-representation. Then\\
$\mr{Wt}(\pi^\he) = \big\{ \overline{w_\Delta t}^{-1} : t \in \mr{Wt}(\pi) \big\}$.
\end{lem}
\begin{proof}
Let $s \in T$ be a weight of $\pi^\he$, with an eigenvector $\lambda \in V_\pi^\he \setminus \{0\}$.
For any $v \in V, x \in X$ we compute, using \eqref{eq:1.3},
\begin{equation}\label{eq:4.1}
s(x) \lambda (v) = (\pi^\he (\theta_x) \lambda) (v) =
\lambda (\pi (\theta_x^*) v) = \lambda (\pi (N_{w_\Delta} \theta_{-w_\Delta (x)} N_{w_\Delta}^{-1}) v) .
\end{equation}
Write $\pi' = \pi \circ \mf{c}_{N_{w_\Delta}}$, so that 
\[
\begin{array}{ccc}
(\pi, V_\pi) & \to & (\pi', V_\pi) \\
v & \mapsto & \pi (N_{w_\Delta}) v
\end{array}
\]
is an isomorphism of $\mc H$-representations. We can rewrite \eqref{eq:4.1} as
\[
\lambda (\overline{s(x)} v) = \lambda (\pi' (\theta_{-w_\Delta (x)}) v) .
\]
Equivalently, for each $x \in X, v \in V_\pi$ the kernel of $\lambda$ contains
\[
\pi' \big( \theta_{-w_\Delta (x)} - \overline{s(x)} \big) v = 
\pi' \big( \theta_{x'} - \overline{s}(-w_\Delta (x')) \big) v = 
\pi' \big( \theta_{x'} - \overline{w_\Delta s}^{-1}(x') \big) v ,
\]
where we abbreviated $x' = -w_\Delta (x)$. Thus $\pi' \big( \theta_{x'} - 
\overline{w_\Delta s}^{-1}(x') \big)$ is not surjective, for any $x' \in X$. Since $V_\pi$
has finite dimension, we can use the decomposition, which shows that $\overline{w_\Delta s}^{-1}$
is a weight of $\pi'$. Via the isomorphism $\pi' \cong \pi$, it is also a weight of $\pi$.

Hence $s \mapsto \overline{w_\Delta s}^{-1}$ maps the weights of $\pi^\he$ to the weights of
$\pi$. As $\dim V_\pi < \infty$, $\pi^{\he \he} = \pi$ and the same arguments apply with the
roles of $\pi$ and $\pi^\he$ exchanged. Therefore we have found a bijection between the set of
weights of $\pi^\he$ and of $\he$, with inverse $t \mapsto \overline{w_\Delta t}^{-1}$.
\end{proof}

For any $t \in T$ we have $|t| \in \Hom_\Z (X,\R_{>0})$ and $\log |t| \in \Hom_\Z (X,\R) =
Y \otimes_\Z \R$. Given $P \subset \Delta$ we define the positive cones
\begin{align*}
& (Y \otimes_\Z \R )^{P+} = \{ y \in Y \otimes_\Z \R : \langle \alpha ,y \rangle = 0 \; \forall 
\alpha \in P, \langle \alpha ,y \rangle >0 \; \forall \alpha \in \Delta \setminus P \} ,\\
& T^{P+} = \exp \big( (Y \otimes_\Z \R)^{P+} \big)  \qquad \subset T^P .
\end{align*}
The same can be done in $X \otimes_\Z \R$, and then taking anti-duals yields the obtuse negative
cones
\begin{align*}
& (Y \otimes_\Z \R )_P^- = \big\{ \sum\nolimits_{\alpha \in P} c_\alpha \alpha^\vee : 
c_\alpha \in \R_{\leq 0} \big\} ,\\
& T_P^- = \exp \big( (Y \otimes_\Z \R )_P^- \big) \qquad \subset T_P .
\end{align*}
By definition, a finite dimensional $\mc H^P$-module $V$ is tempered if $|t| \in T_P^-$ for all
$t \in \mr{Wt}(V)$. Similarly we say that $V$ is anti-tempered if $|t|^{-1} \in T_P^-$
for all $t \in \mr{Wt}(V)$. These two properties are preserved by taking Hermitian duals:

\begin{lem}\label{lem:4.2}
Let $(\pi,V_\pi)$ be a finite dimensional $\mc H$-representation. If $V_\pi$ is tempered (resp.
anti-tempered), then $V_\pi^\he$ is tempered (resp. anti-tempered).
\end{lem}
\begin{proof}
Since $-w_\Delta$ stabilizes $\Delta$, $w_{\Delta} s^{-1} \in T^{\Delta -}$ if and only if
$s \in T^{\Delta -}$. Apply that to $s = |t|$ (resp. $s = |t|^{-1}$) for a weight $t$ of $V_\pi$,
and use Lemma \ref{lem:4.1}.
\end{proof}

The following result is an obvious generalization of the Langlands classification for affine 
Hecke algebras \cite{Eve,SolAHA}.

\begin{thm}\label{thm:4.3}
Let $\pi \in \Irr (\mc H^P)$ and $t \in T^P$. Suppose that (i) or (ii) holds:
\begin{itemize}
\item[(i)] $\pi$ is tempered and $t \in T^{P+}$,
\item[(ii)] $\pi$ is anti-tempered and $t^{-1} \in T^{P+}$.
\end{itemize}
\enuma{
\item The $\mc H$-representation $\ind_{\mc H^P}^{\mc H}(\pi \otimes t)$ has a unique irreducible
quotient $L(P,\pi,t)$. It is the unique irreducible subquotient $\rho$ of 
$\ind_{\mc H^P}^{\mc H}(\pi \otimes t)$ which admits an injective $\mc H^P$-homomorphism
$\pi \otimes t \to \Res^{\mc H}_{\mc H^P} (\rho)$.
\item Every irreducible $\mc H$-representation is of the form $L(P,\pi,t)$, for unique 
$(P,\pi,t)$ as in (i). This also holds with (ii) instead of (i). 
}
\end{thm}
\begin{proof}
(i) The Langlands classification, as in \cite[Theorem 2.1]{Eve} and \cite[Theorem 2.2.4]{SolAHA},
states (a) and (b). Although the characterizing property of $L(P,\pi,t)$ is not made explicit in 
these sources, it plays an important role in \cite[\S 2.7]{Eve} and in 
\cite[proof of Theorem 2.2.4.a]{SolAHA}. \\
(ii) The same proof as for (i) applies, when we rewrite all the arguments in $Y \otimes_\Z \R$ with
respect to $-\Delta$ instead of $\Delta$.
\end{proof}

An $\mc H$-representation $\ind_{\mc H^P}^{\mc H}(\pi \otimes t)$ as in Theorem \ref{thm:4.3}.i
is called a standard module, and $L(P,\pi,t)$ is called its Langlands quotient.
With the usage of Hermitian duals, we can deduce a version of Theorem \ref{thm:4.3} in terms of
``Langlands" subrepresentations.

\begin{prop}\label{prop:4.4}
Let $\pi \in \Irr (\mc H^P)$ and $t \in T^P$. Suppose that (i) or (ii) holds:
\begin{itemize}
\item[(i)] $\pi$ is tempered and $t^{-1} \in T^{P+}$,
\item[(ii)] $\pi$ is anti-tempered and $t \in T^{P+}$.
\end{itemize}
\enuma{
\item The $\mc H$-representation $\ind_{\mc H^P}^{\mc H}(\pi \otimes t)$ has a unique irreducible
subrepresentation, which we call the Langlands subrepresentation $\tilde{L}(P,\pi,t)$. It is the 
unique irreducible subquotient $\sigma$ of $\ind_{\mc H^P}^{\mc H}(\pi \otimes t)$ that admits a 
surjective $\mc H^P$-homomorphism 
$\Res_{\mc H^{P^{op}}}^{\mc H} (\sigma) \circ \psi_{\Delta P} \to \pi \otimes t$.
\item Every irreducible $\mc H$-representation is of the form $\tilde{L}(P,\pi,t)$ for unique
$(P,\pi,t)$ as in (i). This also holds with (ii) instead of (i).
}
\end{prop}
\begin{proof}
We assume (i). The proof when (ii) holds is completely analogous, only using the other assumption
in Theorem \ref{thm:4.3}.\\
(a) By Proposition \ref{prop:2.5} and Lemma \ref{lem:2.6}
\begin{equation}\label{eq:4.3}
\ind_{\mc H^P}^{\mc H}(\pi \otimes t )^\he \cong \ind_{\mc H^P}^{\mc H} \big( (\pi \otimes t)^\he \big)
\cong \ind_{\mc H^P}^{\mc H}(\pi^\he \otimes \overline{t}^{-1}) .
\end{equation}
Here $\bar t = t$ because it is real-valued, and we know from Lemma \ref{lem:4.2} that $\pi^\he$
is tempered. Theorem \ref{thm:4.3}.a says that \eqref{eq:4.3} has a unique irreducible 
quotient $L(P,\pi^\he,t^{-1})$, which can be characterized by the existence of an injection
\[
\pi^\he \otimes t^{-1} \to \Res_{\mc H^P}^{\mc H} \big( L(P,\pi^\he,t^{-1}) \big). 
\]
Passing to Hermitian duals and using \eqref{eq:2.10}, we find that 
$\ind_{\mc H^P}^{\mc H}(\pi \otimes t)$ has a unique irreducible subrepresentation 
$\sigma \cong L(P,\pi^\he,t^{-1})^\he$. Via Proposition \ref{prop:2.3}.a, the characterizing
property becomes a surjection 
$\Res_{\mc H^{P^{op}}}^{\mc H} (\sigma) \circ \psi_{\Delta P} \to \pi \otimes t$.\\
(b) Let $\tau \in \Irr (\mc H)$. With Theorem \ref{thm:4.3}.b we write $\tau^\he \cong L(P,\pi,t)$ 
for suitable $P \subset \Delta$, tempered $\pi \in \Irr (\mc H^P)$ and $t \in T^{P+}$. 
Then \eqref{eq:2.10} gives an injection 
\[
\tau \cong L(P,\pi,t)^\he \to \ind_{\mc H^P}^{\mc H} (\pi \otimes t)^\he .
\]
From the proof of part (a) we know that the right hand side is isomorphic with
$\ind_{\mc H^P}^{\mc H}(\pi \otimes t^{-1})$, where $(P,\pi,t^{-1})$ is as in (i). Then part (a)
says $\tau \cong \tilde L (P,\pi,t^{-1})$. The uniqueness
in Theorem \ref{thm:4.3}.b implies the uniqueness of $(P,\pi,t^{-1})$.
\end{proof}

We would like to express the unique Langlands quotient or subrepresentation from Theorem 
\ref{thm:4.3} and Proposition \ref{prop:4.4} as the coimage or image of a suitable intertwining 
operator. To that end we establish the uniqueness (up to scalars) of those operators.

\begin{lem}\label{lem:4.8}
Let $\pi \in \Irr (\mc H^P)$ and assume that the $W$-stabilizer of $t$ is contained in $W_P$
for all $t \in \mr{Wt}(\pi)$.
Then $\Res^{\mc H}_{\mc H^P} \ind_{\mc H^Q}^{\mc H} (\psi_{\gamma *} \pi)$ is a direct sum of
$\mc H^P$-representations $\ind_{\mc H^{P \cap d(Q)}}^{\mc H^P} (\psi_{d *} \psi_{\gamma *} \pi)$ 
with $d \in W^{P,Q}$, whose sets of $Z(\mc H^P)$-weights are mutually disjoint.
\end{lem}
\begin{proof}
From Proposition \ref{prop:3.2}.c we know that $\Res^{\mc H}_{\mc H^P} \ind_{\mc H^Q}^{\mc H} 
(\psi_{\gamma *} \pi)$ has a filtration with successive subquotients
\begin{equation}\label{eq:4.5}
\ind_{\mc H^{P \cap d(Q)}}^{\mc H^P} \big( \psi_d \, \Res^{\mc H^Q}_{\mc H^{d^{-1}(P) \cap Q}} 
(\psi_\gamma \, \pi) \big) = \ind_{\mc H^{P \cap d(Q)}}^{\mc H^P} (\psi_d \psi_\gamma \, \pi )
\qquad d \in W^{P,Q}.
\end{equation}
By construction $\mr{Wt}(\psi_{d *} \psi_{\gamma *} \pi) = d \gamma \mr{Wt}(\pi)$, and with
\cite[Proposition 4.20]{Opd-Sp} we obtain
\begin{equation}\label{eq:4.9}
\mr{Wt} \big( \ind_{\mc H^{P \cap d(Q)}}^{\mc H^P} (\psi_{d *} \psi_{\gamma *} \, \pi ) \big) 
\: \subset \: W_P d \gamma \mr{Wt}(\pi) .
\end{equation}
Equivalently, the set of $Z(\mc H^P)$-weights of 
$\ind_{\mc H^{P \cap d(Q)}}^{\mc H^P} (\psi_{d *} \psi_{\gamma *} \, \pi )$ is contained in\\ 
$W_P d \gamma \mr{Wt}(\pi) / W_P$. Suppose that $d,d' \in W^{P,Q}, t, t' \in \mr{Wt}(\pi)$ and 
\[
W_P d \gamma t \cap W_P d' \gamma t' \neq \emptyset .
\]
Pick $w_1,w_2 \in W_P$ such that $w_1 d \gamma t = w_2 d' \gamma t'$. By the irreducibility of
$\pi$, $t$ and $t'$ belong to the same $W_P$-orbit. Furthermore we assumed $W_t \subset W_P$, so
\begin{equation}\label{eq:4.10}
w_1 d \gamma W_P = w_2 d' \gamma W_P .
\end{equation}
From that we obtain $\gamma^{-1} d^{-1} w_1^{-1} w_2 d' \gamma \in W_P$ and
\[
d^{-1} w_1^{-1} w_2 d' \in \gamma W_P \gamma^{-1} = W_Q .
\]
We note that now $w_1^{-1} w_2 d' \in W_P d' \cap d W_Q$. As $W^{P,Q}$ represents $W_P \backslash
W / W_Q$, this shows that $d' = d$. 

Thus, for different $d,d' \in W^{P,Q}$ the $\mc H^P$-representations \eqref{eq:4.5} have disjoint
sets of $\mc A$-weights and disjoint sets of $Z(\mc H^P)$-weights. In particular every extension
of one of these modules by the other is a trivial extension. It follows that the aforementioned
filtration of $\Res^{\mc H}_{\mc H^P} \ind_{\mc H^Q}^{\mc H} (\psi_{\gamma *} \pi)$ actually splits,
and that $\Res^{\mc H}_{\mc H^P} \ind_{\mc H^Q}^{\mc H} (\psi_{\gamma *} \pi)$ is the direct sum of
the modules \eqref{eq:4.5}. 
\end{proof}

The conditions in Lemma \ref{lem:4.8} are often satisfied, but they do not cover all cases of
Theorem \ref{thm:4.3}. Inspired by \cite[\S VII.3.3]{Ren}, we say that an $\mc H^P$-representation 
$\pi$ is $W,\!P$-regular if 
\begin{equation}\label{eq:4.4}
w t \notin \mr{Wt}(\pi) \text{ for all } t \in \mr{Wt}(\pi) \text{ and }
w \in W_P (W^{P,P} \setminus \{e\}).
\end{equation}
Let $P,Q \subset \Delta$ and $\gamma \in W$, such that $\gamma (P) = Q$. Like in \eqref{eq:3.9},
there is an algebra isomorphism $\psi_\gamma : \mc H^P \to \mc H^Q$.

\begin{lem}\label{lem:4.5}
Let $\pi \in \Irr (\mc H^P)$ be $W,\!P$-regular. 
\enuma{
\item $\pi$ has multiplicity one in $\Res^{\mc H}_{\mc H^P} \ind_{\mc H^Q}^{\mc H} 
(\psi_{\gamma *} \pi)$, and is a direct summand of the latter. 
\item $\dim \Hom_{\mc H} \big( \ind_{\mc H^P}^{\mc H} (\pi), \ind_{\mc H^Q}^{\mc H} 
(\psi_{\gamma *} \, \pi) \big) = 1$.
}
\end{lem}
\begin{proof}
(a) The element $\gamma^{-1} \in W$ belongs to $W^{P,Q}$ because $\gamma^{-1}(Q) \subset R^+$ and
$\gamma (P) \subset R^+$. We follow the proof of Lemma \ref{lem:4.8}, with $d' = \gamma^{-1}$. 
This time we cannot conclude \eqref{eq:4.10}, but our weaker assumption still provides a reasonable 
substitute. Namely, from $w_1 d \gamma t = w_2 t'$ we get $w_2^{-1} w_1 d \gamma t \in 
\mr{Wt}(\pi)$, which by the $W,\!P$-regularity of $\pi$ implies
\begin{equation}\label{eq:4.15}
w_2^{-1} w_1 d \gamma \notin W_P (W^{P,P} \setminus \{e\}).
\end{equation}
Notice that $d \gamma (P) = d(Q) \subset R^+$, which says that $d \gamma \in W^P$. By 
\cite[Proposition 2.7.5]{Car} we can write $d \gamma = a \tilde d$ with $\tilde d \in W^{P,P}$ 
and $a \in W_P$. It follows that $W_P d \gamma = W_P \tilde d$. Then \eqref{eq:4.15} forces 
\[
\tilde d = e \quad \text{and} \quad d \gamma \in W_P \cap W^P = \{e\}. 
\]
Hence the representations \eqref{eq:4.5} with $d \neq \gamma^{-1}$ do not 
have the central character of $\pi \in \Irr (\mc H^P)$ as $Z(\mc H^P)$-weight. Like in the proof
of Lemma \ref{lem:4.8}, this entails that $\pi$ appears with multiplicity one in 
$\Res^{\mc H}_{\mc H^P} \ind_{\mc H^Q}^{\mc H} (\psi_{\gamma *} \pi)$, as a direct summand.\\
(b) By Frobenius reciprocity
\[
\Hom_{\mc H} \big( \ind_{\mc H^P}^{\mc H} (\pi), \ind_{\mc H^Q}^{\mc H} (\psi_{\gamma *} \,
\pi) \big) \cong \Hom_{\mc H^P} \big( \pi, \Res^{\mc H}_{\mc H^P} \ind_{\mc H^Q}^{\mc H} 
(\psi_{\gamma *} \pi) \big) .
\]
Now apply part (a).
\end{proof}

Lemma \ref{lem:4.5} tells us that, whenever $\pi \in \Irr (\mc H^P)$ is $W,\!P$-regular, 
there exists a nonzero intertwining operator
\begin{equation}\label{eq:4.8}
I (\gamma, P, \pi) : \ind_{\mc H^P}^{\mc H}(\pi) \to 
\ind_{\mc H^Q}^{\mc H} (\psi_{\gamma *} \, \pi),
\end{equation}
unique up to scalars. The $\Delta P$-genericity is only a very mild restriction. Namely,
for every finite dimensional $\mc H^P$-representation $\tau$ there exists a Zariski-open nonempty 
subset $T^P_\tau \subset T^P$ such that $\tau \otimes t$ is $\Delta P$-generic for all 
$t \in T^P_\tau$.

The next result and its proof are similar to \cite[Th\'eor\`eme VII.4.2]{Ren}.

\begin{thm}\label{thm:4.7}
Let $P \subset \Delta$ and $\pi \in \Irr (\mc H^P)$.
\enuma{
\item Suppose that $\pi$ is $W,\!P$-regular. Then $\ind_{\mc H^P}^{\mc H}(\pi)$ has a unique
irreducible quotient, namely
\[
\ind_{\mc H^P}^{\mc H} (\pi) / \ker I (w_\Delta w_P, P, \pi) \cong 
\mr{im} \, I (w_\Delta w_P, P, \pi) .
\]
\item Suppose that $\psi_{w_\Delta w_P *} \pi$ is $W,\!P^{op}$-regular. Then 
$\ind_{\mc H^P}^{\mc H}(\pi)$ has a unique irreducible subrepresentation, namely the image of
\[
I ( w_P w_\Delta , P^{op}, \psi_{w_\Delta w_P *} \pi ) : \ind_{\mc H^{P^{op}}}^{\mc H} 
(\psi_{w_\Delta w_P *} \pi) \to \ind_{\mc H^P}^{\mc H} (\pi) .
\]
} 
\end{thm}
\begin{proof}
(a) Let $\rho$ be any quotient $\mc H$-representation of $\ind_{\mc H^P}^{\mc H}(\pi)$. The
quotient map gives a nonzero element of
\[
\Hom_{\mc H}(\ind_{\mc H^P}^{\mc H}(\pi), \rho) \cong 
\Hom_{\mc H^P}(\pi, \Res^{\mc H}_{\mc H^P} \rho) ,
\]
so $\pi$ is a subrepresentation of $\Res^{\mc H}_{\mc H^P} \rho$. By Lemma
\ref{lem:4.5}.a with $\gamma = e$, $\pi$ is a direct summand of $\Res^{\mc H}_{\mc H^P} \rho$,
and appears with multiplicity one. The projection $\Res^{\mc H}_{\mc H^P} \rho \to \pi$
and the adjunction from Theorem \ref{thm:3.1}.a yield a nonzero $\mc H$-homomorphism from
$\rho$ to $\ind_{\mc H^{P^{op}}}^{\mc H} (\psi_{\Delta P *} \pi)$. Thus we have 
$\mc H$-homomorphisms
\begin{equation}\label{eq:4.6}
\ind_{\mc H^P}^{\mc H} (\pi) \to \rho \to \ind_{\mc H^{P^{op}}}^{\mc H} (\psi_{\Delta P *} \pi) .
\end{equation}
Suppose now that $\rho$ is irreducible. Then the second map in \eqref{eq:4.6} is injective,
and first map is surjective by definition, so the composition of the two maps in \eqref{eq:4.6}
is nonzero. Lemma \ref{lem:4.5}.a guarantees that \eqref{eq:4.6} is a multiple of 
$I(w_\Delta w_P,P,\pi)$. In particular
\[
\ker I(w_\Delta w_P,P,\pi) = \ker \eqref{eq:4.6} = 
\ker \big( \ind_{\mc H^P}^{\mc H} (\pi) \to \rho \big) .
\]
We conclude that $\rho$ equals 
$\ind_{\mc H^P}^{\mc H} (\pi \otimes t) / \ker I(w_\Delta w_P,P,\pi)$.\\
(b) Let $\sigma$ be any subrepresentation of $\ind_{\mc H^P}^{\mc H}(\pi)$. The inclusion map
and Theorem \ref{thm:3.1}.b give a nonzero element of
\[
\Hom_{\mc H}(\sigma, \ind_{\mc H^P}^{\mc H}(\pi)) \cong
\Hom_{\mc H^P} \big( \psi_{\Delta P^{op} *} (\Res^{\mc H}_{\mc H^{P^{op}}} \sigma), \pi \big)
\cong \Hom_{\mc H^{P^{op}}} (\sigma, \psi_{\Delta P *} \pi ).
\]
In the setting of Lemma \ref{lem:4.5} we take $P^{op},P,\psi_{\Delta P^{op}},\psi_{\Delta P}\pi$
in the roles of, respectively, $P,Q,\gamma,\pi$. Then Lemma \ref{lem:4.5}.a says that
$\psi_{\Delta P *} \pi$ appears with multiplicity one in 
\begin{equation}\label{eq:4.11}
\Res^{\mc H}_{\mc H^{P^{op}}} \ind_{\mc H^P}^{\mc H} (\psi_{\Delta P^{op}} \psi_{\Delta P} \pi) = 
\Res^{\mc H}_{\mc H^{P^{op}}} \ind_{\mc H^P}^{\mc H} (\pi) ,
\end{equation}
as a direct summand. Since $\psi_{\Delta P} (\pi)$ appears in the subrepresentation 
$\Res^{\mc H}_{\mc H^{P^{op}}} \sigma$ of \eqref{eq:4.11}, it is also a direct summand thereof.
In particular there exists a nonzero element of 
\[
\Hom_{\mc H^{P^{op}}} (\psi_{\Delta P *} \pi, \Res^{\mc H}_{\mc H^{P^{op}}} \sigma) \cong
\Hom_{\mc H^P} (\ind_{\mc H^{P^{op}}}^{\mc H} (\psi_{\Delta P *} \pi), \sigma).
\]
Thus we have nonzero $\mc H$-homomorphisms
\begin{equation}\label{eq:4.12}
\ind_{\mc H^{P^{op}}}^{\mc H} (\psi_{\Delta P *} \pi) \to \sigma \to \ind_{\mc H^P}^{\mc H}(\pi) .
\end{equation}
Now we assume that $\sigma$ is irreducible. Then the first map in \eqref{eq:4.12} is surjective
and the second map is injective, so their composition is nonzero. The same argument as for
\eqref{eq:4.6} shows that $\sigma$ is isomorphic to
\[
\ind_{\mc H^{P^{op}}}^{\mc H} (\psi_{\Delta P *} \pi) / \ker I ( w_P w_\Delta , P^{op}, 
\psi_{w_\Delta w_P *} \pi ) \cong \mr{im}\, I ( w_P w_\Delta , P^{op}, \psi_{w_\Delta w_P *} \pi ) .
\qedhere
\]
\end{proof}

With Theorem \ref{thm:4.7} and Langlands' geometric lemmas \cite[\S 4]{Lan}, we can provide
alternative proofs of Theorem \ref{thm:4.3} and Proposition \ref{prop:4.4}.

\begin{prop}\label{prop:4.6}
Let $P \subset \Delta , \pi \in \Irr (\mc H^P)$  and $t \in T^P$. 
\enuma{
\item Suppose that
\begin{itemize}
\item[(i)] $\pi$ is tempered and $t \in T^{P+} \quad$ or
\item[(ii)] $\pi$ is anti-tempered and $t^{-1} \in T^{P+}$.
\end{itemize}
Then $\pi \otimes t$ is $W,\!P$-regular and the Langlands quotient $L(P,\pi \otimes t)$ 
from Theorem \ref{thm:4.3} equals 
$\ind_{\mc H^P}^{\mc H} (\pi \otimes t) / \ker I (w_\Delta w_P, P, \pi \otimes t)$.
\item Suppose that 
\begin{itemize}
\item[(iii)] $\pi$ is tempered and $t^{-1} \in T^{P+} \quad$ or
\item[(iv)] $\pi$ is anti-tempered and $t \in T^{P+}$.
\end{itemize}
Then $\psi_{w_\Delta w_P *} (\pi \otimes t)$ is $W,\!P^{op}$-regular and the Langlands 
subrepresentation\\ $\tilde{L}(P,\pi,t)$ from Proposition \ref{prop:4.4} 
is the image of $I \big( w_P w_\Delta , P^{op}, \psi_{w_\Delta w_P *} (\pi \otimes t) \big)$.
}
\end{prop}
\begin{proof}
First we establish the regularity in all four cases.\\
(i) Recall from \cite[Lemma 4.4]{Lan} that every $\lambda \in Y \otimes_\Z \R$ can be
expressed uniquely as 
\[
\lambda = \lambda_- + \lambda_+ , \quad \text{where} \quad \lambda \in (Y \otimes_\Z \R)_{Q-}
,\: \lambda_+ \in (Y \otimes_\Z \R)^{Q+} ,\: Q = Q(\lambda) \subset \Delta .
\]
For any $s \in \mr{Wt}(\pi)$ we have
\begin{equation}\label{eq:4.13}
\log |st|_+ = \log |t| ,\; \log |st|_- = \log |s| \text{ and } Q(\log |st|) = P .
\end{equation}
Assume that $w_1 \in W_P, w_2 \in W^{P,P} \setminus \{e\}$ and 
$w_1 w_2 s t = s' t \in \mr{Wt}(\pi \otimes t)$. Then 
\begin{equation}\label{eq:4.14}
w_2 s t = w_1^{-1} s' t \quad \text{and} \quad \log |w_2 s t|_+ = \log |w_1^{-1} s' t|_+ .
\end{equation}
By \cite[p. 38]{KrRa} $\log |w_1^{-1} s' t|_+ \geq \log |s' t|_+$ by \cite[(2.13)]{KrRa}
$\log |w_2 s t|_+ < \log |st|_+$. Together with \eqref{eq:4.13} we obtain
\[
\log |w_2 s t|_+ < \log |st|_+ = \log |s' t| \leq \log |w_1^{-1} s' t|_+ .
\]
That contradicts \eqref{eq:4.13} and hence our assumption is untenable. In other words,
$\pi \otimes t$ is $W,\!P$-regular.\\
(ii) Notice that the proof of (i) only involves root systems and
Weyl groups, no Hecke algebras. It can also be applied to the current $\pi \otimes t$,
when we replace the basis $\Delta$ of $R$ by $-\Delta$.\\
(iii) As $\psi_{w_\Delta w_P} : \mc H^P \to \mc H^{P^{op}}$ is an isomorphism that respects all
the structure of these affine Hecke algebras, $\psi_{w_\Delta w_P}(\pi)$ is tempered and
$w_\Delta w_P (t) \in T^{P^{op}}$. For $\alpha \in \Delta$ there are equalities
\begin{equation}\label{eq:4.7}
\langle \alpha , \log |w_\Delta \! w_P t| \rangle = \langle w_\Delta \! (\alpha), \log |w_P t| \rangle
= \langle w_\Delta \! (\alpha), \log |t| \rangle = \langle -w_\Delta \! (\alpha), \log |t^{-1}| \rangle ,
\end{equation}
where we used $t \in T^P$ in the second step. If $\alpha \in \Delta \setminus P^{op}$, then
$-w_\Delta \alpha \in \Delta \setminus P$ and \eqref{eq:4.7} is strictly positive because 
$t^{-1} \in T^{P+}$. Therefore $w_\Delta w_P (t) \in T^{P^{op}+}$. Now part (i) says that
\[
\psi_{w_\Delta w_P *}(\pi \otimes t) =  \psi_{w_\Delta w_P *}(\pi) \otimes w_\Delta w_P (t)
\]
is $W,\!P^{op}$-regular.\\
(iv) With the same method as for (iii) this can be reduced to part (ii).\\
In the cases (i) and (ii), Theorem \ref{thm:4.7}.a says that $\ind_{\mc H^P}^{\mc H} (\pi \otimes t)$ 
has a unique irreducible quotient, which moreover has the given shape. In the cases (iii) and (iv)
Theorem \ref{thm:4.7}.b says that $\ind_{\mc H^P}^{\mc H} (\pi \otimes t)$ has a unique irreducible
subrepresentation, namely the image of the indicated intertwining operator.
\end{proof}

\section{Comparison with Hermitian duals for reductive $p$-adic groups}
\label{sec:padic}

Consider a non-archimedean local field $F$ and a reductive group $G$ over $F$, connected in the
Zariski topology. We briefly call $G$ a reductive $p$-adic group. As is well-known, affine Hecke 
algebras often arise from Bernstein blocks in the category Rep$(G)$ of smooth complex 
$G$-representations. In such a situation there are two notions of a Hermitian dual: in Rep$(G)$ and 
in the module category of the appropriate affine Hecke algebra. We will show that in many such 
cases the two Hermitian duals agree. 

Let $M$ be a Levi factor of a parabolic subgroup $P$ of $G$, and let $\sigma \in \Irr (G)$ be
supercuspidal. The inertial equivalence class $\mf s = [M,\sigma]_G$ determines a Bernstein block
of $\Rep (G)^{\mf s}$ in $\Rep (G)$. By tensoring with a suitable unramified character we may assume
that $\sigma$ is unitary. Then its smooth Hermitian dual $\sigma^\he$ can be identified with $\sigma$
itself. It is not hard to show that the smooth Hermitian dual functor stabilizes every Bernstein
block in $\Rep (G)$, see \cite[Lemma 2.2]{SolQS}. 

One way to relate $\Rep (G)^{\mf s}$ to Hecke algebras stems from \cite{Hei}. Let $M^1 \subset M$
be the subgroup generated by all compact subgroups of $M$, and let $\sigma_1$ be an irreducible
subrepresentation of $\Res^M_{M^1} (\sigma)$. Let ind denote smooth induction with compact supports,
in contrast to Ind, which will denote smooth induction without any support condition. 

Then $\Pi_{\mf s_M} = \ind_{M^1}^M (\sigma_1 )$ is a progenerator of $\Rep (M)^{\mf s_M}$, where 
$\mf s_M = [M,\sigma]_M$. Moreover $\Pi_{\mf s} = I_P^G \Pi_{\mf s_M}$ is a progenerator of
$\Rep (G)^{\mf s}$, see \cite[\S VI.10.1]{Ren}. As worked out in \cite[Theorem 1.8.2.1]{Roc}, there
is an equivalence of categories
\begin{equation}\label{eq:5.1}
\begin{array}{ccc}
\Rep (G)^{\mf s} & \to & \Mod (\End_G (\Pi_{\mf s})^{op} ) \\
\pi & \mapsto & \Hom_G (\Pi_{\mf s}, \pi) 
\end{array}.
\end{equation}
It is known from \cite{SolEnd} that $\End_G (\Pi_{\mf s})^{op}$ is always very similar to an affine
Hecke algebra. To stay within the setting of the paper we assume in the remainder of this section:

\begin{cond}\label{cond:5.1}
$\mr{Res}^M_{M^1}(\sigma)$ is multiplicity-free and $\End_G (\Pi_{\mf s})^{op}$ is isomorphic to an 
affine Hecke algebra $\mc H$ with $q$-parameters in $\R_{\geq 1}$, via an isomorphism as in 
\cite{Hei} or \cite[\S 10.2]{SolEnd}. 
\end{cond}
This condition and \cite[\S 10]{SolEnd} imply that $\End_M (\Pi_{\mf s_M})^{op}$ is isomorphic to
the minimal parabolic subalgebra $\mc H^\emptyset \cong \C [X]$ of $\mc H$. 
By \cite[(IV.2.1.2)]{Ren} there are isomorphisms
\begin{equation}\label{eq:5.16}
\Pi_{\mf s}^\he = I_P^G \big( \ind_{M^1}^M (\sigma_1) \big)^\he \cong 
I_P^G \big( \ind_{M^1}^M (\sigma_1)^\he \big) \cong I_P^G \big( \mr{Ind}_{M^1}^M (\sigma_1^\he) \big) 
\cong I_P^G \big(\mr{Ind}_{M^1}^M (\sigma_1) \big) .
\end{equation}
In particular $\Pi_{\mf s}^\he$ contains $\Pi_{\mf s}$ as a dense submodule. The action of 
$\mc H^{op} = \End_G (\Pi_{\mf s})$ on $\Pi_{\mf s}$ extends to an action on $\Pi_{\mf s}^\he$
in the following way. Write $v \in \Pi_{\mf s}^\he$ as a limit of elements $v_n \in \Pi_{\mf s}$.
For $h \in \mc H^{op}$ we define $h \cdot v = \lim_{n \to \infty} h \cdot v_n$. 

Then $\Hom_G (\Pi_{\mf s}, \Pi_{\mf s}^\he)$ becomes an $\mc H \times \mc H^{op}$-module with action
\begin{equation}\label{eq:5.14}
h \cdot f \cdot h' = h' \circ f \circ h \qquad 
h,h' \in \mc H \cong \End_G (\Pi_{\mf s})^{op}, f : \Pi_{\mf s} \to \Pi_{\mf s}^\he.
\end{equation}

\begin{prop}\label{prop:5.2}
Assuming Condition \ref{cond:5.1}, $\Hom_G (\Pi_{\mf s}, \Pi_{\mf s}^\he)$ is isomorphic to
$\mc H^\he$ as $\mc H$-bimodules.
\end{prop}
\begin{proof}
First we consider the supercuspidal case, so with $\Pi_{\mf s_M} = \ind_{M^1}^M (\sigma_1)$.
With Frobenius reciprocity we compute
\begin{equation}\label{eq:5.3}
\Hom_M \big( \Pi_{\mf s_M}, \Pi_{\mf s_M}^\he \big) = \Hom_M \big( \ind_{M^1}^M (\sigma),  
\ind_{M^1}^M (\sigma)^\he \big) \cong \Hom_{M^1} \big( \sigma, \ind_{M^1}^M (\sigma_1)^\he \big) .
\end{equation}
Hermitian duals turn ind into Ind, so the right-hand side of \eqref{eq:5.3} is also
\begin{equation}\label{eq:5.4}
\Hom_{M^1} \big( \sigma_1, \mr{Ind}_{M^1}^M (\sigma_1^\he) \big) \cong 
\Hom_{M^1} \big( \sigma_1, \mr{Ind}_{M^1}^M (\sigma_1) \big) .
\end{equation}
An analogous computation (with ind instead of Ind) applies to $\End_M (\Pi_{\mf s_M})$. 
We note that the set 
\[
M^\sigma = \{m \in M : m \cdot \sigma_1 \cong \sigma_1\}
\] 
is a finite index subgroup of $M$ which does not depend on the choice of $\sigma_1$, see
\cite[\S 1.6]{Roc}. With the Mackey decomposition we obtain
\[
\mc A \cong \End_M (\Pi_{\mf s_M})^{op} \cong 
\bigoplus\nolimits_{m \in M^\sigma / M^1} \Hom_{M^1}(\sigma_1, m \cdot \sigma_1) ,
\]
so in particular $X = M^\sigma / M^1$. Similarly \eqref{eq:5.3} and \eqref{eq:5.4} become
\begin{equation}\label{eq:5.5}
\Hom_{M^1} \big( \sigma_1, \mr{Ind}_{M^1}^M (\sigma_1) \big) \cong
\prod\nolimits_{m \in M^\sigma / M^1} \Hom_{M^1}(\sigma_1, m \cdot \sigma_1) .
\end{equation}
This is isomorphic to $\mc A^\he \cong \prod_{x \in X} \C \{x\}$ as $\C [X]$-bimodule, so
\begin{equation}\label{eq:5.17}
\Hom_M (\Pi_{\mf s_M}, \Pi_{\mf s_M}^\he ) \cong \mc A^\he .
\end{equation}
In the non-supercuspidal case \eqref{eq:5.16} gives an $\mc H$-isomorphism
\begin{equation}\label{eq:5.6}
\Hom_G (\Pi_{\mf s}, \Pi_{\mf s}^\he) \cong 
\Hom_G \big( \Pi_{\mf s}, I_P^G (\ind_{M^1}^M (\sigma_1)^\he ) \big) .
\end{equation}
By \cite[Proposition 1.8.5.1]{Roc} the right-hand side is naturally isomorphic with
\[
\ind_{\mc A}^{\mc H} \: \Hom_M \big( \Pi_{\mf s_M}, \ind_{M^1}^M (\sigma_1)^\he \big) .
\]
From the supercuspidal case we know that this $\mc H$-module is isomorphic with \\
$\ind_{\mc A}^{\mc H} (\mc A^\he)$, which by Lemma \ref{lem:2.4} is isomorphic with $\mc H^\he$.
It remains to see that the resulting $\mc H$-module isomorphism
\begin{equation}\label{eq:5.7}
\phi : \mc H^\he \to \Hom_G (\Pi_{\mf s}, \Pi_{\mf s}^\he) 
\end{equation}
is an isomorphism of $\mc H$-bimodules. We already knew from \eqref{eq:2.2} that $\mc H \subset \mc H^\he$.
On the $\mc H$-submodule $\Hom_G (\Pi_{\mf s}, \Pi_{\mf s})$ of $\Hom_G (\Pi_{\mf s}, 
\Pi_{\mf s}^\he)$, the isomorphisms \eqref{eq:5.6}, \eqref{eq:5.3}, \eqref{eq:5.4} and \eqref{eq:5.5} 
become just the identity. Hence \eqref{eq:5.7} extends the given algebra isomorphism 
$\mc H \cong \End_G (\Pi_{\mf s})^{op}$.

By Proposition \ref{prop:2.5}, any element $h^+$ of $\mc H^\he$ admits a unique expression as
\begin{equation}\label{eq:5.18}
h^+ = \sum\nolimits_{w \in W} \sum\nolimits_{x \in X} c_{w,x} N_w \imath (\theta_x) \qquad c_{w,x} \in \C,
\end{equation}
with $\imath$ as in \eqref{eq:2.19}. The element $h_w^+ := \sum_{x \in X} c_{w,x} \theta_x$ belongs to 
$\mc A^\he$, which via \eqref{eq:5.17} and $I_P^G$ embeds naturally in 
$\Hom_G (\Pi_{\mf s}, \Pi_{\mf s}^\he )$. From \eqref{eq:5.5} we see that
\[
\phi \imath (h_w^+) = \phi \imath \big( \sum\nolimits_{x \in X} c_{w,x} \theta_x \big) = 
\sum\nolimits_{x \in X} c_{w,x} \phi \imath (\theta_x) .
\]
Using the $\mc H$-linearity of $\phi$ we find
\begin{equation}\label{eq:5.8}
\begin{aligned}
\phi (h^+) & = \phi \big( \sum_{w \in W} N_w \imath (h_w^+) \big) = 
\sum_{w \in W} N_w \cdot \phi \imath (h_w^+) =
\sum_{w \in W} N_w \cdot \sum_{x \in X} c_{w,x} \phi \imath (\theta_x) \\
& = \sum_{w \in W} \phi (N_w) \sum_{x \in X} c_{w,x} \phi \imath (\theta_x) = 
\sum_{w \in W} \sum_{x \in X} c_{w,x} \phi (N_w \imath (\theta_x)) .
\end{aligned}
\end{equation}
This shows that $\phi$ commutes with infinite sums of elements of $\mc H$ in the form \eqref{eq:5.18}.
Hence $\phi$ commutes with limits of sequences in $\mc H$ that converge in $\mc H^\he$. Let 
$(h_n^+)_{n=1}^\infty$ be a sequence in $\mc H$ with limit $h^+ \in \mc H^\he$. 
For any $h \in \mc H$, \eqref{eq:5.8} yields
\[
\phi (h^+ h) = \phi \big( \lim_{n \to \infty} h_n^+ h \big) = \lim_{n \to \infty} \phi (h_n^+ h) . 
\]
As $\phi |_{\mc H}$ is an algebra homomorphism, the right-hand side equals
\[
\lim_{n \to \infty} \phi (h_n^+) \phi (h) = 
\lim_{n \to \infty} \phi (h_n^+) \cdot h = \phi (h^+) \cdot h , 
\]
where the dot indicates the right action of $\mc H$ on $\Hom_G (\Pi_{\mf s}, \Pi_{\mf s}^\he)$
from \eqref{eq:5.14}.
\end{proof}

Proposition \ref{prop:5.2} serves as the starting point for the next result. It says that the
Hermitian dual functors in $\Rep (G)^{\mf s}$ and in $\Mod (\mc H)$ match via \eqref{eq:5.1}.

\begin{thm}\label{thm:5.3}
Let $\pi \in \Rep (G)^{\mf s}$ and assume Condition \ref{cond:5.1}. Then the $\mc H$-modules 
$\Hom_G (\Pi_{\mf s}, \pi^\he)$ and $\Hom_G (\Pi_{\mf s},\pi )^\he$ are isomorphic.
\end{thm}
\begin{proof}
First we consider the special case where $\Hom_G (\Pi_{\mf s},\pi)$ is a finitely generated
$\mc H$-module. Since $\mc H$ is Noetherian, there exists a projective resolution
\begin{equation}\label{eq:5.9}
\Hom_G (\Pi_{\mf s},\pi) \xleftarrow{d_0} \mc H \otimes F_0 \xleftarrow{d_1} \mc H \otimes F_1 
\leftarrow \cdots 
\end{equation}
where each $\mc H \otimes F_i$ is a free $\mc H$-module with a finite dimensional multiplicity
space $F_i$. We note that here $d_i (i>0)$ is determined entirely by the map 
\[
d_i |_{F_i} : F_i \cong \C 1 \otimes F_i \to \mc H \otimes F_{i-1} .
\]
The conjugate-transpose of \eqref{eq:5.9} is an injective resolution
\begin{equation}\label{eq:5.10}
\Hom_G (\Pi_{\mf s},\pi)^\he \xrightarrow{d_0^\he} \mc H^\he \otimes F_0^\he \xrightarrow{d_1^\he} 
\mc H^\he \otimes F_1^\he \rightarrow \cdots 
\end{equation}
Via the equivalence of categories \eqref{eq:5.1}, \eqref{eq:5.9} becomes a projective resolution
of $G$-representations
\begin{equation}\label{eq:5.11}
\pi \xleftarrow{d_0} \Pi_{\mf s} \otimes F_0 \xleftarrow{d_1} \Pi_{\mf s} \otimes F_1 \leftarrow \cdots
\end{equation}
Here $d_i (i>0)$ is determined by the map $d_i |_{F_i} : F_i \to \End_G (\Pi_{\mf s}) \otimes F_{i-1}$
given by $d_i |_{F_i}$ above composed with $\mc H \cong \End_G (\Pi_{\mf s})^{op}$.
The conjugate-transpose of \eqref{eq:5.11} is the injective resolution
\begin{equation}\label{eq:5.15}
\pi^\he \xrightarrow{d_0^\he} \Pi_{\mf s}^\he \otimes F_0^\he \xrightarrow{d_1^\he} 
\Pi_{\mf s}^\he \otimes F_1^\he \rightarrow \cdots 
\end{equation}
in $\Rep (G)^{\mf s}$. Again applying \eqref{eq:5.1}, we obtain an injective resolution
\begin{multline}\label{eq:5.12}
\Hom_G (\Pi_{\mf s},\pi^\he) \xrightarrow{\Hom_G(\Pi_{\mf s},d_0^\he)} \Hom_G (\Pi_{\mf s}, \Pi_{\mf s}^\he) 
\otimes F_0^\he \\ \xrightarrow{\Hom_G (\Pi_{\mf s},d_1^\he)} 
\Hom_G (\Pi_{\mf s}, \Pi_{\mf s}^\he) \otimes F_1^\he \rightarrow \cdots 
\end{multline}
The maps $\Hom_G (\Pi_{\mf s},d_i^\he)$ are still determined by $d_i |_{F_i}$. To \eqref{eq:5.12}
we can apply Proposition \ref{prop:5.2}, that yields
\begin{equation}\label{eq:5.13}
\Hom_G (\Pi_{\mf s},\pi^\he ) \rightarrow \mc H^\he \otimes F_0^\he \rightarrow
\mc H^\he \otimes F_1^\he \rightarrow \cdots 
\end{equation}
The maps in this sequence (except the leftmost) are induced by $d_i |_{F_i}$, so they equal the maps
in \eqref{eq:5.10}. We deduce isomorphisms of $\mc H$-modules
\[
\Hom_G (\Pi_{\mf s},\pi)^\he \cong \ker \big( \mc H^\he \otimes F_0^\he \rightarrow
\mc H^\he \otimes F_1^\he \big) \cong \Hom_G (\Pi_{\mf s},\pi^\he ) .
\] 
Now we consider the general case. Write $\Hom_G (\Pi_{\mf s},\pi)$ as the direct limit of its finitely
generated submodules $\Hom_G (\Pi_{\mf s},\pi_i)$, where $i$ runs through some index set. Then 
$\pi \cong \varinjlim \pi_i$ and $\pi^\he \cong \varprojlim \pi_i^\he$, which gives
\[
\Hom_G (\Pi_{\mf s},\pi^\he) \cong \varprojlim \Hom_G (\Pi_{\mf s},\pi_i^\he) .
\]
By \eqref{eq:5.13}, the right-hand side is isomorphic to
\[
\varprojlim \Hom_G (\Pi_{\mf s},\pi_i )^\he \cong \big( \varinjlim \Hom_G (\Pi_{\mf s},\pi_i) \big)^\he
\cong \Hom_G (\Pi_{\mf s},\pi)^\he . \qedhere
\]
\end{proof}

Theorem \ref{thm:5.3} implies among others that the equivalence of categories \eqref{eq:5.1} sends
Hermitian representations (i.e. $\pi^\he \cong \pi$) to Hermitian representations.

\section{Equivalent characterizations of genericity}
\label{sec:equiv}

Recall that a $G$-representation $\pi$ is called generic if there exist 
\begin{itemize}
\item a nondegenerate character $\xi$ of the unipotent radical $U$ of a minimal parabolic 
subgroup $B$ of $G$,
\item a nonzero $U$-homomorphism from $\pi$ to $\xi$.
\end{itemize}
More precisely, $\pi$ is $(U,\xi)$-generic if 
\[
\Hom_G \big( \pi, \Ind_U^G (\xi) \big) \cong \Hom_U (\pi,\xi) \quad \text{is nonzero.}
\]
Following \cite{BuHe}, we say that $\pi$ is simply generic if $\dim \Hom_U (\pi,xi) = 1$.
For irreducible representations of quasi-split reductive $p$-adic groups, simple genericity
is equivalent to genericity \cite{Shal,Rod}.

From \cite[(2.1.1)]{BuHe} we know that $\ind_U^G (\xi)^\he \cong \mr{Ind}_U^G (\xi)$, and hence
there is a natural conjugate-linear isomorphism
\begin{equation}\label{eq:6.1}
\Hom_G \big( \pi, \Ind_U^G (\xi) \big) \cong \Hom_G (\ind_U^G (\xi), \pi^\he) .
\end{equation}
By \eqref{eq:5.1} the right-hand side of \eqref{eq:6.1} is isomorphic to
\begin{equation}\label{eq:6.2}
\Hom_{\End_G (\Pi_{\mf s})} \big( \Hom_G (\Pi_{\mf s}, \ind_U^G (\xi)), 
\Hom_G (\Pi_{\mf s},\pi^\he) \big)
\end{equation}
Recall that $\mf s = [M,\sigma]_G$ and notice that $\xi$ restricts to a nondegenerate character 
of $U \cap M$.
Suppose now that $\pi \in \Rep (G)^{\mf s}$ and assume Condition \ref{cond:5.1}.
By Theorem \ref{thm:5.3}, \eqref{eq:6.2} is isomorphic to
\[
\Hom_{\mc H} \big( \Hom_G (\Pi_{\mf s}, \ind_U^G (\xi)), \Hom_G (\Pi_{\mf s},\pi)^\he  \big). 
\]
The explicit structure of affine Hecke algebras (in comparison with reductive $p$-adic groups)
will make it possible to characterize genericity of $\pi$ much more simply than above.
Recall that the Steinberg representation of $\mr{St} : \mc H (W,q^\lambda) \to \C$ is given
by $\mr{St}(N_s) = - q_s^{-1/2}$ for every simple reflection $s \in W$.

Part (a) of the next result stems from \cite{BuHe}. We include it because it compares well with
part (b), which generalizes \cite{ChSa,MiPa}.

\begin{thm}\label{thm:6.1}
\enuma{
\item Suppose that $\sigma$ is not $(U \cap M, \xi)$-generic. Then \\
$\Hom_G (\Pi_{\mf s},\ind_U^G (\xi)) = 0$ and no object of $\Rep (G)^{\mf s}$ is $(U,\xi)$-generic.
\item Assume that $\sigma$ is simply $(U \cap M, \xi)$-generic and assume Condition \ref{cond:5.1}. 
Then 
\[
\Hom_G (\Pi_{\mf s}, \ind_U^G (\xi)) \cong \ind_{\mc H (W,q^\lambda)}^{\mc H} (\mr{St})
\quad \text{as} \quad \mc H\text{-modules.}
\]
}
\end{thm}
\begin{proof}
Recall that $\Pi_{\mf s } = I_P^G (\ind_{M^1}^M (\sigma_1))$ for an irreducible subrepresentation
$\sigma_1$ of $\Res^M_{M^1} (\sigma_1)$.
By \cite[Theorem 2.2]{BuHe} there is a natural isomorphism $J^G_{\overline P} \ind_U^G (\xi)
\cong \ind_{U \cap M}^M (\xi)$. With Bernstein's second adjointness we find
\begin{equation}\label{eq:6.3}
\begin{aligned}
\Hom_G (\Pi_{\mf s}, \ind_U^G (\xi)) & = \Hom_G \big( I_P^G (\ind_{M^1}^M (\sigma_1)), 
\ind_U^G (\xi) \big) \\ & \cong \Hom_M \big( \ind_{M^1}^M (\sigma_1), \ind_{U \cap M}^M (\xi) \big).
\end{aligned}
\end{equation}
(a) The non-genericity of $\sigma$ implies, by \cite[Corollary 4.2]{BuHe}, that the component
of $\ind_{U \cap M}^M (\xi)$ in $\Rep (M)^{\mf s_M}$ is zero. Hence \eqref{eq:6.3} reduces to 0.
That and \eqref{eq:6.1}--\eqref{eq:6.2} imply the second claim of part (a).\\
(b) Since $U \cap M \subset M^1$, there is a unique irreducible constituent $\sigma_1$ of 
$\Res^M_{M^1} (\sigma)$ such that $\Hom_{U \cap M} (\sigma_1,\xi)$ is nonzero. Then
\begin{equation}\label{eq:6.6}
\dim_\C \Hom_{U \cap M} (\sigma_1, \xi) = 1 \text{ and } \sigma_1 \text{ appears with multiplicity
1 in } \Res^M_{M^1} (\sigma).
\end{equation}
Now \cite[Proposition 9.2]{BuHe} says that
\[
\ind_{U \cap M}^{M} (\xi) \cong \ind_{M^1}^M (\sigma_1) .
\]
By Condition \ref{cond:5.1} $\End_M (\mf s_M) \cong \mc A \cong \C [X]$.
Hence \eqref{eq:6.3} is isomorphic to
\begin{equation}\label{eq:6.5}
\End_M (\ind_{M^1}^M (\sigma_1)) \cong \End_M (\Pi_{\mf s_M}) \cong \C [X] 
\end{equation}
as modules for $\End_M (\mf s_M) \cong \C [X]$. That makes our setup is almost the same 
as in \cite[\S 2]{SolQS}, which means that the arguments from there remain valid in our setting. 
Then \cite[Lemma 3.1]{SolQS} proves the theorem.
\end{proof}

The Hermitian duals of the representations in Theorem \ref{thm:6.1} can be described in various
ways, which has interesting consequences.

\begin{prop}\label{prop:6.3}
Assume Condition \ref{cond:5.1} and that $\sigma$ is simply $(U \cap M,\xi)$-generic.
\enuma{
\item There are isomorphisms of $\mc H$-modules
\[
\Hom_G (\Pi_{\mf s}, \mr{Ind}_U^G (\xi)) \cong \Hom_G (\Pi_{\mf s}, \ind_U^G (\xi))^\he \cong
\ind_{\mc H(W,q^\lambda)}^{\mc H} (\mr{St})^\he 
\cong \mc H^\he \underset{\mc H (W,q^\lambda)}{\otimes} \mr{St}.
\]
\item For $\pi \in \Rep (G)^{\mf s}$ there are isomorphisms of complex vector spaces
\begin{align*}
\Hom_U (\pi,\xi) \cong \Hom_G (\pi, \mr{Ind}_U^G (\xi)) & \cong 
\Hom_{\mc H} \big( \Hom_G (\Pi_{\mf s},\pi), \mc H^\he 
\underset{\mc H (W,q^\lambda)}{\otimes} \mr{St} \big) \\ 
& \cong \Hom_{\mc H (W,q^\lambda)} (\Hom_G (\Pi_{\mf s},\pi), \mr{St}) .
\end{align*}
\item A representation $\pi \in \Rep (G)^{\mf s}$ is $(U,\xi)$-generic if and only if the 
$\mc H (W,q^\lambda)$-module $\Hom_G (\Pi_{\mf s},\pi)$ contains $\mr{St}$.
}
\end{prop}
\begin{proof}
(a) The first isomorphism is an instance of Theorem \ref{thm:5.3}, the second is Theorem \ref{thm:6.1}
and the third is Lemma \ref{lem:6.2}.a for $V = \mr{St} = \mr{St}^\he$. \\
(b) The first isomorphism is a version of Frobenius reciprocity and the third is Lemma \ref{lem:6.2}.b.
The second isomorphism follows from the equivalence of categories \eqref{eq:5.1} and part (a).\\
(c) This follows from part (b) and the semisimplicity of $\mc H (W,q^\lambda)$.
\end{proof}

We note that Proposition \ref{prop:6.3}.c is almost the same as \cite[Theorem 3.4]{SolQS}. The latter was
only proven for representations of finite length, and did not include Proposition \ref{prop:6.3}.a,b.

\section{Generic representations of affine Hecke algebras}
\label{sec:generic}

Let us return to a more general setting, where $\mc H$ is an affine Hecke algebra with $q$-parameters in
$\R_{\geq 1}$, but $\mc H$ does not have to come from a reductive $p$-adic group. Motivated by
Proposition \ref{prop:6.3}, we put
 
\begin{defn}\label{def:6.4}
An $\mc H$-module $V$ is generic if and only if $\Res^{\mc H}_{\mc H (W,q^\lambda)} V$ contains St. 
\end{defn}

From this definition the multiplicity one property of generic constituents of standard modules, as in 
\cite{Shal,Rod} for quasi-split reductive $p$-adic groups, follows quickly.

\begin{lem}\label{lem:6.5}
Let $P \subset \Delta$ and let $V \in \mr{Mod}(\mc H^P)$. 
\enuma{
\item $V$ is generic if and only if $\ind_{\mc H^P}^{\mc H} V$ is generic.
\item Suppose $V$ is irreducible and generic. Then
$\dim \Hom_{\mc H (W,q^\lambda)} (\ind_{\mc H^P}^{\mc H} V, \mr{St}) = 1$
and $\ind_{\mc H^P}^{\mc H} V$ has a unique generic irreducible subquotient.
This constituent appears with multiplicity one in $\ind_{\mc H^P}^{\mc H} V$.
}
\end{lem}
\begin{proof}
(a) The Bernstein presentation of $\mc H$ shows that
\[
\Res^{\mc H}_{\mc H (W,q^\lambda)} \big( \ind_{\mc H^P}^{\mc H} V \big) = \ind_{\mc H 
(W_P ,q^\lambda)}^{\mc H (W,q^\lambda)} \big( \Res^{\mc H^P}_{\mc H (W_P,q^\lambda)} V \big) .
\]
Then by Frobenius reciprocity 
\begin{equation}\label{eq:6.4}
\Hom_{\mc H (W,q^\lambda)} (\ind_{\mc H^P}^{\mc H} V, \mr{St}) \cong 
\Hom_{\mc H (W_P,q^\lambda)} (V, \mr{St}) .
\end{equation}
(b) By \cite[Lemma 3.5]{SolQS}
\[
\dim \Hom_{\mc H (W,q^\lambda)} (\ind_{\mc H^P}^{\mc H} V, \mr{St}) \leq 1 ,
\]
and by part (a) it is not 0. In view of the semisimplicity of $\mc H (W,q^\lambda)$,
this shows that $\ind_{\mc H^P}^{\mc H} V$ contains a unique copy of St, say $\C v$.
It follows that $\ind_{\mc H^P}^{\mc H} V$ has a generic irreducible subquotient,
which appears with multiplicity one. It can be described as $\mc H v$ modulo the
maximal submodule that does not contain $v$.
\end{proof}

We will investigate when the generic constituent of $\ind_{\mc H^P}^{\mc H} V$ is
a quotient or a subrepresentation. That is related to the generalized injectivity
conjecture \cite{CaSh} about representations of reductive $p$-adic groups. It asserts
that the generic irreducible subquotient of a generic standard representation is always a 
subrepresentation.

The last isomorphism in Proposition \ref{prop:6.3}.b provides a useful alternative (but
equivalent) condition for genericity of an $\mc H$-representation $\pi$, namely that 
\[
\Hom_{\mc H} (\pi, \mc H^\he \otimes_{\mc H (W,q^\lambda)} \mr{St}) \text{ is nonzero.}
\]
By Proposition \ref{prop:2.5} there are isomorphisms of $\mc A$-modules
\begin{equation}
\begin{aligned}
\Res_{\mc A}^{\mc H} (\mc H^\he \otimes_{\mc H (W,q^\lambda)} \mr{St}) & \cong
\Res_{\mc A}^{\mc H} (\ind_{\mc A}^{\mc H}(\mc A^\he) \otimes_{\mc H (W,q^\lambda)} \mr{St} ) \\
& \cong \mc A^\he \otimes_\C \mc H (W,q^\lambda) \otimes_{\mc H (W,q^\lambda)} \mr{St} 
\cong \mc A^\he .
\end{aligned}
\end{equation}
The composed isomorphism is given explicitly by
\[
\begin{array}{ccc}
\mc A^\he & \to & \mc H^\he \otimes_{\mc H (W,q^\lambda)} \mr{St} \\
h & \mapsto & \imath (h) \otimes 1
\end{array},
\]
where $\imath$ is as in \eqref{eq:2.19}. For $t \in T$ we write 
\[
f_t = \imath \big( \sum\nolimits_{x \in X} x(t)^{-1} \theta_x \big) \in \mc H^\he .
\]
In these terms 
\begin{equation}\label{eq:7.7}
\Hom_{\mc H} \big( \ind_{\mc A}^{\mc H}(t), 
\mc H^\he \otimes_{\mc H (W,q^\lambda)} \mr{St} \big)
\cong \Hom_{\mc A} \big( t, \mc A^\he \big) \cong \C f_t .
\end{equation}

\begin{lem}\label{lem:7.1}
Let $P \subset \Delta$ and let $\pi$ be an irreducible generic $\mc H^P$-representation.
\enuma{
\item For any $t \in T^P$: $\dim \Hom_{\mc H} \big( \ind_{\mc H^P}^{\mc H} (\pi \otimes t),  
\mc H^\he \otimes_{\mc H (W,q^\lambda)} \mr{St}  \big) = 1$.
\item Let $s \in T$ be a weight of $\pi$. Then the image of $\ind_{\mc H^P}^{\mc H} 
(\pi \otimes t)$ in $\mc H^\he \otimes_{\mc H (W,q^\lambda)} \mr{St}$, via a nonzero
$\mc H$-homomorphism as in part (a), is generated as $\mc H$-module by $f_{st} \otimes 1$.
}
\end{lem}
\begin{proof}
(a) This follows from Lemmas \ref{lem:6.2}.b and \ref{lem:6.5}.b. \\
(b) For any weight $s$ of $\pi$, Frobenius reciprocity yields a nonzero (and hence surjective) 
$\mc H^P$-homomorphism $\ind_{\mc A}^{\mc H^P} (s) \to \pi$. Hence the unique (up to scalars)
$\mc H^P$-homomorphism $\pi \to \mc H^\he \otimes_{\mc H (W,q^\lambda)} \mr{St}$ can be inflated 
to a nonzero $\mc H$-homomorphism $\phi_s : \ind_{\mc A}^{\mc H^P} (s) \to 
\mc H^\he \otimes_{\mc H (W,q^\lambda)} \mr{St}$. Similarly the unique (up to scalars) homomorphism 
$\pi \otimes t \to \mc H^\he \otimes_{\mc H (W,q^\lambda)} \mr{St}$ inflates to 
\[
\phi_{st} : \ind_{\mc A}^{\mc H^P} (st) \to \mc H^\he \otimes_{\mc H (W,q^\lambda)} \mr{St} .
\]
By \eqref{eq:7.7} we may assume that
\[
\phi_{st}(h \otimes 1) = h (f_{st} \otimes 1).
\]
Let $v_s$ be the image of $N_e \otimes 1 \in \ind_{\mc A}^{\mc H^P}(s)$ in $\pi$, by irreducibility 
it generates $\pi$. Then $\phi_{st}$ factors through a homomorphism $\pi \otimes t \to \mc H^\he 
\otimes_{\mc H (W,q^\lambda)} \mr{St}$ that sends $v_s$ to $f_{st} \otimes 1$.
Frobenius reciprocity produces a homomorphism
\begin{equation}\label{eq:7.8}
\begin{array}{cccc}
\mf{Wh} (P,\pi \otimes t,v_s) : &
\ind_{\mc H^P}^{\mc H} (\pi \otimes t) & \to & \mc H^\he \otimes_{\mc H (W,q^\lambda)} \mr{St} \\
& \ind_{\mc H^P}^{\mc H} (\pi \otimes t)(h) v_s & \mapsto & h (f_{st} \otimes 1) 
\end{array}.
\end{equation}
Clearly the image of $\mf{Wh} (P,\pi \otimes t,v_s)$ is generated by $f_{st} \otimes 1$.
\end{proof}

The $\mf{Wh} (P,\pi \otimes t,v_s)$ with $t \in T^P$ form an algebraic family of $\mc H$-homomorphisms, 
in the sense that for any fixed $h \in \mc H$ the image $\mf{Wh}(P,\pi \otimes t,v_s)(h \otimes v_s)$ 
is a regular function of $t$.

To analyse the unique irreducible generic subquotient of $\ind_{\mc H^P}^{\mc H}(\pi \otimes t)$,
which exists by Lemma \ref{lem:6.5}.b, we use a version of Shahidi's local constant \cite{Sha1}. 
Let us set up the intertwining operators $I(\gamma,\pi \otimes t)$ more systematically.
For $\alpha \in \Delta$ we define $\imath_{s_\alpha}^\circ \in \C (T)^W \otimes_{\mc O (T)^W} \mc H$ by
\begin{equation}\label{eq:7.9}
1 + q^{\lambda (\alpha) /2 } N_{s_\alpha} = (1 + \imath_{s_\alpha}^\circ)
\frac{\theta_\alpha q^{(\lambda (\alpha) + \lambda^* (\alpha))/2} -1}{\theta_\alpha -1}
\frac{\theta_\alpha q^{(\lambda (\alpha) - \lambda^* (\alpha))/2} + 1}{\theta_\alpha +1} .
\end{equation}
By \cite[\S 5]{Lus-Gr}, $s_\alpha \mapsto \imath_{s_\alpha}^\circ$ extends to a group homomorphism
\begin{equation}\label{eq:7.11}
W \to \big( \C (T)^W \otimes_{\mc O (T)^W} \mc H \big)^\times : w \mapsto \imath_w^\circ .
\end{equation}
These elements provide an algebra isomorphism
\[
\begin{array}{ccc}
\C (T) \rtimes W & \to & \C (T)^W \otimes_{\mc O (T)^W} \mc H \\
f w & \mapsto & f \imath_w^\circ 
\end{array}.
\]
Assume that $P, w(P) \subset \Delta$. It follows from \eqref{eq:7.9} that
\[
\imath_w^\circ N_{s_\alpha} \imath^\circ_{w^{-1}} = N_{s_{w(\alpha)}} \text{ for all } \alpha \in P. 
\]
Hence the isomorphism $\psi_w : \mc H^P \to \mc H^{w(P)}$ equals conjugation by $\imath^\circ_w$ in
$\C (T)^W \otimes_{\mc O (T)^W} \mc H$. Let $t \in T^P$ and $(\pi,V_\pi) \in \Mod (\mc H^P)$.
Consider the bijection
\[
\begin{array}{ccc}
\C (T)^W \otimes_{\mc O (T)^W} \mc H \otimes_{\mc H^P} V_\pi & \to &
\C (T)^W \otimes_{\mc O (T)^W} \mc H \otimes_{\mc H^{w(P)}} V_\pi \\
h \otimes v & \mapsto & h \imath^\circ_{w^{-1}} \otimes v 
\end{array},
\]
where $V_\pi$ is endowed with the representation $\pi \otimes t$. For $t$ in a Zariski-open dense 
subset of $T^P$, this defines an intertwining operator
\[
I(w,P,\pi,t) : \ind_{\mc H^P}^{\mc H}(\pi \otimes t) \to 
\ind_{\mc H^{w(P)}}^{\mc H}(\psi_w (\pi \otimes t)),
\]
which is rational as function of $t \in T^P$. By \eqref{eq:7.11}, whenever 
$\gamma w (P) \subset \Delta$:
\begin{equation}\label{eq:7.14}
I(\gamma, w(P), \psi_w (\pi), w(t)) \circ I(w,P,\pi,t) = I (\gamma w, P, \pi, t) .
\end{equation}
The Whittaker functionals from \eqref{eq:7.8} satisfy
\[
\mf{Wh}(w(P),\psi_w (\pi \otimes t),v_s) \circ I(w,P,\pi , t) \in \Hom_{\mc H} \big( 
\ind_{\mc H^P}^{\mc H}(\pi \otimes t), \mc H^\he \otimes_{\mc H (W,q^\lambda)} \mr{St} \big) ,
\]
at least away from the poles of $I(w,P,\pi,t)$. By Lemma \ref{lem:7.1} there exists a unique
$C(w,P,\pi,t) \in \C \cup \{\infty\}$ such that
\begin{equation}\label{eq:7.10}
\mf{Wh}(P,\pi \otimes t,v_s) = 
C(w,P,\pi,t) \mf{Wh}(w(P),\psi_w (\pi \otimes t),v_s) \circ I(w,P,\pi ,t) .    
\end{equation}
For $\gamma = w^{-1}$, \eqref{eq:7.14} implies 
\begin{equation}\label{eq:7.2}
C(w^{-1}, w(P), \psi_w (t), w(t)) = C(w,P,\pi,t)^{-1} . 
\end{equation}
Notice that $C(w,P,\pi,?)$ is a rational function on $T^P$, because all the other terms in
\eqref{eq:7.10} are so. This $C(w,P,\pi,t)$ is the local constant for affine Hecke algebras,
analogous to \cite{Sha1}. We note that this is based on the normalized intertwining operators
that involve $\imath_w^\circ$. An even stronger analogy with \cite{Sha1} can be obtained 
by using intertwining operators based on the elements $\imath_w$ from \cite[(4.1)]{Opd-Sp}.
The normalization of the intertwining operators does not affect the poles of the local 
constants.

\begin{lem}\label{lem:7.2}
Let $\pi \in \Irr (\mc H^P)$ be generic. Then $C(w,P,\pi,?)$ is regular at $t \in T^P$
if and only if $\ker I (w,P,\pi,t) \subset \ker \mf{Wh}(P,\pi \otimes t, v_s)$.
\end{lem}
\begin{proof}
Suppose that $C(w,P,\pi,t) = \infty$. Then 
\[
\mf{Wh}(w(P),\psi_w (\pi \otimes t),v_s) \circ I(w,P,\pi ,t) = 0,
\]
so $\text{im}\, I(w,P,\pi,t)$ is not generic. In the short exact sequence
\[
0 \to \ker I(w,P,\pi,t) \to \ind_{\mc H^P}^{\mc H}(\pi \otimes t) \to 
\text{im}\, I(w,P,\pi,t) \to 0,
\]
the middle term is generic by Lemma \ref{lem:7.1}. Hence $\mf{Wh}(w,P,\pi \otimes t,v_s)$
does not vanish on $\ker I(w,P,\pi,t)$, and the latter is generic.

When $C(w,P,\pi,t) \in \C^\times$, the equivalence is clear from \eqref{eq:7.10}.

Suppose that $C(w,P,\pi,t) = 0$. Then $I(w,P,\pi,t)$ has a pole at $t \in T^P$, caused
by a pole of $\imath^\circ_{w^{-1}}$. Let $f$ be a holomorphic function on a neighborhood $U$
of $t$ in $T^P$, such that $f(t') I(w,P,\pi,t')$ is regular and nonzero on $U$.
We can replace \eqref{eq:7.10} by
\begin{equation}\label{eq:7.1}
\mf{Wh}(P,\pi \otimes t,v_s) = 
C \, \mf{Wh}(w(P),\psi_w (\pi \otimes t),v_s) \circ f (t) I(w,P,\pi, t) ,
\end{equation}
for some $C \in \C \cup \{\infty\}$. As $\mf{Wh}(P,\pi \otimes t,v_s)$ is nonzero,
so is $C$. Then \eqref{eq:7.1} shows that $\mf{Wh}(w(P),\psi_w (\pi \otimes t), v_s)$
is nonzero on $\mr{im}\, f(t) I(w,P,\pi, t)$. Therefore $C \neq \infty$, and we conclude that 
$\mf{Wh}(P,\pi \otimes t,v_s)$ factors through $f(t) I(w,P,\pi, t)$. In particular 
\[
\ker I(w,P,\pi,t) \:\subset\: \ker f(t) I(w,P,\pi,t) \:\subset\: 
\ker \mf{Wh}(P,\pi \otimes t, v_s). \qedhere
\]
\end{proof}

Next we prove some cases of the generalized injectivity conjecture for affine Hecke algebras.
For that we need $q_s \geq 1$ for all $s \in S_\mr{aff}$, otherwise the statement would be 
false (the Steinberg representation of $\mc H$ would violate it).

\begin{thm}\label{thm:7.3}
Assume that $\lambda (\alpha) \geq \lambda^* (\alpha) \geq 0$ for all $\alpha \in R$.
Let $t \in T^P$ and let $\pi \in \Irr (\mc H^P)$ be generic, tempered and anti-tempered.
\enuma{
\item When $t^{-1} \in T^{P+}$, the unique generic irreducible subquotient of 
$\ind_{\mc H^P}^{\mc H}(\pi \otimes t)$ is 
\[
L(P,\pi \otimes t) = \ind_{\mc H^P}^{\mc H}(\pi \otimes t) \big/
\ker I (w_\Delta w_P, P, \pi \otimes t) .
\]
\item When $t \in T^{P+}$, the unique generic irreducible subquotient of $\ind_{\mc H^P}^{\mc H}
(\pi \otimes t)$ is its unique irreducible subrepresentation $\tilde L (P,\pi, t)$.
}
\end{thm}
\begin{proof}
(a) In view of Proposition \ref{prop:4.6}.ii and Lemma \ref{lem:7.2}, it suffices to check that 
$C(w_\Delta w_P,P,\pi,?)$ does not have a pole at $t$. By \eqref{eq:7.2}, it is 
equivalent to show that 
\begin{equation}\label{eq:7.3}
C(w_P w_\Delta,P^{op},\psi_{w_\Delta w_P}(\pi), w_\Delta w_P (t)) \text{ is nonzero.}
\end{equation}
Suppose that \eqref{eq:7.3} is zero. By \eqref{eq:7.10},
$I(w_P w_\Delta,P^{op},\psi_{w_\Delta w_P}(\pi), w_\Delta w_P (t))$ has a pole at $t$.
It is known from \cite[Theorem 4.33.i]{Opd-Sp} that every such $t$ is a zero of
\[
\big( \alpha (w_\Delta w_P (rt)) - q^{(\lambda (\alpha) + \lambda^* (\alpha))/2} \big)
\big( \alpha (w_\Delta w_P (rt)) + q^{(\lambda (\alpha) - \lambda^* (\alpha))/2} \big),
\]
for some weight $r$ of $\pi$ and $-\alpha, w_P w_\Delta (\alpha) \in R^+$. Equivalently,
such a $t$ satisfies
\begin{equation}\label{eq:7.12}
\big( \beta (rt) - q^{(\lambda (\alpha) + \lambda^* (\alpha))/2} \big)
\big( \beta (rt) + q^{(\lambda (\alpha) - \lambda^* (\alpha))/2} \big) = 0
\end{equation}
for some weight $r$ of $\pi$ and $\beta, -w_\Delta w_P (\beta) \in R^+$. The eligible
$\beta$ are precisely the roots in $R^+ \setminus R_P^+$. From $t^{-1} \in T^{P+}$ we
get $\beta (t) \in (0,1)$ for all $\beta \in R^+ \setminus R_P^+$. By the temperedness
and anti-temperedness of $\pi$, $|r| = 1$. Hence $|\beta (rt)| < 1$, which in
combination with $\lambda (\alpha) \geq \lambda^* (\alpha) \geq 0$ implies that 
\eqref{eq:7.12} never holds under our assumptions. Hence \eqref{eq:7.3} is valid.\\
(b) In the proof of Proposition \ref{prop:4.6}.iii we checked that 
$\psi_{w_\Delta w_P}(\pi)$ is again tempered and anti-tempered, and that 
$w_\Delta w_P (t)^{-1} \in T^{P^{op}+}$. By part (a)
\[
\ind_{\mc H^{P^{op}}}^{\mc H} \psi_{w_\Delta w_P} (\pi \otimes t) / \ker
I (w_P w_\Delta, \psi_{w_\Delta w_P} (\pi \otimes t)) \cong \mr{im}\, 
I (w_P w_\Delta, \psi_{w_\Delta w_P} (\pi \otimes t)) 
\]
is generic. This is an irreducible subrepresentation of 
$\ind_{\mc H^P}^{\mc H}(\pi \otimes t)$, and from Proposition \ref{prop:4.6}.iv we
know that there exists only one such subquotient.
\end{proof}

One interesting application of Theorem \ref{thm:7.3} concerns the induction of suitable
characters of $\mc A$. 

\begin{prop}\label{prop:7.4}
Suppose that $\lambda (\alpha) \geq \lambda^* (\alpha) \geq 0$ for all $\alpha \in R$, and let $t \in T$.
\enuma{
\item Suppose $|t^{-1}|$ lies in the closure of $T^{\emptyset +}$. Then the unique generic irreducible 
constituent of $\ind_{\mc A}^{\mc H}(t)$ is a quotient.
\item Suppose $|t|$ lies in the closure of $T^{\emptyset +}$. Then the unique generic irreducible 
constituent of $\ind_{\mc A}^{\mc H}(t)$ is a subrepresentation.
}
\end{prop}
\begin{proof}
Write $P = \{ \alpha \in \Delta : |\alpha (t)| = 1 \}$. Then $\ind_{\mc A}^{\mc H^P} ( t |t|^{-1})$
is an $\mc H^P$-representation all whose $\mc A$-weights belong to $\Hom_\Z (X,S^1)$, so it is both
tempered and anti-tempered. By \cite[Proposition 3.1.4.a]{SolAHA} $\ind_{\mc A}^{\mc H^P} ( t |t|^{-1})$
is completely reducible, say a direct sum of irreducible subrepresentations $\rho_i$. As $|t| \in T^P$:
\begin{equation}\label{eq:7.13}
\ind_{\mc A}^{\mc H^P}(t) = \ind_{\mc A}^{\mc H^P} ( t |t|^{-1}) \otimes |t| =
\bigoplus\nolimits_i \rho_i \otimes |t| ,
\end{equation}
where all the weights of $\rho_i \otimes |t|$ have absolute value $|t|$. Lemma \ref{lem:7.1}.a guarantees 
that exactly one of the $\rho_i \otimes |t|$ is generic, say it is $\rho_1 \otimes |t|$. Now
the arguments for the two parts diverge:\\
(a) By Theorem \ref{thm:7.3}.a $\ind_{\mc H^P}^{\mc H}(\rho_1 \otimes |t|)$ has a generic irreducible
quotient $\pi$. In view of \eqref{eq:7.13} and the transitivity of parabolic induction, $\pi$ is also
a quotient of $\ind_{\mc A}^{\mc H}(t)$. \\
(b) By Theorem \ref{thm:7.3}.b $\ind_{\mc H^P}^{\mc H}(\rho_1 \otimes |t|)$ has a generic irreducible
subrepresentation $\pi$. By \eqref{eq:7.13} $\pi$ is also a subrepresentation of $\ind_{\mc A}^{\mc H}(t)$.
\end{proof}

Proposition \ref{prop:7.4} is a Hecke algebra version of the generalized injectivity
conjecture for inductions of supercuspidal representations of reductive $p$-adic groups
\cite[Theorem 1]{CaSh}. For the current status of the generalized injectivity conjecture we refer
to \cite{Dij}. Possibly our Hecke algebra interpretation can be useful to establish more cases.
However, the generalized injectivity conjecture does not hold for all affine Hecke algebras with
parameters $\lambda (\alpha) \geq \lambda^* (\alpha) \geq 0$, as witnessed by the next example.

\begin{ex}\label{ex:7}
Consider the based root datum 
\[
\mc R = \big( \Z^2, B_2, \Z^2, C_2, \{\alpha = e_1 - e_2, \beta = e_2\} \big).
\]
We take $\lambda (\beta) = \lambda^* (\beta) = 1$ and $\lambda (\alpha) = \lambda^* (\alpha) = 6$.
(It would also work with any number $>2$ instead of 6.) 
The algebra $\mc H = \mc H (\mc R,\lambda,\lambda^*,q)$ has a one-dimensional discrete series
representation $\delta$ given by:
\begin{itemize}
\item $\mc A$ acts on $\delta$ via the weight $t_\delta = (q^{-5},q)$,
\item $\delta (N_{s_\alpha}) = - q^{-3}$ and $\delta (N_{s_\beta}) = q^{1/2}$.
\end{itemize}
As $W_{t_\delta} = \{e\}$, $\delta$ is the unique irreducible $\mc H$-representation with 
$\mc A$-weight $t_\delta$. Notice that $\delta$ is not generic. We will show that $\delta$ occurs
as the unique irreducible subrepresentation of a standard module.

Let $\mr{St}_\alpha$ be the Steinberg representation of $\mc H^{\{\alpha\}}$, a generic discrete
series representation with $\mc A$-weight $(q^{-3},q^3)$. For $(q^2,q^2) \in T^{\{\alpha\}+}$,
$\mr{St}_\alpha \otimes (q^2,q^2)$ has the unique $\mc A$-weight $t = (q^{-1},q^5)$.
By \cite[Lemma 3.3]{SolHecke} the standard module 
\[
\pi = \ind_{\mc H^{\{\alpha\}}}^{\mc H} \big( \mr{St}_\alpha \otimes (q^2,q^2) \big)
\]
has set of $\mc A$-weights 
\[
W^{\{\alpha\}} t = \{t, s_\beta t, s_\alpha s_\beta t, s_\beta s_\alpha s_\beta t = t_\delta \} .
\]
By considering the invertibility of intertwining operators, one sees that $t, s_\beta t$ and
$s_\alpha s_\beta t$ sit together in one irreducible subquotient of $\pi$. That representation
involves the maximal weight $t$ of $\pi$, so by Theorem \ref{thm:4.3}.a it is the Langlands 
quotient $L \big( \{\alpha\},\mr{St}_\alpha, (q^2,q^2) \big)$. Further $t_\delta$ is not a weight of  
$L \big( \{\alpha\},\mr{St}_\alpha, (q^2,q^2) \big)$, because the only irreducible 
$\mc H$-representation with that property is $\delta$. Thus $\pi$ is reducible and has a 
subquotient $\delta$, which is in fact a subrepresentation because it equals the kernel of $\pi \to
L \big( \{\alpha\},\mr{St}_\alpha, (q^2,q^2) \big)$. Lemma \ref{lem:6.5}.b says that $\pi$ has a 
unique generic irreducible constituent and it is not $\delta$, so it must be the Langlands quotient
$L \big( \{\alpha\},\mr{St}_\alpha, (q^2,q^2) \big)$.
\end{ex}

A weaker version of the generalized injectivity conjecture is known as the standard module conjecture
\cite{CaSh}. It asserts that the Langlands quotient of a generic standard representation is generic
if and only if that standard module is irreducible. This has been proven for all quasi-split
reductive $p$-adic groups \cite{HeMu,HeOp}. Using Section \ref{sec:equiv}, one
can deduce the standard module conjecture for all affine Hecke algebras whose parameters 
come from a generic Bernstein component for a quasi-split reductive $p$-adic group. 

Nevertheless, our above counterexample to the generalized injectivity conjecture is also
a counterexample to the standard module conjecture for affine Hecke algebras with arbitrary
parameters $\geq 1$.

\section{Affine Hecke algebras extended with finite groups}
\label{sec:extend}

For comparison with reductive $p$-adic groups it is useful to consider a slightly larger class of
algebras. Let $\Gamma$ be a finite group acting on the based root datum $\mc R = (X,R,Y,R^\vee,\Delta)$.
Then $\Gamma$ acts on $W$ by 
\[
\gamma (s_\alpha) = s_{\gamma \alpha} = \gamma s_\alpha \gamma^{-1}  \qquad \alpha \in R , 
\]
where the conjugation takes place in $\mr{Aut}_\Z (X)$. This yields a semidirect product
$(X \rtimes W) \rtimes \Gamma$. We also suppose that $\Gamma$ acts on $\mc A \cong \C [X] \cong
\mc O (T)$, such that the induced action on $\mc A^\times / \C^\times \cong X$ recovers the given
action on $X$. Thus $\Gamma$ acts on $T = \Irr (\mc A)$, but it need not fix a point of $T$. 

Further we assume that the label functions $\lambda, \lambda^* : R \to \R$ are $\Gamma$-invariant.
Then $\Gamma$ acts on $\mc H$ by the algebra automorphisms
\[
\gamma (N_w \theta_x) = N_{\gamma (w)} \gamma (\theta_x) \qquad \gamma \in \Gamma, w \in W, x \in X.
\]
The algebra $\mc H \rtimes \Gamma = \Gamma \ltimes \mc H$ has an Iwahori--Matsumoto basis
$\{ N_w : w \in (X \rtimes W) \rtimes \Gamma \}$ and a Bernstein basis
$\{ \theta_x N_w : x \in X, w \in W \rtimes \Gamma \}$. The length function of $X \rtimes W$ extends
naturally to $X \rtimes (W \rtimes \Gamma)$, and then it becomes zero on $\Gamma$.
The involution * of $\mc H$ extends to $\mc H \rtimes \Gamma$ by $N_\gamma^* = N_{\gamma^{-1}}$ for
$\gamma \in \Gamma$. We extend the trace $\tau$ of $\mc H$ to $\mc H \rtimes \Gamma$ by defining
$\tau |_{\mc H N_\gamma} = 0$ for all $\gamma \in \Gamma \setminus \{e\}$.

More generally we can involve a 2-cocycle $\natural : \Gamma^2 \to \C^\times$. It gives rise to a
twisted group algebra $\C [\Gamma,\natural]$, with multiplication rules
\[
N_\gamma \cdot N_{\gamma'} = \natural (\gamma,\gamma') N_{\gamma \gamma'} .
\]
From that we can build the twisted affine Hecke algebra $\mc H \rtimes \C [\Gamma,\natural]$,
which is like $\mc H \rtimes \Gamma$, only with $\C [\Gamma]$ replaced by $\C [\Gamma,\natural]$.
These twisted algebras can also be constructed with central idempotents. Namely, let
\begin{equation}\label{eq:8.19}
1 \to Z_\Gamma^+ \to \Gamma^+ \to \Gamma \to 1
\end{equation}
be a finite central extension, such that the pullback of $\natural$ to $\Gamma^+$ splits. Then
there exist a minimal idempotent $p_\natural \in \C [Z_\Gamma^+]$ and an algebra isomorphism
\[
\phi_\natural : p_\natural \C [\Gamma^+] \to \C[\Gamma,\natural] .
\]
For each lift $\gamma^+ \in \Gamma^+$ of $\gamma \in \Gamma$, $\phi_\natural 
(p_\natural N_{\gamma^+}) \in \C^\times N_\gamma$. Then $p_\natural$ is also a central idempotent
in $\mc H \rtimes \Gamma^+$ and
\begin{equation}\label{eq:8.20}
\mc H \rtimes \C [\Gamma,\natural] \cong p_\natural (\mc H \rtimes \Gamma^+) =
(\mc H \rtimes \Gamma^+) p_\natural .
\end{equation}
Since $p_\natural$ comes from a unitary character of $Z_\Gamma^+$, it stable under the natural 
*-operation on $\C [\Gamma^+]$. We define the * on $\C [\Gamma,\natural]$ by
\[
\phi (p_\natural N_{\gamma^+})^* = \phi (p_\natural N_{\gamma^+}^*) = 
\phi (p_\natural N_{\gamma^+}^{-1}) .  
\]
In combination with the * on $\mc H$, this endows \eqref{eq:8.20} with a *-operation. We define the 
trace on $\mc H \rtimes \C [\Gamma,\natural]$ just like for $\mc H \rtimes \Gamma$.

To deal with parabolic induction, we use a subgroup $\Gamma_P \subset \Gamma$ for each 
$P \subset \Delta$. 

\begin{cond}\label{cond:8.1}
\begin{enumerate}[(i)]
\item $\Gamma_P \subset \Gamma_Q$ whenever $P \subset Q$,
\item the action of $\Gamma_P$ on $T$ stabilizes $P, T_P$ and $T^P$ (and hence normalizes $W_P$),
\item $\Gamma_P$ acts on $T^P$ by multiplication with elements of the finite group $T_P \cap T^P$,
\item if $\gamma \in W \rtimes \Gamma, P \subset \Delta$ and $\gamma (P) \subset \Delta$, then 
$\gamma \Gamma_P \gamma^{-1} = \Gamma_{\gamma (P)}$,
\item $\natural$ is trivial on $\Gamma_\emptyset^2$.
\end{enumerate}
\end{cond}
Let $\Gamma_P^+$ be the inverse image of $\Gamma_P$ in $\Gamma^+$ for the map \eqref{eq:8.19}, then
Condition \ref{cond:8.1} holds for $\Gamma^+$ as well.
We say that $\mc H^P \rtimes \C [\Gamma_P,\natural]$ is a parabolic subalgebra of 
$\mc H \rtimes \C[\Gamma,\natural]$. Notice that $\mc H^\Delta \rtimes \C[\Gamma_\Delta,\natural]
= \mc H \rtimes \C[\Gamma_\Delta,\natural]$ (but $\Gamma_\Delta$ need not be the whole of $\Gamma$).
By Condition \ref{cond:8.1}.iii, $\Gamma_\emptyset$ acts trivially on $T^\emptyset = T$, so
$\Gamma_\emptyset$ acts trivially on $\mc H$. Together with Condition \ref{cond:8.1}.v that implies
\begin{equation}\label{eq:8.22}
\mc H^\emptyset \rtimes \C[\Gamma_\emptyset,\natural] = \mc A \otimes \C[\Gamma_\emptyset] . 
\end{equation}
By Condition \ref{cond:8.1}.ii $\Gamma_P$ stabilizes $P, X \cap \Q P$ and $X \cap (P^\vee)^\perp$.
Then Condition \ref{cond:8.1}.iii says that $\Gamma_P$ fixes $\Q X \cap (P^\vee)^\perp \cong
\Q X / \Q P$ pointwise. 
Let us write the action of $\gamma \in \Gamma$ on $\mc A \cong \C [X]$ as
\[
\gamma (\theta_x) = z_\gamma (x) \theta_{\gamma(x)} \qquad \text{where } z_\gamma \in T.
\]
For $t \in T^P = \Hom (X / X \cap \Q P, \C^\times), w \in W_P, x \in X$ we compute
\begin{equation}\label{eq:8.9}
\begin{aligned}
& \gamma (\psi_t (\theta_x N_w)) = \gamma (t(x) \theta_x N_w) = 
t(x) z_\gamma (x) \theta_{\gamma (x)} N_{\gamma (w)} ,\\
& \psi_t (\gamma (\theta_x N_w)) = \psi_t (z_\gamma (x) \theta_{\gamma (x)} N_{\gamma (w)}) =
t(\gamma (x)) z_\gamma (x) \theta_{\gamma (x)} N_{\gamma (w)} .
\end{aligned}
\end{equation}
These two lines are equal because $\gamma$ fixes $X / X \cap \Q P$ pointwise, so that $t(\gamma (x))
= t(x)$. Thus $\psi_t \in \mr{Aut} (\mc H^P)$ is $\Gamma_P$-equivariant and extends to an automorphism of 
$\mc H^P \rtimes \C[\Gamma_P,\natural]$. That enables us to define $\pi \otimes t$ for 
$\pi \in \Mod (\mc H^P \rtimes \C[\Gamma_P,\natural])$ and $t \in T^P$.\\

Assuming all the above, we will check what is needed to make the results from the previous sections
valid for $\mc H \rtimes \Gamma$ and for $\mc H \rtimes \C[\Gamma,\natural]$. To ease the notation
we will sometimes write things down for $\mc H \rtimes \Gamma$ and then indicate how they can be
generalized to $\mc H \rtimes \C[\Gamma,\natural]$. Of course this means that everywhere we should 
also replace $\mc H^P$ by $\mc H^P \rtimes \C[\Gamma_P,\natural]$ and $\mc H (W,q^\lambda)$ by 
$\mc H (W,q^\lambda) \rtimes \C[\Gamma,\natural]$.
The role of $W^P$ can be played by $\Gamma^P W^P$, where $\Gamma^P \subset \Gamma$ is a set of 
representatives for $\Gamma / \Gamma_P$. Notice that $\Gamma^P W^P$ is a set of shortest length 
representatives for $W \rtimes \Gamma / W_P \rtimes \Gamma_P$, because 
\[
\Gamma^P W^P (P) \subset \Gamma^P (R^+) = R^+ .
\]
In Lemma \ref{lem:1.1} we replace $h = N_w \theta_x$ by $N_{\gamma w} \theta_x$ and $N_{w^{-1}}$ by
$N_{\gamma w}^* = N_{w^{-1}} N_\gamma^* \in \C^\times N_{w^{-1}} N_{\gamma^{-1}}$. For $\gamma \in 
\Gamma \setminus \Gamma_P$ both $(h^*)^P_e$ and $(h^P_e)^{*_P}$ are zero, while for $\gamma \in \Gamma_P$
the calculations from the proof of Lemma \ref{lem:1.1} remain valid with an extra factor $N_\gamma^*$
at the right.

In Section \ref{sec:herm} and Theorem \ref{thm:3.1} there are few additional complications, almost 
everything holds just as well for $\mc H \rtimes \C[\Gamma,\natural]$. Only
in Lemma \ref{lem:2.6} we need to be careful: the same argument works for $\mc H \rtimes \Gamma^+$,
and from there we can restrict to $\mc H \rtimes \C[\Gamma,\natural]$ via \eqref{eq:8.20}.

To generalize Proposition \ref{prop:3.2} we need some preparations. Let $P,Q \subset \Delta$ and let
$D^{P,Q}$ be a set of shortest length representatives for $W_P \Gamma_P \backslash W \Gamma / 
W_Q \Gamma_Q$. In contrast with $W^{P,Q}$, $D^{P,Q}$ need not be unique. Like in \eqref{eq:3.9}, every
$d \in D^{P,Q}$ gives rise to an algebra isomorphism
\[
\begin{array}{cccc}
\psi_d : & \mc H^{d^{-1}(P) \cap Q} \rtimes (d^{-1} \Gamma_P d \cap \Gamma_Q) & \to &
\mc H^{P \cap d(Q)} \rtimes (\Gamma_P \cap d \Gamma_Q d^{-1}) \\
& \theta_x N_w N_\gamma & \mapsto & \theta_{d(x)} N_{d w d^{-1}} N_{d \gamma d^{-1}}
\end{array}.
\]
Kilmoyer's result \eqref{eq:3.10} can be generalized as follows:

\begin{lem}\label{lem:8.5}
Let $d \in D^{P,Q}$.
\enuma{
\item $d^{-1} W_P d \cap W_Q$ equals $W_{d^{-1}(P) \cap Q}$.
\item $d^{-1} (W_P \rtimes \Gamma_P) d \cap (W_Q \rtimes \Gamma_Q)$ equals 
$W_{d^{-1}(P) \cap Q} \rtimes (d^{-1} \Gamma_P d \cap \Gamma_Q)$.
}
\end{lem}
\begin{proof}
Write $d = \gamma_d w_d$ with $\gamma_d \in \Gamma$ and $w_d \in W$. 
For $\alpha \in P$ we have $\ell (d s_\alpha) < \ell (d)$, so $d (\alpha) \in R^+$. 
As $\gamma_d (R^+) = R^+$, also $w_d (\alpha) \in R^+$.
For $\alpha \in P$ we have $\ell (s_\beta d) < \ell (d)$, so 
$R^+ \ni d^{-1}(\beta) = w_d^{-1} \gamma_d^{-1} (\beta)$. Thus $w_d (Q) \subset R^+$ and
$w_d^{-1}( \gamma_d^{-1} P) \subset R^+$, which means that $w_d \in W^{\gamma_d^{-1}(P),Q}$.\\
(a) We compute
\[
d^{-1} W_P d \cap W_Q = w_d^{-1} \gamma_d^{-1} W_P \gamma_d w_d \cap W_Q =
w_d^{-1} W_{\gamma_d^{-1}(P)} w_d \cap W_Q .
\]
By \eqref{eq:3.10} the right hand side equals 
$W_{w_d^{-1} \gamma_d^{-1}(P) \cap Q} = W_{d^{-1}(P) \cap Q}$.\\
(b) First we note that by Condition \ref{cond:8.1}.iv 
\begin{align*}
d^{-1} (W_P \rtimes \Gamma_P) d & = w_d^{-1} \gamma_d^{-1} (W_P \rtimes \Gamma_P) \gamma_d w_d =
w_d^{-1} (\gamma_d^{-1} W_{P} \gamma_d \rtimes \gamma_d^{-1} \Gamma_P \gamma_d ) w_d \\
& = w_d^{-1} (W_{\gamma_d^{-1}(P)} \rtimes \Gamma_{\gamma_d^{-1}(P)} ) w_d .
\end{align*}
Consider $w_1 \in W_Q, \gamma_1 \in \Gamma_Q, w_2 \in W_{\gamma_d^{-1}(P)}, \gamma_2 \in 
\Gamma_{\gamma_d^{-1}(P)}$ such that 
\begin{equation}\label{eq:8.13}
w_1 \gamma_1 = w_d^{-1} w_2 \gamma_2 w_d . 
\end{equation}
Via the isomorphism $W \rtimes \Gamma / W \cong \Gamma$ we see that 
$\gamma_1 = \gamma_2 \in \Gamma_Q \cap \Gamma_{\gamma_d^{-1}(P)}$. Then 
\begin{align*}
& \gamma_2 w_d \gamma^{-1}(Q) = \gamma_2 w_d (Q) \subset \gamma_2 (R^+) = R^+ ,\\
& (\gamma_2 w_d \gamma_1)^{-1} (P) = \gamma_1 w_d^{-1} \gamma_2^{-1} (P) = \gamma_1 w_d^{-1} (P)
\subset \gamma_1 (R^+) = R^+ ,
\end{align*}
so $\gamma_2 w_d \gamma_1^{-1} = \gamma_1 w_d \gamma^{-1} \in W^{P,Q}$. Now 
\[
w_1 = w_d^{-1} w_2 (\gamma_2 w_d \gamma_1^{-1}) \in W_Q \cap w_d^{-1} w_d^{-1} W_P D^{P,Q} ,
\]
which by \cite[Lemma 2.7.2]{Car} is only possible when $\gamma_2 w_d \gamma^{-1} = w_d$. Hence
\begin{equation}\label{eq:8.14}
w_1 = w_d^{-1} w_2 w_d \in W_Q \cap w_d^{-1} W_{\gamma_d^{-1}(P)} w_d ,
\end{equation}
and from \eqref{eq:3.10} we know that the right hand side equals $W_{Q \cap d^{-1}(P)}$. From
\eqref{eq:8.13} and \eqref{eq:8.14} we obtain $\gamma_1 = w_d^{-1} \gamma_2 w_d$, so 
\begin{multline*}
(W_Q \rtimes \Gamma_Q) \cap w_d^{-1} ( W_{\gamma_d^{-1}(P)} \rtimes \Gamma_{\gamma_d^{-1}(P)}) w_d
= W_{Q \cap \gamma_d^{-1}(P)} \rtimes (\Gamma_Q \cap w_d^{-1} \Gamma_{\gamma_d^{-1}(P)} w_d) \\
= W_{Q \cap d^{-1}(P)} \rtimes (\Gamma_Q \cap w_d \gamma_d^{-1} \Gamma_P \gamma_d w_d) =
W_{Q \cap d^{-1}(P)} \rtimes (\Gamma_Q \cap d^{-1} \Gamma_P d) . \qedhere
\end{multline*}
\end{proof}

Let $(\pi,V_\pi) \in \Mod (\mc H^Q \rtimes \Gamma_Q)$. Analogous to \eqref{eq:3.11}, 
$\ind_{\mc H^Q \rtimes \Gamma_Q}^{\mc H \rtimes \Gamma} (V_\pi)$ has linear subspaces
\[
(\Res^{\mc H \rtimes \Gamma}_{\mc H^P \rtimes \Gamma_P} \ind_{\mc H^Q \rtimes \Gamma_Q}^{
\mc H \rtimes \Gamma} )_{\leq d} (V_\pi) = \bigoplus_{d' \in D^{P,Q}, d' \leq d}
\mc H (W_P \Gamma_P d' W_Q \Gamma_Q) \mc A \otimes_{\mc H^Q \rtimes \Gamma_Q} V_\pi .
\]
With Lemma \ref{lem:8.5} at hand, the proof of Proposition \ref{prop:3.2} becomes valid for
$\mc H \rtimes \Gamma$. The above also works for $\mc H \rtimes \C[\Gamma,\natural]$, that is only
a notational difference. The result is:

\begin{prop}\label{prop:8.6}
For each $d \in D^{P,Q}$, $(\Res^{\mc H \rtimes \Gamma}_{\mc H^P \rtimes \Gamma_P} \ind_{\mc H^Q 
\rtimes \Gamma_Q}^{\mc H \rtimes \Gamma} )_{\leq d} (V_\pi)$ is an $\mc H^P \rtimes \Gamma_P$-submodule
of $\ind_{\mc H^Q \rtimes \Gamma_Q}^{\mc H \rtimes \Gamma} (V_\pi)$. There is an isomorphism of
$\mc H^P \rtimes \Gamma_P$-modules
\begin{multline*}
\big( \Res^{\mc H \rtimes \Gamma}_{\mc H^P \rtimes \Gamma_P} \ind_{\mc H^Q \rtimes \Gamma_Q}^{\mc H 
\rtimes \Gamma} \big)_{\leq d} (V_\pi) \big/ \big( \Res^{\mc H \rtimes \Gamma}_{\mc H^P \rtimes \Gamma_P}
\ind_{\mc H^Q \rtimes \Gamma_Q}^{\mc H \rtimes \Gamma} \big)_{<d} (V_\pi) \cong \\
\ind_{\mc H^{P \cap d(Q)} \rtimes (\Gamma_P \cap d \Gamma_Q d^{-1})}^{\mc H^P \rtimes \Gamma_P} \big( 
\psi_{d *} \: \Res^{\mc H^Q \rtimes \Gamma_Q}_{\mc H^{d^{-1}(P) \cap Q} \rtimes (d^{-1} \Gamma_P d \cap 
\Gamma_Q)} (V_\pi) \big) .
\end{multline*}
Thus the functor $\Res^{\mc H \rtimes \Gamma}_{\mc H^P \rtimes \Gamma_P} \ind_{\mc H^Q 
\rtimes \Gamma_Q}^{\mc H \rtimes \Gamma}$ has a filtration with successive subquotients
\[
\ind_{\mc H^{P \cap d(Q)} \rtimes (\Gamma_P \cap d \Gamma_Q d^{-1})}^{\mc H^P \rtimes \Gamma_P} 
\circ \psi_{d *} \circ 
\Res^{\mc H^Q \rtimes \Gamma_Q}_{\mc H^{d^{-1}(P) \cap Q} \rtimes (d^{-1} \Gamma_P d \cap \Gamma_Q)},
\]
where $d$ runs through $D^{P,Q}$. 

The same holds with $\C[\Gamma,\natural]$ instead of $\C[\Gamma]$.
\end{prop}

In Section \ref{sec:Langlands} the elementary Lemmas \ref{lem:4.1} and \ref{lem:4.2} also hold for 
$\mc H \rtimes \Gamma$. However, the Langlands classification and its variations (Theorem 
\ref{thm:4.3} and Propositions \ref{prop:4.4}, \ref{prop:4.6}) are just not valid any more in this 
form. An extension of Theorem \ref{thm:4.3} to $\mc H \rtimes \Gamma$ was established in 
\cite[Corollary 2.2.5]{SolAHA}, but it is more involved.

The main issue with the Langlands classification for $\mc H \rtimes \Gamma$ is the uniqueness, as 
witnessed by the following example. Let $R = A_2, \Delta = \{\alpha,\beta\}, X = \Z R$ and let 
$\Gamma = \{e,\gamma\}$ with $\gamma$ the unique nontrivial automorphism of $(X,\Delta)$. The 
parabolic subalgebras of $\mc H \rtimes \Gamma$ are $\mc A, \mc H^{\{\alpha\}}, \mc H^{\{\beta\}}$ 
and $\mc H \rtimes \Gamma$. Pick a $t \in T^{\emptyset+}$ which is fixed by $\gamma$. 
Then $\ind_{\mc A}^{\mc H}(t)$ has a unique irreducible quotient but 
$\ind_{\mc A}^{\mc H \rtimes \Gamma}(t)$ has two inequivalent irreducible quotients.

Lemma \ref{lem:4.8} and its proof still work with our standard modifications. However, to generalize 
Lemma \ref{lem:4.5} and Theorem \ref{thm:4.7} we first have to extend the notion of $W,\!P$-regularity.
We say that an $\mc H^P \rtimes \C[\Gamma_P,\natural]$-representation $\pi$ is $W\Gamma,\!P$-regular 
if $wt \notin \mr{Wt}(\pi)$ for all $t \in \mr{Wt}(\pi)$ and all $w \in W_P \Gamma_P D_+^{P,P}$, where
\[
D_+^{P,P} = \{ d \in W\Gamma : d (P) \subset R^+, d^{-1}(P) \subset R^+, d \notin \Gamma_P \} .
\]

\begin{lem}\label{lem:8.7}
Let $P,Q \subset \Delta$ and $\gamma \in W \Gamma$ such that $\gamma (P) = Q$ and let\\ 
$\pi \in \Irr (\mc H^P \rtimes \C[\Gamma_P,\natural])$ be $W\Gamma,\!P$-regular. 
\enuma{
\item The representation $\pi$ appears with multiplicity one in\\ 
$\Res^{\mc H \rtimes \C[\Gamma,\natural]}_{\mc H^P \rtimes \C[\Gamma_P,\natural]} 
\ind_{\mc H^Q \rtimes \C[\Gamma_Q,\natural]}^{\mc H\rtimes \C[\Gamma,\natural]} (\psi_{\gamma *} \pi)$, 
as a direct summand.
\item $\dim \Hom_{\mc H \rtimes \C[\Gamma,\natural]} \big( \ind_{\mc H^P \rtimes \C[\Gamma_P,\natural]}^{
\mc H \rtimes \C[\Gamma,\natural]} (\pi), \ind_{\mc H^Q \rtimes \C[\Gamma_Q,,\natural]}^{\mc H \rtimes 
\C[\Gamma,\natural]} (\psi_{\gamma *} \pi) \big) = 1$.
}
\end{lem}
\begin{proof}
(a) Since $\gamma (P) \subset R^+$ and $\gamma^{-1}(Q) \subset R^+$, $\gamma^{-1}$ has minimal
length in\\ $W_P \Gamma_P \gamma^{-1} W_Q \Gamma_Q$. Hence we may choose $D^{P,Q}$ so that it contains
$\gamma^{-1}$. We follow the proof of Lemma \ref{lem:4.8} with $d' = \gamma^{-1}$. Instead of
\eqref{eq:4.10} we find $w_1, w_2 \in W_P \rtimes \Gamma_P$ and $t \in \mr{Wt}(\pi)$ such that
$w_2^{-1} w_1 d \gamma t \in \mr{Wt}(\pi)$. The $W \Gamma,\!P$-regularity of $\pi$ says that
\begin{equation}\label{eq:8.15} 
w_3 w_2^{-1} w_1 d \gamma \notin W_P \Gamma_P D_+^{P,P} \qquad \text{for all } w_3 \in W_P \Gamma_P .
\end{equation}
Notice that $d \gamma (P) = d(Q) \subset R^+$, which means that $d \gamma \in \Gamma W^P$. Suppose
that $d \gamma$ does not have minimal length in $W_P d \gamma$. There exists $\alpha \in P$ with 
$\gamma^{-1} d^{-1} (\alpha) \in -R^+$. Then $\gamma^{-1} d^{-1} s_\alpha (\alpha) \in R^+$ and
\[
\ell (s_\alpha d \gamma ) = \ell (\gamma^{-1} d^{-1} s_\alpha) < \ell (\gamma^{-1} d^{-1}) =
\ell (d \gamma) .
\]
As $\gamma^{-1}(Q) \subset R^+$, $d^{-1}(\alpha) \notin Q$ and $\alpha \notin d^{-1}(Q)$. That gives
\[
(s_\alpha d \gamma) (P) = s_\alpha d (Q) \subset s_\alpha (R^+ \setminus \{\alpha\}) \subset R^+.
\]
The reasoning can be applied to $s_\alpha d \gamma$. Repeating that if necessary, we find $w_4 \in 
W_P$ such that $w_4 d \gamma (P) \subset R^+$ and $w_4 d \gamma$ has minimal length in
$W_P d \gamma$. Thus $(w_4 d \gamma)^{-1}(P) \subset R^+$, $w_4 d \gamma \in D_+^{P,P} \cup \Gamma_P$ 
and $d \gamma \in W_P (D_+^{P,P} \cup \Gamma_P)$. Combining that with \eqref{eq:8.15}, we find 
$d \gamma \in W_P \Gamma_P = \Gamma_P W_P$. Also $d \gamma \in \Gamma W^P$, so in fact
$d \gamma \in \Gamma_P$. Then $\Gamma_P d = \Gamma_P \gamma^{-1}$, and using $d, \gamma^{-1} \in
D^{P,Q}$ we obtain $d = \gamma^{-1}$. From this point on, we can conclude in the same way as in the
proof of Lemma \ref{lem:4.5}.a.\\
(b) This can be shown exactly as in the proof of Lemma \ref{lem:4.5}.
\end{proof}

Lemma \ref{lem:8.7}.c yields a nonzero intertwining operator
\[
I(\gamma,P,\pi) : \ind_{\mc H^P \rtimes \C[\Gamma_P,\natural]}^{\mc H \rtimes \C[\Gamma,\natural]} 
(\pi) \to \ind_{\mc H^Q \rtimes \C[\Gamma_Q,\natural]}^{\mc H \rtimes \C[\Gamma,\natural]} 
(\psi_{\gamma *} \pi) ,
\]
unique up to scalars. With those operators Theorem \ref{thm:4.7} and its proof become valid for
$\mc H \rtimes \C[\Gamma,\natural]$. That provides $\mc H \rtimes \C[\Gamma,\natural]$ with a 
substitute for the uniqueness of Langlands quotients and Langlands representations (for $\mc H$). 
We warn that Proposition \ref{prop:4.6} fails for $\mc H \rtimes \Gamma$: 
$\pi \in \Irr (\mc H^P \rtimes \Gamma_P)$ tempered
and $t \in T^{P+}$ does not enforce $W\Gamma,\!P$-regularity of $\pi \otimes t$.\\

We may relax Condition \ref{cond:5.1} by replacing $\mc H$ with $\mc H \rtimes \C[\Gamma,\natural]$, 
let us call that Condition \ref{cond:5.1}'. The advantage is that it becomes valid for more Bernstein 
components of representations of $p$-adic groups. For instance, Condition \ref{cond:5.1}' applies to 
all smooth representations of classical groups \cite{Hei,AMS4} and in those cases 
Condition \ref{cond:8.1} follows from the same checks as in \cite[\S 5]{SolComp}.
Under Condition \ref{cond:5.1}', the indecomposability of $\Rep (M)^{\mf s_M}$ forces 
$\Mod (\mc H^\emptyset \rtimes \C[\Gamma_\emptyset,\natural])$ to be indecomposable. Hence 
the algebra \eqref{eq:8.22} is also decomposable, which forces $\Gamma_\emptyset = \{1\}$.
All the arguments and results in Section \ref{sec:padic} remain valid with 
$\mc H \rtimes \C[\Gamma,\natural]$ instead of $\mc H$, no further adjustments are necessary.

In the setting of Section \ref{sec:equiv}, Condition \ref{cond:5.1}' turns out to hold automatically
with trivial 2-cocycle, see Theorem \ref{thm:8.8}.
The crucial part of the proof of Theorem \ref{thm:6.1} is the reference to \cite[\S 2]{SolQS}.
Since that work was conceived for algebras of the form $\mc H \rtimes \Gamma$, 
Theorem \ref{thm:6.1} applies to all extended affine Hecke algebras that satisfy Condition 
\ref{cond:5.1}' with $\natural = 1$. More precisely, we extend the Steinberg 
representation of $\mc H (W,q^\lambda)$ to $\mc H (W,q^\lambda) \rtimes \Gamma$ by
\[
\mr{St}(N_w N_\gamma) = \mr{St}(N_w) \mr{det}_X (\gamma) \qquad w \in W, \gamma \in \Gamma ,
\]
where $\det_X$ means the determinant of the action of $\gamma$ on $X$. The more general version
of Theorem \ref{thm:6.1}.b says: 
\begin{equation}\label{eq:8.1}
\Hom_G (\Pi_{\mf s}, \ind_U^G (\xi)) \cong \ind_{\mc H (W,q^\lambda) \rtimes \Gamma}^{\mc H \rtimes 
\Gamma} (\mr{St}) \qquad \text{as } \mc H \rtimes \Gamma\text{-representations.}
\end{equation}
That and Proposition \ref{prop:6.3} prompt us to define:
\begin{equation}\label{eq:8.2}
\text{an } \mc H \rtimes \Gamma\text{-module } V \text{ is generic if and only if }
\Res^{\mc H \rtimes \Gamma}_{\mc H (W,q^\lambda) \rtimes \Gamma} (V) \text{ contains St.}
\end{equation}
With this definition, the part from Proposition \ref{prop:6.3} up to and including Lemma 
\ref{lem:7.2} generalizes readily to $\mc H \rtimes \Gamma$. For representations of 
$\mc H \rtimes \C[\Gamma,\natural]$ with $\natural$ nontrivial in $H^2 (\Gamma,\C^\times)$,
genericity is not defined. In such cases $\C[\Gamma,\natural]$ does not possess onedimensional
representations, so we do not have a good analogue of $\det_X$.\\

Let us discuss the relation between generic representations of $\mc H$ and of $\mc H \rtimes \Gamma$.
The definition of the Steinberg representation of $\mc H (W,q^\lambda)$ shows that 
$\gamma (\mr{St}) = \mr{St}$ for all $\gamma \in \Gamma$. It follows that
\begin{equation}\label{eq:8.3}
\text{the action of } \Gamma \text{ on Mod}(\mc H) \text{ preserves genericity.}
\end{equation}
Suppose now that $(\pi,V_\pi)$ is an irreducible generic $\mc H$-representation. By Lemma 
\ref{lem:6.5}.b there exists a unique (up to scalars) vector $v_{\mr{St}} \in V_\pi \setminus \{0\}$
on which $\mc H (W,q^\lambda)$ acts according to St. Let $\Gamma_\pi$ be the stabilizer (in 
$\Gamma$) of $\pi \in \Irr (\mc H)$. Schur's lemma says there exists a unique (up to scalars)
linear bijection 
\[
\pi (\gamma) : V_\pi \to V_\pi \text{ such that } \pi (\gamma (h)) = 
\pi (\gamma) \pi (h) \pi (\gamma)^{-1} \text{ for all } h \in \mc H.
\]
As $\gamma (\mr{St}) = \mr{St}$, $\pi (\gamma) v_{\mr{St}}$ must belong to $\C v_{\mr{St}}$.
We normalize $\pi (\gamma)$ by the condition $\pi (\gamma) v_{\mr{St}} = v_{\mr{St}}$. In this way
$(\pi,V_\pi)$ extends to a representation of $\mc H \rtimes \Gamma_\pi$. 

Clifford theory \cite[Appendix]{RaRa} tells us how any irreducible $\mc H \rtimes 
\Gamma$-representation containing $\pi$ can be constructed. Namely, let $(\rho,V_\rho) \in 
\Irr (\Gamma_\pi)$ and let $\mc H \rtimes \Gamma_\pi$ act on $V_\pi \otimes V_\rho$ by
\[
(h N_\gamma (v_1 \otimes v_2) = \pi (h N_\gamma) v_1 \otimes \rho (\gamma) v_2. 
\]
Then $\pi \rtimes \rho := \mr{ind}_{\mc H \rtimes \Gamma_\pi}^{\mc H \rtimes \Gamma} 
(V_\pi \otimes V_\rho)$ is irreducible and
\begin{equation}\label{eq:8.4}
\begin{array}{ccc}
\Irr (\Gamma_\pi) & \to & \Irr (\mc H \rtimes \Gamma) ,\\
\rho & \mapsto & \pi \rtimes \rho
\end{array} 
\end{equation}
is injection with as image 
\begin{equation}\label{eq:8.5}
\{ V \in \Irr (\mc H \rtimes \Gamma) : 
\Res_{\mc H}^{\mc H \rtimes \Gamma}(V) \text{ contains } V_\pi \}.
\end{equation}
As for the genericity of $\pi \rtimes \rho$:
\begin{align*}
\Hom_{\mc H (W,q^\lambda) \rtimes \Gamma} (\pi \rtimes \rho, \mr{St}) & 
= \Hom_{\mc H (W,q^\lambda) \rtimes \Gamma} (\mr{ind}_{\mc H \rtimes \Gamma_\pi}^{\mc H 
\rtimes \Gamma} (V_\pi \otimes V_\rho), \mr{St}) \\
& \cong \Hom_{\mc H (W,q^\lambda) \rtimes \Gamma_\pi} (\pi \otimes \rho, \mr{St}) .
\end{align*}
By Lemma \ref{lem:6.5}.b and because $\pi (\Gamma_\pi)$ fixes $v_{\mr{St}}$, the last 
expression is isomorphic with $\Hom_{\Gamma_\pi} (\rho, \det_X)$. We conclude that
\begin{equation}\label{eq:8.6}
\pi \rtimes \rho \text{ is } \left\{ \begin{array}{l}
\text{generic if } \rho = \det_X ,\\
\text{not generic otherwise.}
\end{array} \right.
\end{equation}
Conversely, consider an irreducible generic $\mc H \rtimes \Gamma$-representation
$(\sigma,V_\sigma)$. Let $\pi$ be an irreducible $\mc H$-subrepresentation of $\sigma$.
Then $\ind_{\mc H}^{\mc H \rtimes \Gamma}(\pi)$ surjects onto $\pi$, so every irreducible
$\mc H$-subquotient of $\sigma$ is isomorphic to $\gamma (\pi)$ for some 
$\gamma \in \Gamma$. As $\Res_{\mc H (W,q^\lambda)}^{\mc H \rtimes \Gamma} \sigma$
contains St, at least one of the $\gamma (\pi)$ is generic. In view of \eqref{eq:8.1},
actually all of them are generic, and in particular $\pi$. 
Then \eqref{eq:8.4}--\eqref{eq:8.6} show that 
\begin{equation}\label{eq:8.7}
\sigma \cong \pi \rtimes \mr{det}_X = 
\ind_{\mc H \rtimes \Gamma_\pi}^{\mc H \rtimes \Gamma} (\pi \otimes \mr{det}_X) .
\end{equation}
Next we generalize Theorem \ref{thm:7.3} and Proposition \ref{prop:7.4}. Since the
statements really change, we formulate them as new results.

\begin{thm}\label{thm:8:2}
Assume that $\lambda (\alpha) \geq \lambda^* (\alpha) \geq 0$ for all $\alpha \in R$.
Let $t \in T^P$ and let $\pi \in \Irr (\mc H^P \rtimes \Gamma_P)$ be tempered,
anti-tempered and generic. The unique generic irreducible constituent of 
$\ind_{\mc H^P \rtimes \Gamma_P}^{\mc H \rtimes \Gamma} (\pi \otimes t)$:
\enuma{
\item is a quotient when $t^{-1} \in T^{P+}$,
\item is a subrepresentation when $t \in T^{P+}$.
}
\end{thm}
\begin{proof}
With \eqref{eq:8.7} we can write 
\[
\pi \cong \ind_{\mc H^P \rtimes \Gamma_{P,\tau}}^{\mc H^P \rtimes \Gamma_P} 
(\tau \otimes \mr{det}_X),
\]
where $\tau$ is an irreducible generic $\mc H^P$-subrepresentation of $\pi$. Notice
that Wt$(\tau) \subset \mr{Wt}(\pi)$, so that $\tau$ is also tempered and anti-tempered.
By Condition \ref{cond:8.1}.ii,iii:
\begin{equation}\label{eq:8.8}
\begin{aligned}
& \ind_{\mc H^P \rtimes \Gamma_P}^{\mc H \rtimes \Gamma} (\pi \otimes t) \; \cong \; 
\ind_{\mc H^P \rtimes \Gamma_P}^{\mc H \rtimes \Gamma} \big( \ind_{\mc H^P \rtimes 
\Gamma_{P,\tau}}^{\mc H^P \rtimes \Gamma_P} (\tau \otimes \mr{det}_X) \otimes t \big) \\
& \cong \; \ind_{\mc H^P \rtimes \Gamma_P}^{\mc H \rtimes \Gamma} \big( \ind_{\mc H^P \rtimes 
\Gamma_{P,\tau}}^{\mc H^P \rtimes \Gamma_P} (\tau \otimes t \otimes \mr{det}_X) \big) \\
& \cong \; \ind_{\mc H^P \rtimes \Gamma_{P,\tau}}^{\mc H \rtimes \Gamma} (\tau \otimes t \otimes 
\mr{det}_X) \; \cong \; \ind_{\mc H \rtimes \Gamma_{P,\tau}}^{\mc H \rtimes \Gamma} 
\big( \ind_{\mc H^P}^{\mc H}(\tau \otimes t) \otimes \mr{det}_X \big) .
\end{aligned}
\end{equation}
(a) Theorem \ref{thm:7.3}.a says that the quotient $L(P,\tau \otimes t)$ of 
$\ind_{\mc H^P}^{\mc H}(\tau \otimes t)$ is generic. By the uniqueness in Theorem \ref{thm:4.3}.b,
$\Gamma_{L(P,\tau,t)} \cap \Gamma_P = \Gamma_{P,\tau \otimes t}$, which by the remarks following
\eqref{eq:8.1} equals $\Gamma_{P,\tau}$. From \eqref{eq:8.6} we know that
\begin{equation}\label{eq:8.10}
L(P,\tau,t) \rtimes \mr{det}_X = \ind_{\mc H \rtimes \Gamma_{P,\tau}}^{\mc H \rtimes \Gamma}
(L(P,\tau,t) \otimes \mr{det}_X) \quad \text{is generic.}
\end{equation}
Clearly \eqref{eq:8.10} is a quotient of the final term in \eqref{eq:8.8}.\\
(b) This is analogous to part (a), instead of $L(P,\tau,t)$ we use $\tilde L (P,\tau,t)$ from
Proposition \ref{prop:4.4}.
\end{proof}

\begin{prop}\label{prop:8.3}
Assume that $\lambda (\alpha) \geq \lambda^* (\alpha) \geq 0$ for all $\alpha \in R$, and let 
$t \in T$. The unique generic irreducible constituent of $\ind_{\mc A}^{\mc H \rtimes \Gamma}(t)$:
\enuma{
\item is a quotient if $|t^{-1}|$ lies in the closure of $T^{\emptyset +}$,
\item is a subrepresentation if $|t|$ lies in the closure of $T^{\emptyset +}$.
}
\end{prop}
\begin{proof}
(a) Proposition \ref{prop:7.4}.a says that $\ind_{\mc A}^{\mc H}(t)$ has a generic irreducible 
quotient, say $\pi$. By \eqref{eq:8.6}, $\pi \rtimes \det_X$ is the unique generic irreducible
$\mc H \rtimes \Gamma$-representation that contains $\pi$. With Frobenius reciprocity we compute
\begin{equation}\label{eq:8.11}
\Hom_{\mc H \rtimes \Gamma} \big( \ind_{\mc A}^{\mc H \rtimes \Gamma}(t), \pi \rtimes 
\mr{det}_X \big) \cong \Hom_{\mc A} (t, \pi \rtimes \mr{det}_X) = 
\Hom_{\mc A}(t, \ind_{\mc A \rtimes \Gamma_\pi}^{\mc A \rtimes \Gamma} \pi ) .
\end{equation}
The right hand side of \eqref{eq:8.11} contains
\[
\Hom_{\mc A} (t,\pi) \cong \Hom_{\mc H} (\ind_{\mc A}^{\mc H}(t), \pi) \neq 0 .
\]
Hence \eqref{eq:8.11} is nonzero, which means that $\pi \rtimes \det_X$ is a quotient of
$\ind_{\mc A}^{\mc H \rtimes \Gamma}(t)$.\\
(b) Proposition \ref{prop:7.4}.b yields a generic irreducible subrepresentation of 
$\ind_{\mc A}^{\mc H}(t)$, say $\sigma$. From \eqref{eq:8.6} is a generic irreducible
$\mc H \rtimes \Gamma$-representation. We compute
\begin{multline}\label{eq:8.12}
\Hom_{\mc H \rtimes \Gamma} \big( \sigma \rtimes \mr{det}_X, \ind_{\mc A}^{\mc H \rtimes 
\Gamma}(t) \big) = \Hom_{\mc H \rtimes \Gamma} \big( \ind_{\mc H \rtimes \Gamma_\sigma}^{\mc H 
\rtimes \Gamma} (\sigma \otimes \mr{det}_X), \ind_{\mc A}^{\mc H \rtimes \Gamma}(t) \big) \\
\cong \Hom_{\mc H \rtimes \Gamma_\sigma} \big( \sigma \otimes \mr{det}_X, \ind_{\mc A}^{\mc H 
\rtimes \Gamma}(t) \big) \supset \Hom_{\mc H \rtimes \Gamma_\sigma} \big( \sigma \otimes 
\mr{det}_X, \ind_{\mc A}^{\mc H \rtimes \Gamma_\sigma}(t) \big) .
\end{multline}
It follows from Clifford theory, in the version \cite[Theorem 11.2]{SolGHA}, that
\[
\ind_{\mc H}^{\mc H \rtimes \Gamma_\sigma} (\sigma) \cong 
\bigoplus\nolimits_{\rho \in \Irr (\Gamma_\sigma)} (\sigma \otimes \rho)^{\oplus \dim \rho} .
\]
Hence there exist injective $\mc H \rtimes \Gamma_\sigma$-homomorphisms
\[
\sigma \otimes \mr{det}_X \to \ind_{\mc H}^{\mc H \rtimes \Gamma_\sigma} (\sigma) \to
\ind_{\mc H}^{\mc H \rtimes \Gamma} (\ind_{\mc A}^{\mc H} t) = \ind_{\mc A}^{\mc H}(t) .
\]
Thus all terms in \eqref{eq:8.12} are nonzero, which by irreducibility means that
$\sigma \rtimes \det_X$ is a subrepresentation of $\ind_{\mc A}^{\mc H \rtimes \Gamma} (t)$.
\end{proof}

Let us summarize the findings of this section.

\begin{cor}\label{cor:8.4}
Suppose that $\Gamma$ is as at the start of Section \ref{sec:extend}, and assume in particular
Condition \ref{cond:8.1}. All the results of Sections \ref{sec:herm}--\ref{sec:padic} 
generalize to $\mc H \rtimes \C[\Gamma,\natural]$, except Theorem \ref{thm:4.3} and 
Propositions \ref{prop:4.4}, \ref{prop:4.6}. Sections \ref{sec:equiv} and \ref{sec:generic}
generalize to $\mc H \rtimes \Gamma$.
\end{cor}

\appendix

\renewcommand{\theequation}{A.\arabic{equation}}
\counterwithin*{equation}{section}

\section{Hecke algebras for simply generic Bernstein blocks}

Let $G$ be a reductive $p$-adic group and let $U$ be the unipotent radical of a minimal 
parabolic subgroup of $G$. Let $\xi$ be a nondegenerate character of $U$. Let $P = M U_P$ be a 
parabolic subgroup of $G$ containing $U$. Let $(\sigma,E)$ be an irreducible unitary supercuspidal 
$M$-representation which is simply $(U \cap M,\xi)$-generic, that is,
\[
\dim \Hom_{U \cap M}(\sigma,\xi) = 1 .
\]
We call $\Rep (G)^{\mf s}$ with $\mf s = [M,\sigma]$ a simply generic Bernstein block for $G$, 
because most irreducible representations in there are simply $(U,\xi)$-generic. In this appendix we 
show that $\Rep (G)^{\mf s}$ is equivalent to the module category of an extended affine Hecke algebra.

Let $(\sigma_1,E_1)$ be the unique irreducible generic $M^1$-subrepresentation from \eqref{eq:6.6}.
Recall from \cite[\S VI.10.1]{Ren} that $\Pi_{\mf s} = I_P^G (\ind^M_{M^1} (\sigma_1))$ is a
progenerator of $\Rep (G)^{\mf s}$. By abstract category theory \cite[Theorem 1.8.2.1]{Roc},
$\Rep (G)^{\mf s}$ is naturally equivalent with $\Mod (\End_G (\Pi_{\mf s})^{op})$.

\begin{thm}\label{thm:8.8}
In the above simply generic setting, $\End_G (\Pi_{\mf s})^{op}$ is isomorphic to an extended 
affine Hecke algebra $\mc H \rtimes \Gamma$ with $q$-parameters in $\R_{\geq 1}$. 
Conditions \ref{cond:5.1}' and \ref{cond:8.1} are satisfied.
\end{thm}
\begin{proof}
We follow \cite[\S 10]{SolEnd}, with some improvements that are made possible by the simple genericity 
of $\sigma$. Notice that \cite[Working hypothesis 10.2]{SolEnd} holds by \eqref{eq:6.6}. On the 
supercuspidal level with $\Pi_{\mf s_M} = \ind^M_{M^1}(\sigma_1)$, \cite[Lemma 10.1]{SolEnd} says that
\begin{equation}\label{eq:8.16}
\End_M (\Pi_{\mf s_M}) = \End_M (\Pi_{\mf s_M})^{op} = \C [\mc O_3] = \C [M_\sigma / M^1 ] ,
\end{equation}
where $M_\sigma$ is stabilizer of $E_1$ in $M$. In \cite[Lemma 10.3]{SolEnd} the multiplicity one of
$\sigma_1$ in $\sigma$ implies that the operator $\rho_{\sigma,w} : E \to E$ automatically stabilizes
$E_1$. Therefore we may choose as the element $m_w \in M$ from \cite[Lemma 10.3.a]{SolEnd} just the
identity element. We do that that for all $w$ in the group
\[
W (M,\mc O) = W(\Sigma_{\mc O,\mu}) \rtimes R(\mc O)
\]
from \cite{SolEnd}, which will play the role of $W \rtimes \Gamma$. With that simplification, the
2-cocycle $\natural_J : W(M,\mc O)^2 \to \C^\times \times M_\sigma / M^1$ takes values in $\C^\times$.
Then \cite[Theorem 10.9]{SolEnd} gives:
\begin{itemize}
\item an affine Hecke algebra $\mc H = \mc H (\mc O,G)$, with lattice $M_\sigma / M^1 = X^* (\mc O_3)$
and a reduced root system $\Sigma_{\mc O,\mu}$, 
\item parameters $q_\alpha = q_F^{(\lambda (\alpha) + \lambda^* (\alpha))/2}$ and 
$q_{\alpha*} = q_F^{(\lambda (\alpha) - \lambda^* (\alpha))/2}$ with\\
$1 \neq q_\alpha \geq q_{\alpha*} \geq 1$ for all $\alpha \in \Sigma_{\mc O,\mu}$,
\item elements $T'_r$ for $r \in R(\mc O)$, such that as vector spaces
\[
\End_G (\Pi_{\mf s}) = \bigoplus\nolimits_{r \in R(\mc O)} \mc H \, T'_r .
\]
\end{itemize}
From \cite[(10.20) and Lemma 10.4.a]{SolEnd} we see that these $T'_r$ multiply as in the twisted group
algebra $\C [R(\mc O), \natural_J]$. Conjugation by $T'_r$ is an automorphism of $\mc H (\mc O,G)$, which
by \cite[Theorem 10.6.a]{SolEnd} has the desired effect on $\mc A \cong \C [\mc O_3]$. For a simple root
$\alpha$, \cite[(10.24)]{SolEnd} shows that 
\[
T_r^{'-1} T'_{s_\alpha} T_r \in \C 1 + \C T'_{r^{-1} s_\alpha r} .
\]
From that and the quadratic relations that $T_{s_\alpha}$ and $T'_{r^{-1} s_\alpha r} = 
T'_{s_{r^{-1} \alpha}}$ satisfy, we deduce that $T_r^{'-1} T'_{s_\alpha} T_r$ must equal 
$T'_{r^{-1} s_\alpha r}$. Altogether this shows that $\End_G (\Pi_{\mf s})$ is the twisted affine Hecke 
algebra $\mc H (\mc O,G) \rtimes \C [R(\mc O),\natural_J]$. There is an isomorphism
\[
(\mc H (\mc O,G) \rtimes \C [R(\mc O),\natural_J])^{op} \to 
\mc H (\mc O,G) \rtimes \C [R(\mc O),\natural_J^{-1}]
\]
which is the identity on $\mc A$ and sends each $T'_w$ with $w \in W(M,\mc O)$ to $T_w^{'-1}$. Thus
\begin{equation}\label{eq:8.17}
\End_G (\Pi_{\mf s})^{op} \cong \mc H (\mc O,G) \rtimes \C [R(\mc O),\natural_J^{-1}] .
\end{equation}
By \eqref{eq:6.3}, \eqref{eq:6.5} and \eqref{eq:8.16} 
\begin{equation}\label{eq:8.18}
\Hom_M ( \Pi_{\mf s_M}, \ind_U^G (\xi) ) \cong \End_M (\ind^M_{M^1} (\sigma_1)) \cong
\C [M_\sigma / M^1] .
\end{equation}
That brings us almost to the setting of \cite[\S 2]{SolQS}, with \eqref{eq:8.17} and \eqref{eq:8.18}
the arguments from there work. In particular the Whittaker datum $(U,\xi)$ can be used to normalize
the operators $T'_w$ with $w \in W(M,\mc O)$, and \cite[Theorem 2.7]{SolQS} provides canonical
algebra isomorphisms
\[
\End_G (\Pi_{\mf s}) \cong \mc H (\mc O,G) \rtimes R(\mc O) \cong 
(\mc H (\mc O,G) \rtimes R(\mc O))^{op} \cong \End_G (\Pi_{\mf s})^{op} .
\]
That also finishes the verification of Condition \ref{cond:5.1}'. 
Condition \ref{cond:8.1} was checked in \cite[\S 5]{SolComp}.
\end{proof}

We specialize to the cases where $G$ is quasi-split. It turns out that the $q$-parameters from
Theorem \ref{thm:8.8} have an interesting property, which means that $\mc H$ is close to 
an affine Hecke algebra with equal parameters.

We may assume that $\sigma$ corresponds to the basepoint of $\mc O_3$ in the proof of Theorem
\ref{thm:8.8}, so that all $\alpha \in \Sigma_{\mc O,\mu}$ take the value 1 at $\sigma$.
Let $\sigma' = \sigma \otimes \chi$ be a twist of $\sigma$ by a unitary unramified character of
$\chi$ of $M$. Via $M_\sigma \subset M$ we can consider $\chi$ as a character of the lattice
$M_\sigma / M^1$ involved in $\mc H$. We define a set of roots (in fact a root system)
$\Sigma_{\sigma'} \subset \Sigma_{\mc O,\mu}$ and a parameter function $k^{\sigma'}$ by
\begin{itemize} 
\itemindent = -5mm
\item if $s_\alpha (\sigma') = \sigma'$ and $\chi (\alpha) = 1$, then $\alpha \in \Sigma_{\sigma'}$
and $k^{\sigma'}_\alpha = \log (q_\alpha) / \log (q_F)$,
\item if $s_\alpha (\sigma') = \sigma'$, $\chi (\alpha) = -1$ and $q_{\alpha*} \neq 1$, then 
$\alpha \in \Sigma_{\sigma'}$ and $k^{\sigma'}_\alpha = \log (q_{\alpha*}) / \log (q_F)$,
\item $\alpha \notin \Sigma_{\sigma'}$ for other $\alpha \in \Sigma_{\mc O,\mu}$.
\end{itemize}
With \cite[Lemma 3.15]{Lus-Gr} is not difficult to see that
\[
\Sigma_{\sigma'}^e = \{ \alpha \in \Sigma_{\mc O,\mu} : s_\alpha (\sigma') = \sigma' \}
\]
is a root system and that $\chi (\alpha) \in \{\pm 1\}$ for every $\alpha \in \Sigma_{\sigma'}^e$.
By the $W (\Sigma_{\mc O,\mu})$-invariance of $\lambda$ and $\lambda^*$, the function $k^{\sigma'}$ 
is $W(\Sigma_{\sigma'}^e)$-invariant. The set $\Sigma_{\sigma'}$ is obtained from 
$\Sigma_{\sigma'}^e$ by omitting the $W(\Sigma_{\sigma'}^e)$-stable collection of roots with 
$\chi (\alpha) = -1$ and $q_{\alpha*} = 1$. All such roots are short in a type $B$ irreducible 
component of $\Sigma_{\sigma'}^e$. Thus, for each irreducible component $R^e$ of 
$\Sigma_{\sigma'}^e$, the part in $\Sigma_{\sigma'}$ is either $R^e$ or the set of long roots 
in $R^e$. This shows that $\Sigma_{\sigma'}$ is really a root system. 

By $W(\Sigma_{\sigma'})$-invariance, the function $k^{\sigma'}$ takes the same value on all roots
of a fixed length in one irreducible component.

\begin{prop}\label{prop:8.9}
Let $G$ be quasi-split and recall the notations from Theorem \ref{thm:8.8} and above.
Let $R$ be an irreducible component of $\Sigma_{\sigma'}$, let $\alpha \in R$ be short
and let $\beta \in R$ be long. Then $k^{\sigma'}_\alpha / k^{\sigma'}_\beta$ equals either 1 
or the square of the ratio of the lengths of the coroots $\alpha^\vee$ and $\beta^\vee$ 
(so equals 1, 2 or 3).
\end{prop}
\begin{proof}
We recall from \cite[(3.7)]{SolEnd} that the parameters $q_\alpha$ and $q_{\alpha*}$ in the proof 
of Theorem \ref{thm:8.8} come from poles of Harish-Chandra's function $\mu^\alpha$. In the
notation from \cite{SolEnd}, $\mu^\alpha$ has factors 
\begin{equation}\label{eq:8.21}
\frac{(1 - X_\alpha)}{(1 - q_\alpha^{-1} X_\alpha)} 
\frac{(1 + X_\alpha)}{(1 + q_{\alpha*}^{-1} X_\alpha)} ,
\end{equation} 
where $X_\alpha$ corresponds to evaluation at a certain element $h_\alpha^\vee \in M / M^1$.
In \cite{HeOp} one specializes to twists of $\sigma'$ by unramified characters with
values in $\R_{>0}$, which means the only the left half or the right half in
\eqref{eq:8.21} remains interesting, the other half is put in a holomorphic function
and then ignored. Which of the two halves to chose agrees with how we selected $q_\alpha$ or 
$q_{\alpha*}$ for $k^{\sigma'}$. 
Thus, in the notation of \cite[\S 3]{HeOp}, \eqref{eq:8.21} becomes a factor
\begin{equation}\label{eq:8.23}
\big( 1 - q_F^{\langle \nu, \alpha^\vee \rangle} \big) \big/ 
\big( 1 - q_F^{-1 / \epsilon_{\bar \alpha} + \langle \nu, \alpha^\vee \rangle} \big) .
\end{equation}
Hence $q_\alpha$ or $q_{\alpha*}$ from \cite{SolEnd} equals $q_F^{1 / \epsilon_{\bar \alpha}}$ 
from \cite{HeOp}, and $k^{\sigma'}_\alpha = 1 / \epsilon_{\bar \alpha}$. Now we need to prove
that $\epsilon_{\bar \alpha} / \epsilon_{\bar \beta}$ equals 1 or the square of the ratio
of the lengths of $\alpha$ and $\beta$. That is precisely the condition needed in 
\cite[Theorem 4.1]{HeOp}. It was shown to hold for all generic Bernstein blocks of quasi-split 
reductive $p$-adic groups in \cite[\S 5--6]{HeOp}.
\end{proof}

Proposition \ref{prop:8.9} enables one to reduce the representation theory of $\mc H \rtimes \Gamma$
(as in Theorem \ref{thm:8.8}) to extended graded Hecke algebras with equal parameters, via 
\cite[\S 8--9]{Lus-Gr} or \cite[\S 2.1]{SolAHA}.


\begin{thebibliography}{99}


\bibitem[AMS]{AMS4} A.-M. Aubert, A. Moussaoui, M. Solleveld,
``Affine Hecke algebras for classical $p$-adic groups",
arXiv:2211.08196, 2022

\bibitem[BaCi]{BaCi} D. Barbasch, D. Ciubotaru,
``Hermitian forms for affine Hecke algebras",
arXiv:1312.3316



\bibitem[BaMo]{BaMo3} D. Barbasch, A. Moy,
``Unitary spherical spectrum for $p$-adic classical groups",
Acta Appl. Math. {\bf 44} (1996), 3--37

\bibitem[Bor]{Bor1} A. Borel,
``Admissible representations of a semi-simple group over a local field
with vectors fixed under an Iwahori subgroup",
Inv. Math. {\bf 35} (1976), 233--259

\bibitem[BuHe]{BuHe} C.J. Bushnell, G. Henniart,
``Generalized Whittaker models and the Bernstein center",
Amer. J. Math. {\bf 125.3} (2003), 513--547

\bibitem[BuKu]{BuKu} C.J. Bushnell, P.C. Kutzko,
``Smooth representations of reductive $p$-adic groups: structure theory via types",
Proc. London Math. Soc. {\bf 77.3} (1998), 582--634

\bibitem[Car]{Car} R.W. Carter,
\emph{Finite groups of Lie type. Conjugacy classes and complex characters},
Pure and Applied Mathematics, John Wiley \& Sons, New York NJ, 1985

\bibitem[CaSh]{CaSh} W. Casselman, F. Shahidi,
``On irreducibility of standard modules for generic representations",
Ann. Scient. \'Ec. Norm. Sup (4) {\bf 31} (1998), 561--589

\bibitem[ChSa]{ChSa} K.Y. Chan, G. Savin,
``Iwahori component of the Gelfand--Graev representation",
Math. Z. {\bf 288} (2018), 125--133


\bibitem[Dat]{Dat} J.-F. Dat,	
``Types et inductions pour les repr\'esentations modulaires des groupes $p$-adiques",
Ann. Sci. \'Ec. Norm. Sup. {\bf 32.1} (1999), 1--38

\bibitem[DeOp]{DeOp1} P. Delorme, E.M. Opdam,
``The Schwartz algebra of an affine Hecke algebra",
J. reine angew. Math. {\bf 625} (2008), 59--114


\bibitem[Dij]{Dij} S. Dijols,
``The generalized injectivity conjecture",
Bull. Soc. Math. France {\bf 150.2} (2022), 251-–345

\bibitem[Eve]{Eve} S. Evens,
``The Langlands classification for graded Hecke algebras",
Proc. Amer. Math. Soc. {\bf 124.4} (1996), 1285--1290

\bibitem[Hei]{Hei} V. Heiermann,
``Op\'erateurs d'entrelacement et alg\`ebres de Hecke avec param\`etres d'un groupe r\'eductif 
p-adique - le cas des groupes classiques",
Selecta Math. {\bf 17.3} (2011), 713--756

\bibitem[HeMu]{HeMu} V. Heiermann, G. Mui\'c,
``On the standard modules conjecture",
Math. Z. {\bf 255.4} (2007), 847-–853 	

\bibitem[HeOp]{HeOp} V. Heiermann, E. Opdam,
``On the tempered L-function conjecture",
Amer. J. Math. {\bf 135.3} (2013), 777--799 

\bibitem[Hum]{Hum} J.E. Humphreys,
\emph{Reflection groups and Coxeter groups},
Cambridge Studies in Advanced Mathematics {\bf 29},
Cambridge University Press, Cambridge, 1990

\bibitem[IwMa]{IwMa} N. Iwahori, H. Matsumoto,
``On some Bruhat decomposition and the structure
of the Hecke rings of the $p$-adic Chevalley groups",
Inst. Hautes \'Etudes Sci. Publ. Math {\bf 25} (1965), 5--48
                 
\bibitem[KaLu]{KaLu} D. Kazhdan, G. Lusztig,
``Proof of the Deligne--Langlands conjecture for Hecke algebras",
Invent. Math. {\bf 87} (1987), 153--215

\bibitem[KrRa]{KrRa} C. Kriloff, A. Ram,
``Representations of graded Hecke algebras",
Represent. Theory {\bf 6} (2002), 31--69

\bibitem[Lan]{Lan} R.P. Langlands,	
``On the classification of irreducible representations of real algebraic groups",
pp. 101--170 in: \emph{Representation theory and harmonic analysis on semisimple Lie groups}, 
Math. Surveys Monogr. {\bf 31}, American Mathematical Society, Providence RI, 1989


\bibitem[Lus]{Lus-Gr} G. Lusztig,
``Affine Hecke algebras and their graded version",
J. Amer. Math. Soc {\bf 2.3} (1989), 599--635

\bibitem[MiPa]{MiPa} M. Mishra, B. Pattanayak,
``Principal series component of Gelfand--Graev representation",
Proc. Amer. Math. Soc. {\bf 149.11} (2021), 4955--4962	


\bibitem[Opd1]{Opd1} E.M. Opdam,
``A generating function for the trace of the Iwahori-Hecke algebra",
Progr. Math. {\bf 210} (2003), 301--323

\bibitem[Opd2]{Opd-Sp} E.M. Opdam,
``On the spectral decomposition of affine Hecke algebras",
J. Inst. Math. Jussieu {\bf 3.4} (2004), 531--648





\bibitem[RaRa]{RaRa} A. Ram, J. Rammage,
``Affine Hecke algebras, cyclotomic Hecke algebras and Clifford theory"
pp. 428--466 in: \emph{A tribute to C.S. Seshadri (Chennai 2002)}, 
Trends in Mathematics, Birkh\"auser, 2003

\bibitem[Ren]{Ren} D. Renard,
\emph{Repr\'esentations des groupes r\'eductifs p-adiques},
Cours sp\'ecialis\'es {\bf 17}, Soci\'et\'e Math\'ematique de France, 2010

\bibitem[Roc]{Roc} A. Roche,
``The Bernstein decomposition and the Bernstein centre",
pp. 3--52 in: \emph{Ottawa lectures on admissible representations of reductive $p$-adic groups},
Fields Inst. Monogr. {\bf 26}, Amer. Math. Soc., Providence, RI, 2009

\bibitem[Rod]{Rod} F. Rodier,
``Whittaker models for admissible representations of reductive p-adic split groups", 
pp. 425--430 in: \emph{Harmonic analysis on homogeneous spaces}, 
Proc. Sympos. Pure Math. AMS {\bf 26} (1973)


\bibitem[Shal]{Shal} J.A. Shalika, 
``The multiplicity one theorem for $GL_n$",
Ann. of Math. (2) {\bf 100} (1974), 171--193 

\bibitem[Shah]{Sha1} F. Shahidi,	
``On certain L-functions",
Amer. J. Math. {\bf 103.2} (1981), 297--355

\bibitem[Sol1]{SolGHA} M. Solleveld,
``Parabolically induced representations of graded Hecke algebras",
Algebras and Representation Theory {\bf 15.2} (2012), 233--271

\bibitem[Sol2]{SolAHA} M. Solleveld, ``On the classification of irreducible 
representations of affine Hecke algebras with unequal parameters",
Representation Theory {\bf 16} (2012), 1--87



\bibitem[Sol3]{SolComp} M. Solleveld,
``On completions of Hecke algebras",
pp. 207--262 in: \emph{Representations of Reductive p-adic Groups, 
A.-M. Aubert, M. Mishra, A. Roche, S. Spallone (eds.)},
Progress in Mathematics {\bf 328}, Birkh\"auser, 2019

\bibitem[Sol4]{SolFunct} M. Solleveld,
``Langlands parameters, functoriality and Hecke algebras",
Pacific J. Math. {\bf 304.1} (2020), 209--302

\bibitem[Sol5]{SolHecke} M. Solleveld,
``Affine Hecke algebras and their representations",
Indagationes Mathematica {\bf 32.5} (2021), 1005--1082

\bibitem[Sol6]{SolEnd} M. Solleveld,
``Endomorphism algebras and Hecke algebras for reductive $p$-adic groups",
J. Algebra {\bf 606} (2022), 371--470

\bibitem[Sol7]{SolQS} M. Solleveld,
``On principal series representations of quasi-split reductive $p$-adic groups",
arXiv:2304.06418, 2023

\end{thebibliography}
\end{document}